\documentclass[11pt,letterpaper]{amsart}
\usepackage{enumitem}
\usepackage{moreenum}
\usepackage{mathtools}
\usepackage{csquotes}
\usepackage{hyperref}

\numberwithin{equation}{section}
\newtheorem{theorem}{Theorem}[section]
\newtheorem*{theorem*}{Main Theorem}
\newtheorem{lemma}[theorem]{Lemma}

\newtheorem{proposition}[theorem]{Proposition}
\theoremstyle{definition}

\newtheorem{remark}[theorem]{Remark}

\newcommand{\per}{\mathrm{per}}
\newcommand{\e}{\mathrm{e}}
\newcommand{\R}{\mathbb R}

\newcommand{\N}{\mathbb N}

\newcommand{\Om}{\mathcal O}
\newcommand{\SL}{\mathcal S}
\newcommand{\LL}{\mathcal L}
\newcommand{\T}{\mathcal T}
\newcommand{\M}{\mathcal M}
\newcommand{\K}{\mathcal K}
\newcommand{\Pro}{\mathcal P}
\newcommand{\A}{\mathcal A}

\newcommand{\Ch}{\mathcal C_h^L}
\newcommand{\Cc}{\mathcal C}
\DeclareMathOperator{\im}{im}
\DeclareMathOperator{\sgn}{sgn}
\DeclareMathOperator{\Real}{Re}
\DeclareMathOperator{\Imag}{Im}

\begin{document}
	
\title[Large-amplitude water waves with general vorticity]{Large-amplitude steady gravity water waves with general vorticity and critical layers}
\author{Erik Wahlén}
\address{Centre for Mathematical Sciences, Lund University, 221 00 Lund, Sweden}
\email{erik.wahlen@math.lu.se}
\author{Jörg Weber}
\address{Centre for Mathematical Sciences, Lund University, 221 00 Lund, Sweden}
\email{jorg.weber@math.lu.se}
\thanks{This project has received funding from the European Research Council (ERC) under the European Union’s Horizon 2020 research and innovation program (grant agreement no 678698) and the Swedish Research Council (grant no 2020-00440).}
\subjclass[2010]{76B15 (primary), 35B32, 35C07, 35Q31, 35R35, 76B47}

\begin{abstract}
	We consider two-dimensional steady periodic gravity waves on water of finite depth with a prescribed but arbitrary vorticity distribution. The water surface is allowed to be overhanging and no assumptions regarding the absence of stagnation points and critical layers are made.
	Using conformal mappings and a new reformulation of Bernoulli's equation, we uncover an equivalent formulation as ``identity plus compact'', which is amenable to Rabinowitz's global bifurcation theorem. This allows us to construct a global connected set of solutions, bifurcating from laminar flows with a flat surface. Moreover, a nodal analysis is carried out for these solutions under a certain spectral assumption involving the vorticity function. Lastly, downstream waves are investigated in more detail.
\end{abstract}

\maketitle

\section{Introduction}
The water wave problem concerns the motion of an inviscid, incompressible, and homogeneous fluid with a free surface.
For wavelengths beyond a few centimeters, gravity is the dominating restoring force and it is therefore a good approximation to consider pure gravity waves and ignore the effects of surface tension when focusing on large-scale features. 
Using global bifurcation theory, we consider the classical problem of the existence of large-amplitude two-dimensional periodic steady waves. In contrast to the many previous studies on the subject, the present paper for the first time simultaneously allows for a general vorticity distribution, stagnation points and arbitrarily many critical layers, and overhanging profiles of the free surface. We first provide a short overview of the extensive literature on this or similar problems.

In the previous century, mostly irrotational flows were considered, which enjoy the advantage of being thoroughly treatable by tools of complex analysis, because the water velocity can be written as the gradient of a harmonic potential. We refer to \cite{Groves04, HHSTWWW22, Toland96} for surveys on irrotational water waves. However, in many situations it is important to take vorticity into account, for example, in the presence of underlying non-uniform currents. When vorticity is included, it was unclear for a long time how to leave the regime of small perturbations of configurations with a flat surface. The breakthrough in this direction is due to Constantin and Strauss \cite{ConsStr04}, who utilized the semi-hodograph transformation of Dubreil-Jacotin \cite{Dubreil34}. However, such a transformation a priori precludes the presence of overhanging profiles, stagnation points, and critical layers. After that, similar problems, for example also taking into account capillary effects and/or stratification, were considered in the same spirit, see \cite{Wahlen06b,Wahlen06a} and \cite{EKLM20,Walsh09,Walsh14a,Walsh14b}, respectively.

In order to overcome the inherited downsides of the above-mentioned semi-hodograph transformation and to allow for stagnation points and critical layers, many papers use a ``naive'' flattening transform, where on each vertical ray the vertical coordinate is scaled to a constant; see \cite{AasVar18,EhrnEschWahl11,EhrnVill08,EhrnWahl15,HenMat13,KozKuz14,Varholm20,Wahlen09}. Consequently, a drawback of this naive flattening is that it needs the surface profile to be a graph and can thus not allow for overhanging waves.

However, numerical studies  \cite{DyaHur19b, DyaHur19a, RRP17, SimmenSaffman85, Teles-daSilvaPeregrine88, Vanden-Broeck96} and recent rigorous results \cite{HurWheeler20, HurWheeler21} show the existence of waves with overhanging surface profiles in the case of constant vorticity. Therefore, such profiles should also be expected when the vorticity distribution is more general than just being constant. Then, when one wants to allow for stagnation points and critical layers, and also for overhanging waves, the typical strategy is to use a conformal change of variables. This is classical in the irrotational context and can be used to reformulate the problem either as a singular integral equation for the tangent angle of the free surface \cite{Groves04, HHSTWWW22, Toland96} or as a nonlinear pseudodifferential equation for the surface elevation in the new variables introduced by Babenko \cite{Babenko87b} (see also \cite{Groves04, HHSTWWW22} and references therein). The latter approach was extended to constant vorticity by Constantin and V\u{a}rv\u{a}ruc\u{a} in \cite{ConsVarva11} and then utilized, together with Strauss, in \cite{ConsStrVarv16} to obtain pure gravity water waves of large amplitude. A key idea here was to reformulate the problem elegantly as a Riemann--Hilbert problem, which, however, relies heavily on the vorticity being constant. In the same spirit, certain stratified waves were constructed in \cite{Haziot21}, and local bifurcation for capillary-gravity waves with constant vorticity was investigated in \cite{Martin13}. We also like to mention \cite{HaziotWheeler21} for results on global bifurcation for solitary waves with constant vorticity. There, the authors write the problem as elliptic equations for the (imaginary part of the) conformal map and some quantity related to the stream function. Interestingly, those elliptic equations are purely local, while typically nonlocal terms appear in the previous approaches building upon \cite{ConsStrVarv16}. Our formulation, from this point of view, can be seen  as a hybrid since it consists of local as well as nonlocal ingredients. The authors of \cite{HaziotWheeler21} also mention that their approach can be generalized to arbitrary vorticity distributions, but do not carry this out.

In this paper, we will also use a conformal change of variables, thus allowing for stagnation points, critical layers, and overhanging profiles. The novelty here is however that we impose no assumption on the vorticity distribution, except for a very mild regularity assumption to be specified later, while, as mentioned above, previous papers in this direction only considered the case of constant vorticity. A very delicate issue is to reformulate the problem such that it becomes amenable to a certain global bifurcation theorem in order to construct steady water waves of large amplitude. Let us quickly review the different strategies in this regard that are typically pursued in the literature. First, degree theoretic methods using the Healey--Simpson degree and Kielhöfer degree, respectively, were used in \cite{ConsStr04} and \cite{Walsh14b}, respectively, albeit relying on the semi-hodograph transform in order to prove admissibility of the nonlinear operator. Second, analytic global bifurcation, going back to Dancer \cite{Dancer73} and Buffoni and Toland \cite{BuffTol03}, was, for example, used in \cite{ConsStrVarv16,Varholm20} to construct a global continuum of solutions, and has the advantage of producing a curve of solutions admitting locally a real-analytic reparametrization. However, to this end the occurring nonlinear operators have to be analytic; this, in turn, requires that the vorticity function be real-analytic unless the semi-hodograph transform has been applied in the first place. Third, also the global bifurcation theorem of Rabinowitz \cite{Rabinowitz71,Kielhoefer} has been utilized, where the operator equation needs to have the form \enquote{identity plus compact}, so that the problem has to be reformulated suitably. In this spirit, the strategy of \cite{HenMat14,Matioc14} was to invert the curvature operator in order to \enquote{gain a derivative}, crucially relying on the presence of surface tension; however, restricted to the case where the surface profile is a graph, and where the vorticity function satisfies a monotonicity assumption, thus giving rise to a reformulation via reduction to the boundary. We also mention \cite{AmbroseStraussWright16} where a similar strategy is used for vortex sheets in the irrotational context, allowing for overhanging waves and using a reformulation as a nonlinear singular integral equation in terms of the tangent angle. Then, the paper \cite{WahlenWeber21} first allowed for arbitrary vorticity, stagnation points, critical layers, and overhanging profiles, also crucially making use of nonzero surface tension to derive an \enquote{identity plus compact} operator equation via a novel reformulation of Bernoulli's equation. The main goal of the present paper is to extend this strategy and the results to the pure gravity case. To this end and what may seem surprising, we uncover yet another reformulation as \enquote{identity plus compact}, however, without making use of (the absent) surface tension. Interestingly, this formulation appears to be new even in the irrotational case.

Our paper is organized as follows: In Section \ref{sec:Preliminaries}, we state the problem, some preliminary tools, the functional-analytic setting, and our reformulation. We clarify, motivate, and justify this reformulation then in Section \ref{sec:Reformulation}. Having the new operator equation at hand, we proceed in Section \ref{sec:LocalBifurcation} with the investigation of local bifurcation, while deferring the proofs to Appendix \ref{appx:LocalBifurcation}, since the local theory is already well understood and not the main motivation of the present paper. Then, more importantly, we construct in Section \ref{sec:GlobalBifurcation} large amplitude solutions, applying the global bifurcation theorem of Rabinowitz, and investigate which alternatives for the global solution curve can occur; see Theorem \ref{thm:GlobalBifurcation}. 
The nodal properties of our constructed solutions are analyzed in more detail in Section \ref{sec:nodalanalysis}, leading to a strengthening of the results obtained in Theorems \ref{thm:GlobalBifurcation}; this, however, requires at least the spectral assumption \eqref{ass:Pruefergraph}. Finally, we prove in Section \ref{sec:downstream} that so-called downstream waves have in fact a unidirectional flow and cannot overhang. For completeness, we also present an analytic global bifurcation result, Theorem \ref{thm:GlobalBifurcation_Analytic}, with somewhat stronger conclusions under the assumption that the vorticity function is real-analytic. In order not to disrupt the flow of the paper, this is presented in Appendix \ref{appx:GlobalBifurcation_Analytic}. The abstract global bifurcation theorems that we use are recorded in Appendix \ref{appx:GlobalBif}. Finally, Appendix \ref{appx:degeneracy_conformal} contains a more detailed analysis of one of the alternatives in the global bifurcation result, degeneration of the conformal map.

Let us already state here, somewhat informally in words, the main results of this paper. For more precise statements we refer to Theorems \ref{thm:GlobalBifurcation} and \ref{thm:NodalProperties}.
\begin{theorem*}
	For any vorticity function with bounded first derivative, let us assume that small-amplitude waves (that is, small perturbations of configurations with a flat surface) have already been constructed (the theory in the small being, as mentioned above, already well understood). Then such a local family of solutions can be extended to a global family of solutions, along which one of the following alternatives occurs:
	\begin{itemize}
		\item The solution set is unbounded in the sense that the bifurcation parameter (here, the horizontal velocity at the free surface of the underlying laminar flow) becomes unbounded, or a certain Hölder norm of the function parametrizing the free surface becomes unbounded, or the total vorticity measured in $L^p$ becomes unbounded.
		\item The solution set loops back to a configuration with a flat surface.
		\item The solutions approach a wave of greatest height.
		\item The (conformal) change of variables becomes degenerate.
		\item The solution set contains a wave with self-intersecting surface profile.
		\item The solution set contains a wave whose surface intersects the flat bed.
	\end{itemize}
	Under additional assumptions on the vorticity function, much more can be said about structural properties of the waves, eliminating some of these alternatives. For example, if the conditions \eqref{eq:nodal cond} are satisfied, then the height of the surface profile is always monotone from crest to trough, the solution set cannot loop back to a configuration with a flat surface, the surface can never intersect the flat bed, and self-intersection of the surface has to occur exactly above a trough. In particular, this holds in the case of affine linear vorticity functions with non-positive slope, including constant vorticity.
\end{theorem*}

\section{Statement of the problem and preliminaries}\label{sec:Preliminaries}
\subsubsection*{Some notation}
Let us first introduce some notations that will be used throughout this paper. For a given open set $\Omega\subset\R^n$, $n\in\N\coloneqq\{1,2,\ldots\}$, and parameters $k\in\N_0\coloneqq\N\cup\{0\}$, $0<\alpha\le 1$, we denote by $C^{k,\alpha}(\overline\Omega)$ the usual Hölder space (Lipschitz space if $\alpha=1$), that is, the space of functions on $\Omega$ having derivatives up to order $k$ such that all derivatives of order $k$ are Hölder continuous with parameter $\alpha$. These spaces are equipped with the usual norm
\[\|f\|_{C^{k,\alpha}(\overline\Omega)}=\sum_{|\beta|\le k}\sup_\Omega |D^\beta f|+\sum_{|\beta|=k}\sup_{x,y\in\Omega,x\ne y}\frac{|D^\beta f(x)-D^\beta f(y)|}{|x-y|^\alpha}.\]
The index \enquote{loc} in, for example, $C_{\text{loc}}^{k,\alpha}(\overline\Omega)$ indicates that such functions are locally of class $C^{k,\alpha}$, that is, elements of $C^{k,\alpha}(\overline{\Omega'})$ for all open, bounded sets $\Omega'\subset\R^n$ such that $\overline{\Omega'}\subset\Omega$. Furthermore, for $\Omega$ and $k$ as above and $1\le p<\infty$, $W^{k,p}(\Omega)$ denotes the usual Sobolev space (Lebesgue space if $k=0$) equipped with the standard norm
\[\|f\|_{W^{k,p}(\Omega)}=\left(\sum_{|\beta|\le k}\int_\Omega|D^\beta f|^p\right)^{1/p}.\]
We use the usual abbreviations $L^p=W^{0,p}$ and $H^k=W^{k,2}$. Henceforth, $0<\alpha<1$ is arbitrary and fixed. Moreover, we write $\SL$ for the operator evaluating functions $f$, which are defined (at least) for $(x,y)\in\R\times\{0\}$, at $y=0$, that is,
\begin{align}\label{eq:def_S_eval}
	\SL f\coloneqq f(\cdot,0).
\end{align}
Furthermore, we denote (partial) derivatives by lower indices; for example, $f_x$ is the (partial) derivative of $f$ with respect to $x$. Finally, we will denote by $(X,Y)$ coordinates in the physical domain and by $(x,y)$ coordinates in the flattened domain, which is
\begin{align}\label{eq:def_Omegah}
	\Omega_h\coloneqq\R\times(-h,0)
\end{align}
for some $h>0$; they are related to each other via a conformal map, which we will explain in more detail later.

\subsubsection*{The equations}
Our problem reads in the so-called stream formulation as follows (see \cite{ConstantinBook}, for example):
\begin{subequations}\label{eq:OriginalEquations}
	\begin{align}
	\Delta\psi&=-\gamma(\psi)&\text{in }\Omega,\label{eq:OriginalEquations_Poisson}\\
	\frac{|\nabla\psi|^2}{2}+g(Y-h)&=Q&\text{on }S\label{eq:OriginalEquations_Bernoulli},\\
	\psi&=0&\text{on }S,\label{eq:OriginalEquations_KinematicTop}\\
	\psi&=-m&\text{on }Y=0\label{eq:OriginalEquations_KinematicBottom}.
	\end{align}
\end{subequations}
Here, $\psi$ is the stream function, which is assumed to be $L$-periodic in $X$ (where the period $L>0$ is fixed; we abbreviate $\nu\coloneqq 2\pi/L$) and which gives rise to the velocity field $(\psi_Y,-\psi_X)$ in a frame moving at a constant wave speed. Moreover, $\gamma$ is the vorticity function satisfying (this is the only assumption we shall always impose on $\gamma$)
\[\gamma\in C_{\text{loc}}^{1,1}(\R),\quad\|\gamma'\|_\infty<\infty,\]
$g>0$ is the gravitational constant, $\Omega$ is the a priori unknown $L$-periodic fluid domain bounded by the free surface $S$ and flat bed $\R\times\{0\}$, $h>0$ is the conformal mean depth of $\Omega$, and $Q,m\in\R$ are constants; in physical language, the Bernoulli constant $Q$ is related to the hydraulic head $E$ and the (constant) atmospheric pressure $p_{\mathrm{atm}}$ via $Q=E-p_{\mathrm{atm}}-gh$, and $m$ is the relative mass flux.

\subsubsection*{Parametrization of the fluid domain by a conformal map}
Let us briefly clarify the notion of conformal mean depth and provide a very important preliminary result on conformal mappings. To this end, we introduce the $L$-periodic Hilbert transform on a strip with depth $h$, given by
\begin{align}\label{eq:def_ChL}
	\Ch u(x)=-i\sum_{k\ne 0} \coth(k\nu h) \hat u_k e^{ik\nu x}
\end{align}
for an $L$-periodic function
\[u(x)=\sum_{k\ne 0} \hat u_k e^{ik\nu x}\]
with zero average.
\begin{lemma}\label{lma:ConformalMapping}
	Let $\Omega\subset\R^2$ be an $L$-periodic strip-like domain, that is, a domain contained in the upper half plane such that its boundary consists of the real axis $\R\times\{0\}$ and a simple curve $S=\{(u(s),v(s)):s\in\R\}$ with $u(s+L)=u(s)+L$, $v(s+L)=v(s)$, $s\in\R$. Then:
	\begin{enumerate}[label=(\roman*)]
		\item There exists a unique positive number $h$ such that there exists a conformal mapping $H=U+iV$ from the strip $\Omega_h=\R\times(-h,0)$ to $\Omega$ which admits an extension as a homeomorphism between the closures of these domains, with $\R\times\{0\}$ being mapped onto S and $\R\times\{-h\}$
		being mapped onto $\R\times\{0\}$, and such that $U(x+L,y)=U(x,y)+L$, $V(x+L,y)=V(x,y)$, $(x,y)\in\overline{\Omega_h}$.
		\item The conformal mapping $H$ is unique up to translations in the variable $x$ (in the preimage and the image).
		\item $U$ and $V$ are (up to translations in the variable $x$) uniquely determined by $w=V(\cdot,0)-h$ as follows: $V$ is the unique ($L$-periodic) solution of
		\begin{subequations}\label{eq:BVP_for_V}
		\begin{align}
			\Delta V&=0&\mathrm{in\ }\Omega_h,\\
			V&=w+h&\mathrm{on\ }y=0,\\
			V&=0&\mathrm{on\ }y=-h,
		\end{align}
		\end{subequations}
		and $U$ is the (up to a real constant unique) harmonic conjugate of $-V$. Furthermore, after a suitable horizontal translation, $S$ can be parametrized by
		\[S=\{(x+(\Ch w)(x),w(x)+h):x\in\R\}\]
		and it holds that
		\begin{align}\label{eq:nablaV_on_top}
			\SL\nabla V=(w',1+\Ch w').
		\end{align}
		\item If $S$ is of class $C^{1,\beta}$ for some $\beta>0$, then $U,V\in C^{1,\beta}(\overline{\Omega_h})$ and
		\[|dH/dz|^2=|\nabla V|^2\neq 0\mathrm{\ in\ }\overline{\Omega_h}.\]
	\end{enumerate}
\end{lemma}
\begin{proof}
	See \cite[Theorem 2.2, Appendix A]{ConsVarva11}.\end{proof}

In \eqref{eq:OriginalEquations}, we demand that $\Omega$ is an $L$-periodic strip-like domain with boundary of class $C^{1,\alpha}$ and with conformal mean depth $h$ so that it is, due to Lemma \ref{lma:ConformalMapping}, the conformal image of $\Omega_h$ and the free surface $S$ is determined by some $L$-periodic $w$ of class $C^{1,\alpha}$ with zero average over one period satisfying
\[(1+\Ch w')^2+w'^2\ne 0\text{ on }\R\]
via
\begin{align}\label{eq:Sw}
	S=\{(x+(\Ch w)(x),w(x)+h):x\in\R\}.
\end{align}
Throughout this paper, $h>0$ is fixed. Therefore, all solutions obtained by some bifurcation theorem later will have the same conformal mean depth $h$.

In order to make sense of a conformal map as described above, $\Omega$ has to be a strip-like domain, and hence we have to exclude self-intersection of the wave and intersection of the surface with the bed a priori. Correspondingly, in terms of $w$ we assume
\begin{subequations}\label{eq:additional_requirements}
\begin{gather}
	x\mapsto(x+(\Ch w)(x),w(x)+h)\text{ is injective on }\R,\\
	w>-h\text{ on }\R.\label{eq:intersection_with_bed}
\end{gather}	
\end{subequations}

\subsubsection*{Trivial solutions}
Before being able to apply some bifurcation theorem, one first has to identify a curve of trivial solutions. In our context, they correspond to a configuration with a flat surface and where the stream function does not depend on $x$. In this spirit, we have to solve the boundary value problem
\[\psi_{yy}=-\gamma(\psi)\text{ on }[-h,0],\quad\psi(0)=0,\quad\psi(-h)=-m\]
for $\psi=\psi(y)$; then $\psi(\cdot-h)$ solves \eqref{eq:OriginalEquations} with $w\equiv0$ (that is, $\Omega=\R\times(0,h)$). However, in general, a solution to this boundary value problem neither has to exist nor be unique. Thus, we introduce a new parameter $\lambda\in\R$, which will later serve as the bifurcation parameter, and prescribe $\psi_y(0)=\lambda$. Notice that $\lambda$ can be interpreted as the (horizontal) velocity at the surface (or, rather, the relative velocity as everything is usually written down in a moving frame). The trivial solution corresponding to $\lambda$ is then $\psi=\psi^\lambda$, which is defined to be the unique solution of the Cauchy problem
\[\psi_{yy}^\lambda=-\gamma(\psi^\lambda)\text{ on }[-h,0],\quad\psi^\lambda(0)=0,\quad\psi_y^\lambda(0)=\lambda.\]
Indeed, there exists a unique solution due to the global Lipschitz continuity of $\gamma$. We therefore view $m$ not as a parameter, but as a function $m=m(\lambda)$ of $\lambda$ defined by
\begin{align}\label{eq:def_m}
	m(\lambda)\coloneqq -\psi^\lambda(-h).
\end{align}
At this point we mention that the assumption that $\gamma$ be globally Lipschitz continuous is only needed to ensure that all trivial solutions exist on $[-h,0]$; we could therefore relax this assumption and demand that all trivial solutions under consideration exist on $[-h,0]$, which is, however, not very explicit.

The Bernoulli constant corresponding to $\psi^\lambda$ is $Q=\lambda^2/2$. Thus, in general, we introduce
\begin{align}\label{eq:def_q}
	q=Q-\lambda^2/2
\end{align}
as the deviation of the Bernoulli constant from its value at a trivial solution.

\subsubsection*{The functional-analytic setting}
We now introduce the functional-analytic setting and denote
\begin{subequations}\label{eq:def_X}
\begin{align}
	\tilde X&\coloneqq C_{0,\per,\e}^{1,\alpha}(\R)\\
	&\phantom{\coloneqq\;}\times\left\{\phi\in C_{\per,\e}^{0,\alpha}(\overline{\Omega_h})\cap H_\per^1(\Omega_h):\phi=0\text{ on }y=0\text{ and on }y=-h\right\},\nonumber\\
	X&\coloneqq\R\times\tilde X,
\end{align}
\end{subequations}
with $\Omega_h$ defined in \eqref{eq:def_Omegah}. Here, the indices \enquote{$\per$}, \enquote{$\e$}, and \enquote{$0$} denote $L$-periodicity, evenness (in $x$ with respect to $x=0$), and zero average over one period. Let us briefly mention that our restriction to even functions and thus to symmetric waves is not too limiting since typically such waves are observed in nature and it can in fact be shown rigorously in some situations that solutions to \eqref{eq:OriginalEquations} necessarily have such a symmetry property \cite{ConstEhrnWahlen07,ConstantinEscher04}. Moreover, $X$ is equipped with the norm
\begin{align*}
	\|(q,w,\phi)\|_X&\coloneqq|q|+\|(w,\phi)\|_{\tilde X}\\
	&\coloneqq|q|+\|w\|_{C_\per^{1,\alpha}(\R)}+\|\phi\|_{C_\per^{0,\alpha}(\overline{\Omega_h})\cap H_\per^1(\Omega_h)};\\
	\|w\|_{C_\per^{1,\alpha}(\R)}&\coloneqq\|w\|_{C^{1,\alpha}([0,L])},\\
	\|\phi\|_{C_\per^{0,\alpha}(\overline{\Omega_h})\cap H_\per^1(\Omega_h)}&\coloneqq\|\phi\|_{C_\per^{0,\alpha}(\overline{\Omega_h})}+\|\phi\|_{H_\per^1(\Omega_h)}\\
	&\coloneqq\|\phi\|_{C^{0,\alpha}(\overline{\Omega_h^*})}+\|\phi\|_{H^1(\Omega_h^*)}.
\end{align*}
Here and in the following, $\Omega^*$ denotes one copy of an $L$-periodic domain $\Omega\subset\R^2$, that is,
\[\Omega^*\coloneqq\{(x,y)\in\Omega:x\in(0,L)\}.\]
In general, we will write
\[\|\cdot\|_{C_\per^{k,\alpha}(\R)}\coloneqq\|\cdot\|_{C^{k,\alpha}([0,L])},\;\|\cdot\|_{C_\per^k(\R)}\coloneqq\sum_{j=0}^k\|\cdot^{(j)}\|_\infty,\;\|\cdot\|_{C_\per^{k,\alpha}(\overline\Omega)}\coloneqq\|\cdot\|_{C^{k,\alpha}(\overline{\Omega^*})}\]
for $k\in\N_0$ and $\Omega$ as above, where $C^k$ is the space of $k$-times continuously differentiable functions as usual.
\subsubsection*{The nonlinear operator}
Given $w\in C_{0,\per}^{1,\alpha}(\R)$, we denote by $V$ the unique solution of \eqref{eq:BVP_for_V} in $C_\per^{1,\alpha}(\overline{\Omega_h})$ and thus view it as a function of $w$. Provided \eqref{eq:additional_requirements}, $V$ gives rise to a conformal mapping $H=U+iV$ from $\Omega_h$ to $\Omega=\Omega_w$, the surface of which being determined by \eqref{eq:Sw}, in view of Lemma \ref{lma:ConformalMapping} and the Darboux--Picard theorem (see also \cite{ConsVarva11}). Notice that $V$ is explicitly given by
\begin{align}\label{eq:V_explicit}
	V(x,y)=y+h+\sum_{k\neq 0}\hat w_k e^{ik\nu x} \frac{\sinh(k\nu(y+h))}{\sinh(k\nu h)},
\end{align}
where the $\hat w_k$ are the Fourier coefficients of $w$. We sometimes denote $V=V[w+h]$ (and likewise also for $H$ and $U$) if we want to express the dependency of $V$ on the boundary condition at $y=0$ explicitly. It is clear that $V$ is analytic in $w$.

For $(\lambda,w,\phi)\in  X$, we write $\A=\A(\lambda,w,\phi)$ for the unique solution of
\begin{subequations}\label{eq:def_A}
\begin{align}
	\Delta\A&=-\gamma(\phi+\psi^\lambda)|\nabla V|^2+\gamma(\psi^\lambda)&\text{in }\Omega_h,\\
	\A&=0&\text{on }y=0,\\
	\A&=0&\text{on }y=-h
\end{align}
\end{subequations}
in $C_\per^{2,\alpha}(\overline{\Omega_h})$. Notice that, if \eqref{eq:additional_requirements} is satisfied, this is equivalent to
\begin{align*}
	\Delta\psi&=-\gamma((\phi+\psi^\lambda)\circ H^{-1})&\text{in }\Omega_w,\\
	\psi&=0&\text{on }S_w,\\
	\psi&=-m(\lambda)&\text{on }Y=0,
\end{align*}
where $\psi=(\A+\psi^\lambda)\circ H^{-1}$. Thus, $\A(\lambda,w,\phi)=\phi$ is equivalent to the statement that
\[\psi=(\phi+\psi^\lambda)\circ H^{-1}\]
solves \eqref{eq:OriginalEquations_Poisson}, \eqref{eq:OriginalEquations_KinematicTop}, and \eqref{eq:OriginalEquations_KinematicBottom} with $\Omega=H(\Omega_h)$ and $m=m(\lambda)$, provided \eqref{eq:additional_requirements}. From this point of view, $\phi$ can be interpreted as the deviation from a trivial solution $\psi^\lambda$ in the flattened domain, while the relative mass flux $m(\lambda)$ of $\psi^\lambda$ remains unchanged. Let us remark here that the requirement $\phi\in H_\per^1(\Omega_h)$ in \eqref{eq:def_X} is only needed in the local analysis presented in particular in Appendix \ref{appx:range}, and does not impair any result of the whole paper.

Furthermore, the points $(\lambda,0,0,0)\in\R\times X$ are identified as trivial solutions. Moreover, clearly $V$ and thus $\A(\lambda,w,\phi)$ are even in $x$. We pause here briefly to point out that we have to take the detour via $\A$ and thus consider $(w,\phi)$ and not only $w$ as part of the unknown since a reduction to the boundary, that is, unique solvability of
\begin{align*}
	\Delta\phi&=-\gamma(\phi+\psi^\lambda)|\nabla V|^2+\gamma(\psi^\lambda)&\text{in }\Omega_h,\\
	\phi&=0&\text{on }y=0,\\
	\phi&=0&\text{on }y=-h,
\end{align*}
in terms of $V$ (and thus in terms of $w$) is in general not feasible, as it would require a monotonicity assumption on the nonlinearity $\gamma$, which we do not want to impose.

We will now also state how we rewrite the Bernoulli equation and then the nonlinear operator as a whole; this will seem a bit unclear at first sight, but will be justified below in Lemma \ref{lma:Reformulation}. Let us denote
\begin{align}
	\K(w)&\coloneqq\sqrt{(1+\Ch w')^2+w'^2},\quad w\in C_\per^{1,\alpha}(\R),\label{eq:def_K}\\
	\Om&\coloneqq\Big\{(\lambda,q,w,\phi)\in\R\times X:\K(w)>0\text{ on }\R,\hspace{3.2cm}\nonumber\\
	&\omit\hfill$\displaystyle\SL\partial_y\A(\lambda,w,\phi)+\lambda\ne0\text{ on }\R,\;q+\lambda^2/2-gw>0\text{ on }\R\Big\}.$\nonumber \end{align}
Clearly, $\Om\subset\R\times X$ is open, and
\begin{align}\label{eq:trvial_in_O}
	(\lambda,0,0,0)\in\Om\quad\Leftrightarrow\quad\lambda\ne0.
\end{align}
Moreover, let
\begin{align}\label{eq:def_R}
	R(\lambda,q,w,\phi)\coloneqq \frac{|\SL\partial_y\A(\lambda,w,\phi)+\lambda|}{\sqrt{2q+\lambda^2-2gw}}
\end{align}
for $(\lambda,q,w,\phi)\in\Om$, recalling \eqref{eq:def_S_eval} and \eqref{eq:def_A}. It is evident that $\K(w)$ and $R(\lambda,q,w,\phi)$ are also even in $x$. Let us remark that the conditions in the definition of $\Om$, once such a tuple $(\lambda,q,w,\phi)$ gives rise to a water wave, can be interpreted as \enquote{the conformal map is not degenerate}, \enquote{no surface stagnation in the flattened domain appears}, and \enquote{the wave is not of greatest height}. It is perfectly reasonable that one has to stay away from these scenarios in order to hope for a certain compactness and therefore gain of regularity -- this will be made much more clear in Section \ref{sec:Reformulation}.

We reformulate the original problem \eqref{eq:OriginalEquations} as 
\begin{align}\label{eq:F=0}
	F(\lambda,q,w,\phi)=0
\end{align}
for $(\lambda,q,w,\phi)\in\Om$, with
\begin{align}\label{eq:def_F}
	F\colon\Om\to X,\;F(\lambda,q,w,\phi)=(q,w,\phi)-\M(\lambda,q,w,\phi),
\end{align}
where
\begin{subequations}\label{eq:def_M}
\begin{gather}
	\M(\lambda,q,w,\phi)\coloneqq(\M^1(\lambda,q,w,\phi),\M^2(\lambda,q,w,\phi),\M^3(\lambda,q,w,\phi)),\\
	\M^1(\lambda,q,w,\phi)\coloneqq q+\left\langle R(\lambda,q,w,\phi)\cos\left((\Ch)^{-1}\Pro(\ln R(\lambda,q,w,\phi))\right)\right\rangle-1,\\
	\M^2(\lambda,q,w,\phi)\coloneqq\partial_x^{-1}\left(R(\lambda,q,w,\phi)\sin\left((\Ch)^{-1}\Pro(\ln R(\lambda,q,w,\phi))\right)\right),\\
	\M^3(\lambda,q,w,\phi)\coloneqq\A(\lambda,w,\phi).
\end{gather}
\end{subequations}
Here and in the following, while recalling \eqref{eq:def_ChL},
\[
	\langle f\rangle\coloneqq\frac1L\int_0^Lf(x)\,dx
\]
denotes the average of an $L$-periodic function $f$ over one period, $\Pro$ is the projection onto the space of functions with zero average, that is,
\[
	\Pro f\coloneqq f-\langle f\rangle,
\]
and $\partial_x^{-1}\colon C_{0,\per}^{0,\alpha}(\R)\to C_{0,\per}^{1,\alpha}(\R)$ is the inverse operation to differentiation, that is, the Fourier multiplier with symbol $(ik\nu)^{-1}$:
\[
	(\partial_x^{-1}f)(x)=\Pro\left( \int_0^xf(y)\,dy \right),\quad\text{that is,}\quad\partial_x^{-1}\sum_{k\ne0}\hat f_ke^{ik\nu x}=\sum_{k\ne0}\frac{1}{ik\nu}\hat f_ke^{ik\nu x}.
\]
We point out that (in particular the second component of) $\M$ is well-defined since the function $R(\lambda,q,w,\phi)\sin\left((\Ch)^{-1}\Pro(\ln R(\lambda,q,w,\phi))\right)$ is odd as $(\Ch)^{-1}$ reverses parity. Notice also that $\Ch$ acts as an isomorphism on $C_{0,\per}^{0,\alpha}(\R)$.

For the reader's convenience, let us also state, slightly informally, our reformulation in an easier to read way, which we will however never use in what follows:
\begin{align*}
	\left\langle R\cos\left((\Ch)^{-1}\Pro(\ln R)\right)\right\rangle&=1,\\
	w'&=R\sin\left((\Ch)^{-1}\Pro(\ln R)\right),\\
	\phi&=\A,
\end{align*}
where $\A=\A(\lambda,w,\phi)$ and $R=R(\lambda,q,w,\phi)$ are defined in terms of \eqref{eq:def_A} and \eqref{eq:def_R}, respectively.

\section{On the reformulation}\label{sec:Reformulation}
We now prove two important lemmas which justify and motivate this new reformulation. First, we show that it is indeed equivalent to the original problem, and then demonstrate that $\M$ is compact on certain subsets of $\Om$ as it \enquote{gains a derivative}.
\begin{lemma}\label{lma:Reformulation}
	Under the assumption \eqref{eq:additional_requirements}, a tuple $(\lambda,q,w,\phi)\in\Om$ solves \eqref{eq:F=0} if and only if $w$ is of class $C^{2,\alpha}$, $\Omega_w$, the surface of which being determined by \eqref{eq:Sw}, is of class $C^{1,\alpha}$, and $(w,(\phi+\psi^\lambda)\circ H^{-1})$, where the conformal mapping $H=U+iV\colon\Omega_h\to\Omega_w$ is uniquely (up to translations in $x$) determined by $w$, solves \eqref{eq:OriginalEquations} with $\Omega=\Omega_w$, $Q=q+\lambda^2/2$, and $m=m(\lambda)$ such that $|\nabla((\phi+\psi^\lambda)\circ H^{-1})|>0$ on the surface.
\end{lemma}
\begin{proof}
	As was already observed, $\phi=\A(\lambda,w,\phi)$ is equivalent to that \eqref{eq:OriginalEquations_Poisson}, \eqref{eq:OriginalEquations_KinematicTop}, \eqref{eq:OriginalEquations_KinematicBottom} are solved by $\psi=(\phi+\psi^\lambda)\circ H^{-1}$ with $m=m(\lambda)$. It is also clear that $w$ is of class $C^{2,\alpha}$ and the surface corresponding to $\Omega_w$ is of class $C^{1,\alpha}$ provided $(\lambda,q,w,\phi)\in\Om$ solves \eqref{eq:F=0}.
	
	Moreover, we rewrite the Bernoulli equation \eqref{eq:OriginalEquations_Bernoulli} for $\psi=(\phi+\psi^\lambda)\circ H^{-1}$:
	\begin{align}
		\K(w)=\frac{|\SL\partial_y\A(\lambda,w,\phi)+\lambda|}{\sqrt{2q+\lambda^2-2gw}}=R,\label{eq:OriginalBernoulliFlat}
	\end{align}
	noticing that
	\begin{align}\label{eq:length_of_nablaV_on_top}
		\SL|\nabla V|=\K(w)
	\end{align}
	due to \eqref{eq:nablaV_on_top}, with $\K(w)$ defined in \eqref{eq:def_K}. Above and throughout this proof, we abbreviate $R\coloneqq R(\lambda,q,w,\phi)$, as defined in \eqref{eq:def_R}. Let us now introduce
	\[G\coloneqq H'=V_y+iV_x,\]
	viewed as a function of $z=x+iy$, which is analytic in $\Omega_h$. In particular,
	\[\SL|G|=\K(w)=R.\]
	By \eqref{eq:additional_requirements} and Lemma \ref{lma:ConformalMapping}(iv) we have that $G\neq 0$ everywhere in $\overline{\Omega_h}$. Moreover, by \eqref{eq:BVP_for_V}, \eqref{eq:intersection_with_bed}, and the Hopf boundary-point lemma, it is evident that
	\begin{align}\label{eq:G_bottom}
		G(x-ih)=V_y(x,-h)+iV_x(x,-h)=V_y(x,-h)\in(0,\infty),\quad x\in\R.
	\end{align}
Therefore, we can introduce the complex logarithm of $G$ by
	\[\log G(z)=\ln G(-ih)+\int_{-ih}^z\frac{G'(\zeta)}{G(\zeta)}\,d\zeta,\quad z\in\overline{\Omega_h}.\]
	Clearly,
	\[e^{\log G}=G\]
	since this is certainly true at $z=-ih$ and the derivative of $Ge^{-\log G}$ is easily seen to be identically zero. Let us now define
	\[\rho\coloneqq\Real\log G,\quad\vartheta\coloneqq\Imag\log G.\]
	These two functions are periodic, and $\rho$ is even in $x$, while $\vartheta$ is odd. Indeed,
	\[\rho(x,y)-\rho(a,y)=\Real\int_a^x\frac{G'(s+iy)}{G(s+iy)}\,ds=\int_a^x\frac{V_yV_{xy}+V_xV_{xx}}{V_x^2+V_y^2}(s,y)\,ds,\]
	where the last integrand is periodic and odd in $s$. Therefore, the choice $a=x+L$ yields periodicity of $\rho$ in $x$, and the choice $a=-x$ yields evenness. On the other hand, by the Cauchy--Riemann equations we have
	\begin{gather*}
		\frac{\partial}{\partial y}(\vartheta(x+L,y)-\vartheta(x,y))=\rho_x(x+L,y)-\rho_x(x,y)=0,\\
		\frac{\partial}{\partial y}(\vartheta(x,y)+\vartheta(-x,y))=\rho_x(x,y)+\rho_x(-x,y)=0.
	\end{gather*}
	Since we already know that $\vartheta|_{y=-h}=0$ by \eqref{eq:G_bottom}, we infer that $\vartheta$ is periodic and odd in $x$. In particular, $\langle\SL\vartheta\rangle=0$ and hence
	\[\Pro\SL\rho=\Ch\SL\vartheta.\]
	Therefore, by $\SL e^\rho=\SL|G|=R$,
	\[1+\Ch w'+iw'=\SL G=\SL e^{\log G}=\SL e^\rho e^{i\vartheta}=Re^{i(\Ch)^{-1}\Pro(\ln R)}.\]
	Hence, on the one hand,
	\[w'=R\sin\left((\Ch)^{-1}\Pro(\ln R)\right).\]
	Since $R$ is even in $x$, the right-hand side is odd, and we can integrate this equation to get
	\[w=\M^2(\lambda,q,w,\phi).\] 
	On the other hand,
	\[1=\langle 1+\Ch w'\rangle=\left\langle R\cos\left((\Ch)^{-1}\Pro(\ln R)\right)\right\rangle,\]
	that is,
	\[q=\M^1(\lambda,q,w,\phi).\]
	
	Conversely, assume that \eqref{eq:F=0} holds. Then
	\[w'=R\sin\left((\Ch)^{-1}\Pro(\ln R)\right),\quad\left\langle R\cos\left((\Ch)^{-1}\Pro(\ln R)\right)\right\rangle=1.\]
	Let now $\vartheta$ be the periodic harmonic function with
	\[\SL\vartheta=(\Ch)^{-1}\Pro(\ln R),\quad\vartheta|_{y=-h}=0,\]
	and $\rho$ periodic such that $\rho+i\vartheta$ is analytic and $\SL\rho=\ln R$. Since $\Imag e^{\rho+i\vartheta}=e^\rho\sin\vartheta$ is harmonic and satisfies
	\[\SL\Imag e^{\rho+i\vartheta}=R\sin\left((\Ch)^{-1}\Pro(\ln R)\right)=w',\quad\Imag e^{\rho+i\vartheta}|_{y=-h}=0,\]
	we have \begin{align*}
		\SL\Real e^{\rho+i\vartheta}&=\langle\SL\Real e^{\rho+i\vartheta}\rangle+\Ch w'=\left\langle R\cos\left((\Ch)^{-1}\Pro(\ln R)\right)\right\rangle+\Ch w'\\
		&=1+\Ch w'.
	\end{align*}
	Hence,
	\[\K(w)=\SL|e^{\rho+i\vartheta}|=\SL e^\rho=R,\]
	which is \eqref{eq:OriginalBernoulliFlat}.
\end{proof}

\begin{remark}
In the irrotational case, $\mathcal A \equiv0$ (and thus $\phi\equiv0$ for solutions), whence the problem reduces to
\begin{subequations}\label{eq:reduced_equations}
\begin{gather}
	\left\langle R(\lambda,q,w)\cos\left((\Ch)^{-1}\Pro(\ln R(\lambda,q,w))\right)\right\rangle=1,\\
	w=\partial_x^{-1}\left(R(\lambda,q,w)\sin\left((\Ch)^{-1}\Pro(\ln R(\lambda,q,w))\right)\right),
\end{gather}
\end{subequations}
where
\[
R(\lambda,q,w)=\frac{|\lambda|}{\sqrt{2q+\lambda^2-2gw}}.
\]
Thus, our formulation does not reduce to Babenko's equation for irrotational flows. Further, if $\gamma$ is constant, then $\A=\A(w)$ depends only on $w$ via
\begin{align*}
	\Delta\A&=\gamma(1-|\nabla V|^2)&\text{in }\Omega_h,\\
	\A&=0&\text{on }y=0\text{ and }y=-h.
\end{align*}
Therefore, the equation $\phi=\A$ should here be viewed rather as the a posteriori definition of $\phi$ than an equation to be solved. Thus, the problem reduces again to \eqref{eq:reduced_equations}, but now with
\[
R(\lambda,q,w)=\frac{|\SL\partial_y\A(w)+\lambda|}{\sqrt{2q+\lambda^2-2gw}}.
\]
In particular, also here our formulation does not reduce to the one in \cite{ConsStrVarv16} for flows with constant vorticity.
\end{remark}

\begin{lemma}\label{lma:M_prop}
	The operator $\M$ (in particular, $\A$) and thus $F$ is of class $C^1$ on $\Om$. In fact, $\M$ is even of class $C^1$ viewed as an operator from $\Om$ to $\R\times C_{0,\per,\e}^{2,\alpha}(\R)\times C_{\per,\e}^{2,\alpha}(\overline{\Omega_h})$, the latter equipped with the canonical norm. If additionally $\gamma\in C_{\text{loc}}^{2,1}(\R)$, then the same statements hold with regularity $C^2$ of the operators instead of $C^1$.
	
	Moreover, for every $\varepsilon>0$ the operator $\M$ is compact on the set
	\begin{align}\label{eq:def_Omeps}
		\Om_\varepsilon\coloneqq\big\{(\lambda,q,w,\phi)\in\Om:\K(w),|\SL\partial_y\A(\lambda,w,\phi)+\lambda|,\\
		\omit\hfill$\displaystyle q+\lambda^2/2-gw\ge\varepsilon\emph{ on }\R\big\},$\nonumber
	\end{align}
	so that $F$ has the form \enquote{identity plus compact} on $\Om_\varepsilon$ in view of \eqref{eq:def_F}.
\end{lemma}
\begin{proof}
	The other operations in the definition of $\M$, \eqref{eq:def_M}, being smooth, the property that $\M$ is of class $C^1$ (or $C^2$) follows from the facts that $\A$ is of class $C^1$ (or $C^2$), which is guaranteed by the assumption $\gamma\in C_{\text{loc}}^{1,1}(\R)$ (or $\gamma\in C_{\text{loc}}^{2,1}(\R)$). Now let $(\lambda,q,w,\phi)\in\Om_\varepsilon$ be arbitrary. In the following, we apply standard Schauder estimates; the quantity $c$ can change from line to line, but is always shorthand for a certain expression in its arguments which remains bounded for bounded arguments. First we have
	\begin{align}\label{eq:est_V}
		\|V\|_{C_\per^{1,\alpha}(\overline{\Omega_h})}\le c\left(\|w\|_{C_\per^{1,\alpha}(\R)}\right).
	\end{align}
	Thus and since $\psi^\lambda$ is of class $C^1$ with respect to $\lambda$, we see that
	\begin{align}
		\|\A(\lambda,w,\phi)\|_{C_\per^{2,\alpha}(\overline{\Omega_h})}&\le c\left(\|\phi\|_{C_\per^{0,\alpha}(\overline{\Omega_h})},\|V\|_{C_\per^{1,\alpha}(\overline{\Omega_h})},|\lambda|\right)\label{eq:est_A}\\
		&\le c\left(\|\phi\|_{C_\per^{0,\alpha}(\overline{\Omega_h})},\|w\|_{C_\per^{1,\alpha}(\R)},|\lambda|\right).\nonumber
	\end{align}
	This shows that $\M^3$ is compact on $\Om$ due to the compact embedding of $C_\per^{2,\alpha}(\overline{\Omega_h})$ in $C_\per^{0,\alpha}(\overline{\Omega_h})$ and in $H_\per^1(\Omega_h)$. As for $\M^1$ and $\M^2$, we proceed with $R$ and find that, for $\beta\in(\alpha,1)$,
	\begin{align}\label{eq:est_R}
		\|R(\lambda,q,w,\phi)\|_{C_\per^{0,\beta}(\R)}&\le c\left(\varepsilon^{-1},\|w\|_{C_\per^{0,\beta}(\R)},\|\A(\lambda,w,\phi)\|_{C_\per^{2,\alpha}(\overline{\Omega_h})},|\lambda|,|q|\right)\\
		&\le c\left(\varepsilon^{-1},\|\phi\|_{C_\per^{0,\alpha}(\overline{\Omega_h})},\|w\|_{C_\per^{1,\alpha}(\R)},|\lambda|,|q|\right),\nonumber\\
		R(\lambda,q,w,\phi)^{-1}&\le c\left(\varepsilon^{-1},\|w\|_{L^\infty(\R)},|\lambda|,|q|\right).\nonumber
	\end{align}
	Thus, clearly $\M^1$ is compact on $\Om_\varepsilon$. Moreover, since $\Ch$ is an isomorphism on $C^{0,\beta}_{0,\per}(\R)$, we conclude
	\begin{align}\label{eq:est_M1}
		\|\M^2(\lambda,q,w,\phi)\|_{C_\per^{1,\beta}(\R)}\le c\left(\varepsilon^{-1},\|\phi\|_{C_\per^{0,\alpha}(\overline{\Omega_h})},\|w\|_{C_\per^{1,\alpha}(\R)},|\lambda|,|q|\right).
	\end{align}
	Hence, $\M^2$ is compact on $\Om_\varepsilon$ since $C_\per^{1,\beta}(\R)$ is compactly embedded in $C_\per^{1,\alpha}(\R)$.
\end{proof}
\begin{remark}
	Let us emphasize that the change of variables $H$ being conformal is crucial in the above argument and thus for everything. Indeed, it is important that there appears only one derivative of $V$ in \eqref{eq:def_A} and therefore in the estimate \eqref{eq:est_A}. However, for a general change of variables $T$ we see that, by calculating how the Laplacian transforms with respect to $T$, no second derivatives of $T$ appear in the corresponding version of \eqref{eq:def_A} if and only if both components of $T^{-1}$ are harmonic -- this is in turn a very restrictive condition on $T$, of course satisfied for our conformal map $H$.
\end{remark}
Let us state also a regularity result for solutions $(\lambda,q,w,\phi)\in\Om$:
\begin{proposition}\label{prop:3alpha_regularity}
	Let $(\lambda,q,w,\phi)\in\Om$ solve $F(\lambda,q,w,\phi)=0$. Then $w\in C_{0,\per,\e}^{3,\alpha}(\R)$ and $\phi\in C_{\per,\e}^{3,\alpha}(\overline{\Omega_h})$.
\end{proposition}
\begin{proof}
	It is clear that at least $w\in C_{0,\per,\e}^{2,\alpha}(\R)$ and $\phi\in C_{\per,\e}^{2,\alpha}(\overline{\Omega_h})$, and therefore in particular $V\in C_{\per,\e}^{2,\alpha}(\overline{\Omega_h})$. Now the function $\phi_x$ solves
	\begin{align*}
		\Delta\phi_x&=f&\text{in }\Omega_h,\\
		\phi_x&=0&\text{on }y=0,\\
		\phi_x&=0&\text{on }y=-h,
	\end{align*}
	with
	\[f=-\gamma'(\phi+\psi^\lambda)\phi_x|\nabla V|^2-2\gamma(\phi+\psi^\lambda)(V_xV_{xx}+V_yV_{xy})\in C_\per^{0,\alpha}(\overline{\Omega_h}),\]
	at first at least in the weak sense, but then also in the strong sense due to uniqueness. Thus, $\phi_x\in C_\per^{2,\alpha}(\overline{\Omega_h})$. Moreover, $(\phi_y)_x=(\phi_x)_y\in C_\per^{1,\alpha}(\overline{\Omega_h})$ and
	\[\phi_{yy}=-\phi_{xx}-\gamma(\phi+\psi^\lambda)|\nabla V|^2+\gamma(\psi^\lambda)\in C_\per^{1,\alpha}(\overline{\Omega_h});\]
	hence, $\phi\in C_{\per,\e}^{3,\alpha}(\overline{\Omega_h})$. Therefore, $R(\lambda,q,w,\phi)\in C_{\per,\e}^{2,\alpha}(\R)$ and thus $w\in C_{\per,\e}^{3,\alpha}(\R)$.
\end{proof}

\begin{remark}\label{rem:TrivialSolutions}
	Before we proceed with investigating bifurcations, we need to take a closer look at \enquote{trivial} solutions, the discussion being relevant throughout. A very desirable property would be that points of the form $(\lambda,0,0,0)$ (what we call trivial solutions) are exactly the solutions with a flat surface, that is, $w\equiv0$. In fact, this is true in certain cases for $\gamma$, but false in other cases, and closely connected to strict monotonicity of $m(\lambda)$, which was introduced in \eqref{eq:def_m}, as we will explain now.
	
	On the one hand, if
	\begin{align}\label{eq:sup_gamma'_DirEV}
		\sup\gamma'<\frac{\pi^2}{h^2},
	\end{align}
	we have that $q=0$ and $\phi\equiv0$ whenever $F(\lambda,q,0,\phi)=0$. Indeed, if $\phi\not\equiv0$ in the latter case, since $|\nabla V|=1$ here, then
	\begin{align*}
		\iint_{\Omega_h^*}|\nabla\phi|^2\,dx\,dy&=\iint_{\Omega_h^*}\phi(\gamma(\phi+\psi^\lambda)-\gamma(\psi^\lambda))\,dx\,dy=\iint_{\Omega_h^*}\phi^2\gamma'(\xi)\,dx\,dy\\
		&<\frac{\pi^2}{h^2}\iint_{\Omega_h^*}\phi^2\,dx\,dy
	\end{align*}
	for some $\xi$ between $\psi^\lambda$ and $\phi+\psi^\lambda$ (depending on $(x,y)$), where $\Omega_h^*=(0,L)\times(-h,0)$. But this contradicts the fact that $\pi^2/h^2$ is the first Dirichlet eigenvalue of $-\partial_y^2$ on $[-h,0]$ and therefore the first Dirichlet eigenvalue of $-\Delta$ on $\R\times[-h,0]$ with periodicity in $x$. Thus, $\phi\equiv0$ and hence also $q=0$. Moreover, since
	\[\partial_\lambda\psi^\lambda_{yy}+\gamma'(\psi^\lambda)\partial_\lambda\psi^\lambda=0\quad\text{on }[-h,0],\qquad\partial_\lambda\psi^\lambda(0)=0,\quad\partial_\lambda\psi^\lambda_y(0)=1,\]
	it follows by classical Sturm--Liouville theory that
	\begin{align}\label{eq:est_m'}
		m'(\lambda)=-\partial_\lambda\psi^\lambda(-h)\ge\varepsilon
	\end{align}
	for some $\varepsilon>0$ independent of $\lambda$.
	
	On the other hand, without \eqref{eq:sup_gamma'_DirEV} the above reasoning can no longer be applied. Indeed, if for example $\gamma(\psi)=a\psi$ for a Dirichlet eigenvalue $a>0$ of $-\partial_y^2$ on $[-h,0]$, then clearly $m(\lambda)=0$ for any $\lambda\in\R$. Moreover, for $\lambda\ne\tilde\lambda$ and recalling \eqref{eq:def_q}, the points
	\[(\lambda,0,0,0)\quad\text{and}\quad\left(\tilde\lambda,\frac{\lambda^2-\tilde\lambda^2}{2},0,\psi^\lambda-\psi^{\tilde\lambda}\right)\]
	provide the same physical configuration with a flat surface and stream function $\psi^\lambda$, but $\phi=\psi^\lambda-\psi^{\tilde\lambda}\not\equiv0$ in the second representation. Notice that here it is important that this $\phi$ still vanishes on $y=-h$ as demanded in \eqref{eq:def_X}, which is precisely the fact that $m(\lambda)=m(\tilde\lambda)$. This observation implies, more generally, that such two different representations of one trivial physical configuration can always be found whenever $m$ is not injective.
	
	At this point, we touch upon another issue, which might be seen as a small drawback of our reformulation: It is not clear in general that $m(\R)=\R$, that is, any solution of \eqref{eq:OriginalEquations} for any $m\in\R$ falls within the scope of our setting. In certain cases the property $m(\R)=\R$ holds, as, for example, the above reasoning in the case \eqref{eq:sup_gamma'_DirEV} shows, but may fail in other cases, as described above for certain linear $\gamma$.
	
	To sum up, throughout this paper, the reader should have in mind that our trivial solutions $(\lambda,0,0,0)$ exhaust the set of solutions with a flat surface provided \eqref{eq:sup_gamma'_DirEV}, but if \eqref{eq:sup_gamma'_DirEV} does not hold there might be solutions of the form $(\lambda,q,0,\phi)$ with $q\ne0$ and $\phi\not\equiv0$. Also notice that the distinction between these two scenarios via \eqref{eq:sup_gamma'_DirEV} is sharp, as can be seen by letting $a\coloneqq\pi^2/h^2$ in the example above.
\end{remark}

\section{Local bifurcation}\label{sec:LocalBifurcation}
We first investigate local bifurcations in this section, and then turn to global bifurcations in Section \ref{sec:GlobalBifurcation}. On the one hand, the local theory is already quite well understood \cite{KozKuz14,Varholm20} and not the main motivation of the present paper. On the other hand, we still want to convince the reader that the local theory can also be developed in our new reformulation. Therefore, as a compromise, we only state the main results here and defer their proofs as well as computations on the way towards them to Appendix \ref{appx:LocalBifurcation}.

We strive to apply the classical local bifurcation theorem of Crandall--Rabinowitz \cite[Theorem I.5.1]{Kielhoefer}, stated as Theorem \ref{thm:CrandallRabinowitz} in Appendix \ref{appx:LocalBifurcation}. Since Crandall--Rabinowitz requires $C^2$-regularity of the operator, we assume in view of Lemma \ref{lma:M_prop} that the vorticity function, additionally to being globally Lipschitz, satisfies
\[\gamma\in C_{\text{loc}}^{2,1}(\R)\]
throughout this Section \ref{sec:LocalBifurcation} and Appendix \ref{appx:LocalBifurcation}.

In order to state the results, we introduce the function $\beta=\beta^{\mu,\lambda}$, which is defined to be the unique solution of the boundary value problem
\begin{align}\label{eq:beta}
	\beta_{yy}+(\gamma'(\psi^\lambda)+\mu)\beta=0\text{ on }[-h,0],\quad\beta(-h)=0,\quad\beta(0)=1.
\end{align}
This can be done for all $\lambda,\mu\in\R$ such that $\mu$ is not in the Dirichlet spectrum of $-\partial_y^2-\gamma'(\psi^\lambda)$ on $[-h,0]$. As usual in local bifurcation theory, a dispersion relation and a transversality condition comes into play. Here,
\[d(-(k\nu)^2,\lambda)=0\]
is the dispersion relation, where
\begin{align}\label{eq:d(k,lambda)}
	d(\mu,\lambda)\coloneqq\beta_y^{\mu,\lambda}(0)+\lambda^{-1}\gamma(0)-\lambda^{-2}g,
\end{align}
and
\[d_\lambda(-(k\nu)^2,\lambda)\neq 0\]
is the transversality condition.

Now, using Theorem \ref{thm:CrandallRabinowitz}, we have the following result on local bifurcation:
\begin{theorem}\label{thm:LocalBifurcation}
	Assume that there exists $\lambda_0\neq 0$ such that 
	\begin{align}\label{ass:SL-spectrum}
		0\text{ is not in the Dirichlet spectrum of }-\partial_y^2-\gamma'(\psi^\lambda)\text{ on }[-h,0]
	\end{align}
	for $\lambda=\lambda_0$ and such that the dispersion relation
	\[d(-(k\nu)^2,\lambda_0)=0,\]
	with $d$ given by \eqref{eq:d(k,lambda)}, has exactly one solution $k_0\in\N$, and assume that the transversality condition
	\[d_\lambda(-(k_0\nu)^2,\lambda_0)\neq 0\]
	holds. Then there exists $\varepsilon>0$ and a $C^1$-curve $(-\varepsilon,\varepsilon)\ni s\mapsto(\lambda^s,q^s,w^s,\phi^s)$ such that $(\lambda^0,q^0,w^0,\phi^0)=(\lambda_0,0,0,0)$, $w^s\not\equiv0$ for $s\neq 0$, and such that $F(\lambda^s,q^s,w^s,\phi^s)=0$. Moreover, all solutions of $F(\lambda,q,w,\phi)=0$ in a neighborhood of $(\lambda_0,0,0,0)$ are on this curve or are trivial. Furthermore, the curve admits the asymptotic expansion $(q^s,w^s,\phi^s)=(0,s\T(\lambda_0)\theta)+o(s)$, where
	\begin{align}
		\theta(x,y)&=\beta^{-(k_0\nu)^2,\lambda_0}(y)\cos(k_0\nu x),\nonumber\\
		[\T(\lambda_0)\theta](x,y)&=\left(-\frac{1}{\lambda_0},\beta^{-(k_0\nu)^2,\lambda_0}(y)-\frac{\psi_y^{\lambda_0}(y)\sinh(k_0\nu(y+h))}{\lambda_0\sinh(k_0\nu h)}\right)\label{eq:CR_Ttheta}\\
		&\phantom{=\;}\cdot\cos(k_0\nu x),\nonumber\\
		\frac{o(s)}{s}&\overset{s\to0}\longrightarrow0\quad\mathrm{in\ }\R\times C_{0,\per,\e}^{2,\alpha}(\R)\times C_{\per,\e}^{2,\alpha}(\overline{\Omega_h}).\nonumber
	\end{align}
	These solutions give rise to proper water waves, that is, additionally \eqref{eq:additional_requirements} is satisfied.
\end{theorem}
At this point, $\T(\lambda_0)\theta$ can be just thought of as a placeholder for the term at linear order in $s$ in the asymptotic expansion. However, both $\theta$ and $\T$ have a deeper meaning, and we refer to Appendix \ref{appx:LocalBifurcation:CR} for more details.

Let us also state two general results on the dispersion relation and the transversality condition.
\begin{proposition}\label{prop:solvability_disprel}
	Assume that $\sup \gamma'< \pi^2/h^2$. Then condition \eqref{ass:SL-spectrum} is satisfied. If additionally 
	\[
	\lambda^{-2}g -\lambda^{-1}\gamma(0)> 
	\begin{cases} 
		\sqrt{\inf \gamma'} \cot (h\sqrt{\inf \gamma'}) & \text{if } \inf \gamma'>0,\\
		1/h & \text{if } \inf \gamma'=0,\\
		\sqrt{|\inf \gamma'|} \coth (h\sqrt{|\inf \gamma'}|) & \text{if } \inf \gamma'<0,
	\end{cases}
	\]
	there is precisely one negative solution $\mu_0$ to the equation $d(\mu, \lambda)=0$. In particular,
	\begin{enumerate}[label=(\roman*)]
		\item the dispersion relation $d(-(k\nu)^2,\lambda)=0$ has in $\N$ the unique solution $k_0$ if $k_0\in\N$ and $\nu>0$ are chosen such that $-(k_0\nu)^2=\mu_0$;
		\item this criterion is satisfied for sufficiently small $|\lambda|$.
	\end{enumerate}
\end{proposition}
Note that the above condition reduces to $\lambda^2<gh$ in the irrotational case and to the subcriticality condition $\lambda^{-2}g -\lambda^{-1} \gamma > 1/h$ in the case of constant vorticity $\gamma$. For a further analysis of the dispersion relation we refer to \cite{KozKuz14,KKL14}.
\begin{proposition}\label{prop:transversality_redundant}
	Assume $\sup \gamma'< \pi^2/(4h^2)$, $\mu\le0$, and
	\begin{enumerate}[label=(\roman*)]
		\item $\gamma''\ge0$ and $\lambda>0$, or
		\item $\gamma''\le0$ and $\lambda<0$.
	\end{enumerate}
	Then $d_\lambda(\mu,\lambda)\ne0$ provided $d(\mu,\lambda)=0$.
\end{proposition}
Finally, for an investigation of the conditions for local bifurcation for certain examples of $\gamma$ we refer to Appendix \ref{appx:LocalBifurcation:cond}.

\section{Global bifurcation}\label{sec:GlobalBifurcation}
\subsection{General conclusions}
Having quickly considered local bifurcations, let us now turn to global bifurcation, which is one of the main motivations of this paper and of our formulation \enquote{identity plus compact}.

We apply now the corresponding global bifurcation theorem by Rabinowitz \cite[Theorem II.3.3]{Kielhoefer}, stated for convenience as Theorem \ref{thm:Rabinowitz} in Appendix \ref{appx:GlobalBif}, to our problem. Here, we emphasize that we only restrict ourselves to solutions satisfying \eqref{eq:additional_requirements}, which correspond to meaningful solutions of \eqref{eq:OriginalEquations}; notice that it is not impossible that there might be (irrelevant) solutions of $F(\lambda,q,w,\phi)=0$ for which \eqref{eq:additional_requirements} is not satisfied and thus Lemma \ref{lma:Reformulation} does not provide equivalence to the original equations \eqref{eq:OriginalEquations}.

\begin{theorem}\label{thm:GlobalBifurcation}
	Assume that there exists $\lambda_0\neq 0$ such that $F_{(q,w,\phi)}(\lambda,q,w,\phi)$ has an odd crossing number at $(\lambda_0,0,0,0)$. Let $\mathfrak S$ denote the closure of the set of nontrivial solutions of $F(\lambda,q,w,\phi)=0$ in $\R\times X$ satisfying \eqref{eq:additional_requirements} (that is, solutions giving rise to a proper water wave) and $\Cc$ denote the connected component of $\mathfrak S$ to which $(\lambda_0,0,0,0)$ belongs. Then one of the following alternatives occurs:
	\begin{enumerate}[label=(\roman*)]
		\item $\Cc$ is unbounded in the sense that there exists a sequence of points $(\lambda_n,q_n,w_n,\phi_n)\in\Cc\cap\Om$ such that
		\begin{enumerate}[label=(\alph*)]
			\item $|\lambda_n|\to\infty$, or
			\item $\|w_n\|_{C_\per^{0,\delta}(\R)}\to\infty$ for any $\delta\in(5/6,1]$, or
			\item $\|\gamma((\phi_n+\psi^{\lambda_n})\circ H[w_n+h]^{-1})\|_{L^p(\Omega_{w_n}^*)}\to\infty$ for any $p>1$ (that is, the total vorticity in each copy of the fluid domain measured in $L^p$ is unbounded)
		\end{enumerate}
		as $n\to\infty$;
		\item $\Cc$ contains a point $(\bar\lambda,0,0,0)$ with $\bar\lambda\neq\lambda_0$;
		\item $\Cc\cap\Om$ contains a sequence $(\lambda_n,q_n,w_n,\phi_n)$ such that $q_n+\lambda_n^2/2-g\max_\R w_n\to0$ as $n\to\infty$, that is, a wave of greatest height is approached;
		\item $\Cc\cap\Om$ contains a sequence $(\lambda_n,q_n,w_n,\phi_n)$ such that $\min_\R\K(w_n)\to0$  as $n\to\infty$;
\item $\Cc$ contains a point $(\lambda,q,w,\phi)$ such that
		\[x\mapsto(x+(\Ch w)(x),w(x)+h)\mathrm{\ is\ \textit{not}\ injective\ on\ }\R,\]
		that is, self-intersection of the surface profile occurs;
		\item $\Cc$ contains a point $(\lambda,q,w,\phi)$ such that there exists $x\in\R$ with
		\[w(x)=-h,\]
		that is, intersection of the surface profile with the flat bed occurs.
	\end{enumerate}
\end{theorem}
\begin{proof}
	As was already observed in Lemma \ref{lma:M_prop}, our nonlinear operator $F$ is of class $C^1$ and admits the form \enquote{identity plus compact} on each $\Om_\varepsilon$, $\varepsilon>0$, as defined in \eqref{eq:def_Omeps}. For each $\varepsilon>0$, we can thus apply Theorem \ref{thm:Rabinowitz} with $\mathcal U$ chosen to be the interior of $\Om_\varepsilon$. Suppose now that neither alternative (v) nor (vi) is valid. Then, on each $\Om_\varepsilon$, $\Cc$ coincides with its counterpart obtained from Theorem \ref{thm:Rabinowitz}. Since $\varepsilon>0$ is arbitrary and $\Om=\bigcup_{\varepsilon>0}\Om_\varepsilon$, it is evident that, whenever $\Cc$ is bounded in $\R\times X$ and (ii) fails to hold,
there exists a sequence $(\lambda_n,q_n,w_n,\phi_n)$ in $\Cc\cap\Om$ such that
	\begin{subequations}\label{eq:BoundaryAlternative}
	\begin{align}
		\min_{\R}\K(w_n)&\to0,\text{ or}\label{eq:BoundaryAlternative1}\\
		\min_{\R}|\SL\partial_y\phi_n+\lambda_n|&\to0,\text{ or}\label{eq:BoundaryAlternative2}\\
		q_n+\lambda_n^2/2-g\max_\R w_n&\to 0,\label{eq:BoundaryAlternative3}
	\end{align}
	\end{subequations}
	as $n\to\infty$. Notice, however, that if \eqref{eq:BoundaryAlternative1} fails to hold but \eqref{eq:BoundaryAlternative2} does hold, then automatically also \eqref{eq:BoundaryAlternative3} is satisfied in view of \eqref{eq:OriginalBernoulliFlat}. Therefore, we can ignore \eqref{eq:BoundaryAlternative2}.

	It remains to investigate alternative (i) further. Let us recall at this point the definition of $X$, $F$, and $\M$ given in \eqref{eq:def_X}, \eqref{eq:def_F}, and \eqref{eq:def_M}, respectively. In order to show that (i) can be stated as above, we have to prove that, in view of alternative (i) of Theorem \ref{thm:Rabinowitz}, $\Cc$ is bounded in $\R\times X$ if (a)--(c) (and also the other alternatives (ii),(iii),(iv)) fail to hold. In particular, we shall assume that $(\lambda,q,w,\phi)\in\Cc\cap\Om$, the $\lambda$'s are bounded in $\R$, the $w$'s are bounded in $C_\per^{0,\delta}(\R)$ for some $\delta\in(5/6,1)$, $\|\gamma((\phi+\psi^{\lambda})\circ H[w+h]^{-1})\|_{L^p(\Omega_w^*)}$ remains bounded for some $p>1$, and
	\[\K(w)\ge\varepsilon,\quad q+\lambda^2/2-gw\ge\varepsilon,\quad|\SL\phi_y+\lambda|=\K(w)\sqrt{q+\lambda^2/2-gw}\ge\varepsilon\]
	for some $\varepsilon>0$. Since the  $L^\infty$-norms of the $w$'s are bounded, the Lebesgue measure of $\Omega_w^*$ is bounded so that we can assume without loss of generality that $p$ is close enough to $1$ such that
	\[\eta\coloneqq2-2/p<\min\{1/(1-\delta)-6,\alpha\}\in(0,\infty).\]
	We now consider $\delta$, $\varepsilon$, and $\eta$ fixed, and let $\beta\in(0,\eta]$, $r,t\in(1,\infty)$ to be chosen later. In the following, $c>0$ denotes some generic constants that may depend on $\delta$, $\varepsilon$, $\eta$, $\beta$, $r$, $t$, $|\lambda|$, and $\|\gamma((\phi+\psi^{\lambda})\circ H[w+h]^{-1})\|_{L^p(\Omega_w^*)}$, and can change from step to step. We will also frequently make use of Young's inequality $|ab|\le|a|^s/s+|b|^{s'}/s'$ for $a,b\in\R$ and $s\in(1,\infty)$, where $1/s+1/s'=1$; we let $r'$ and $t'$ be such that $1/r+1/r'=1$ and $1/t+1/t'=1$. First note that
	\[|\gamma(s)|\le c|s|+c,\quad s\in\R,\]
	and
	\[\|V[w+h]\|_{C_\per^{1,\beta}(\overline{\Omega_h})}\le c\|w\|_{C_\per^{1,\beta}(\R)}+c\]
	similarly to \eqref{eq:est_V}. Now, along $\Cc\cap\Om$ we have $\phi=\A(\lambda,w,\phi)$ and thus
	\begin{align*}
		&\|\phi\|_{C_\per^{0,\eta}(\overline{\Omega_h})\cap H_\per^1(\Omega_h)}\\
		&\le c\|\phi\|_{W^{2,p}(\Omega_h^*)}\le c\|-\gamma(\phi+\psi^\lambda)|\nabla V[w+h]|^2+\gamma(\psi^\lambda)\|_{L^p(\Omega_h^*)}\\
		&\le c\|\gamma(\phi+\psi^\lambda)|\nabla V[w+h]|^{1/p}\|_{L^p(\Omega_h^*)}\|\nabla V[w+h]\|_\infty^{1+\eta/2}+c\\
		&\le c\|\gamma((\phi+\psi^{\lambda})\circ H[w+h]^{-1})\|_{L^p(\Omega_w^*)}\left(\|w\|_{C_\per^{1,\beta}(\R)}^{1+\eta/2}+1\right)+c\\
		&\le c\|w\|_{C_\per^{1,\beta}(\R)}^{r(1+\eta/2)}+c
	\end{align*}
	after using Sobolev's embedding, the Calderón--Zygmund inequality (see \cite[Chapter 9]{GilbargTrudinger}; notice that on the right-hand side the term $\|\phi\|_{L^p(\Omega_h^*)}$ can be left out because of unique solvability of the Dirichlet problem associated to $\Delta$), $2-1/p=1+\eta/2$, the Schauder estimate for $V[w+h]$, and a change of variables via $H[w+h]$. We will now derive somewhat refined versions of \eqref{eq:est_R}, \eqref{eq:est_M1}. First, we infer the Schauder estimate
	\begin{align}\label{eq:est_phi}
		\|\phi\|_{C_\per^{2,\beta}(\overline{\Omega_h})}&\le c\left(\|\phi\|_{C_\per^{0,\beta}(\overline{\Omega_h})}+1\right)\|V[w+h]\|_{C_\per^{1,\beta}(\overline{\Omega_h})}^2+c\\
		&\le c\|\phi\|_{C_\per^{0,\beta}(\overline{\Omega_h})}^{\frac{2}{r(1+\eta/2)}+1}+c\|V[w+h]\|_{C_\per^{1,\beta}(\overline{\Omega_h})}^{2+r(1+\eta/2)}+c\nonumber\\
		&\le c\|w\|_{C_\per^{1,\beta}(\R)}^{2+r(1+\eta/2)}+c.\nonumber
	\end{align}
	Next, by the original Bernoulli equation \eqref{eq:OriginalBernoulliFlat} we have
	\[R(\lambda,q,w,\phi)=\K(w)\ge\varepsilon.\]
	Then, because of $w=\M^2(\lambda,q,w,\phi)$ along $\Cc\cap\Om$ and using the Lipschitz continuity of $s\mapsto\ln s$ and $s\mapsto 1/\sqrt{s}$ on $[\varepsilon,\infty)$ for the terms involving $R(\lambda,q,w,\phi)$ and $q+\lambda^2/2-gw$, for $\kappa\in\{\alpha,\delta\}$,
	\begin{align}
		&\|w\|_{C_\per^{1,\kappa}(\R)}\le c\|w'\|_{C_\per^{0,\kappa}(\R)}\le c\|R(\lambda,q,w,\phi)\|_{C_\per^{0,\kappa}(\R)}^2\nonumber\\
		&\le c\left(\|\phi\|_{C_\per^{2,\beta}(\overline{\Omega_h})}^2+1\right)\left(\|w\|_{C_\per^{0,\kappa}(\R)}^2+1\right)\nonumber\\
		&\le c\left(\|w\|_{C_\per^{1,\beta}(\R)}^{4+r(2+\eta)}+1\right)\left(\|w\|_{C_\per^{0,\kappa}(\R)}^2+1\right)\label{eq:w_est0}\\
		&\le c\left(\|w\|_{C_\per^{1,\kappa}(\R)}^{(4+r(2+\eta))(1+\beta-\kappa)}\|w\|_{C_\per^{0,\kappa}(\R)}^{(4+r(2+\eta))(\kappa-\beta)}+1\right)\left(\|w\|_{C_\per^{0,\kappa}(\R)}^2+1\right)\nonumber\\
		&\le\frac12\|w\|_{C_\per^{1,\kappa}(\R)}^{t(4+r(2+\eta))(1+\beta-\kappa)}\label{eq:w_est}\\
		&\phantom{=\;}+\left(\frac{2}{t}\right)^{t'/t}\frac{1}{t'}\left(c\|w\|_{C_\per^{0,\kappa}(\R)}^{(4+r(2+\eta))(\kappa-\beta)}\left(\|w\|_{C_\per^{0,\kappa}(\R)}^2+1\right)\right)^{t'}\nonumber\\
		&\phantom{=\;}+c\left(\|w\|_{C_\per^{0,\kappa}(\R)}^2+1\right)\nonumber
	\end{align}
	where we demanded $\beta\le\kappa$ and interpolated $C^{1,\beta}$ between $C^{0,\kappa}$ and $C^{1,\kappa}$. We now choose $\beta\le\min\{\delta,\eta\}\le\min\{\alpha,\delta\}$ close enough to $0$ and $r,t$ close enough to $1$ such that
	\[t(4+r(2+\eta))(1+\beta-\delta)=1.\]
	This is possible since
	\[(4+(2+\eta))(1-\delta)<1\quad\Leftrightarrow\quad\eta<\frac{1}{1-\delta}-6.\]
	Hence, by \eqref{eq:w_est} for $\kappa=\delta$,
	\[\frac12\|w\|_{C_\per^{1,\delta}(\R)}\le c\|w\|_{C_\per^{0,\delta}(\R)}^c+c.\]
	Together with \eqref{eq:w_est0} for $\kappa=\alpha$, this yields
	\[\|w\|_{C_\per^{1,\alpha}(\R)}\le c\left(\|w\|_{C_\per^{1,\delta}(\R)}^{4+r(2+\eta)}+1\right)\left(\|w\|_{C_\per^{1,\delta}(\R)}^2+1\right)\le c\|w\|_{C_\per^{0,\delta}(\R)}^c+c.\]
	This, inserting this inequality in \eqref{eq:est_phi}, and the obvious estimate
	\[|q|\le c\left(\|\phi\|_{C_\per^{1,\alpha}(\overline{\Omega_h})}^2+\|w\|_{L^\infty(\R)}+1\right)\]
	in view of the Bernoulli equation \eqref{eq:OriginalBernoulliFlat} completes the proof.
\end{proof}
Let us discuss the conditions and statements in Theorem \ref{thm:GlobalBifurcation} in more detail:

\begin{remark}\makeatletter\hyper@anchor{\@currentHref}\makeatother\label{rem:GlobalBifurcation}
	\begin{enumerate}[label=(\alph*),leftmargin=*]
		\item It is well-known that $F_{(q,w,\phi)}(\lambda,q,w,\phi)$ has an odd crossing number at $(\lambda_0,0,0,0)$ provided $F_{(q,w,\phi)}(\lambda_0,0,0,0)$ is a Fredholm operator with index zero and one-dimensional kernel, and the transversality condition holds. These properties, in turn, are consequences of the hypotheses of Theorem \ref{thm:LocalBifurcation}; notice that we needed there the assumption $\gamma\in C_{\text{loc}}^{2,1}(\R)$. However, it is far from necessary to ensure these conditions of the local Crandall--Rabinowitz theorem \ref{thm:CrandallRabinowitz} in order to meet the condition on an odd crossing number. For example, if the kernel is of odd dimension, there is a criterion which includes a certain transversality condition and is thus somewhat an analogue of the \enquote*{one-dimensional conditions} in Crandall--Rabinowitz; see \cite[pp. 213 ff.]{Kielhoefer}. We will however not pursue this issue further in this paper.
		\item We briefly discuss the alternative (i)(a), that is, $\lambda$ is unbounded. Physically speaking, this means that the surface velocity of the underlying (trivial) laminar flow is unbounded. We can also relate this to the relative mass flux $m(\lambda)$, at least in certain cases for $\gamma$: First, if $\lambda>0$ and $\gamma\ge0$, then clearly $\psi^\lambda_y\ge\lambda$ and hence $m(\lambda)=-\psi^\lambda(-h)\ge\lambda h$. Thus, $m(\lambda)\to\infty$ as $\lambda\to\infty$ provided $\gamma\ge0$. Similarly, $m(\lambda)\to-\infty$ as $\lambda\to-\infty$ provided $\gamma\le0$. For more discussion of these situations see Section \ref{sec:downstream}. Second, if $\sup\gamma'<\pi^2/h^2$, then $m(\lambda)\to\pm\infty$ as $\lambda\to\pm\infty$, as was already observed in \eqref{eq:est_m'} in Remark \ref{rem:TrivialSolutions}. Notice, however, that in general we do not necessarily have $|m(\lambda)|\to\infty$ as $\lambda\to\infty$ (or $-\infty$), recalling once again Remark \ref{rem:TrivialSolutions} where we saw that $m(\lambda)=0$ for any $\lambda\in\R$ if $\gamma(\psi)=a\psi$ for a Dirichlet eigenvalue $a>0$ of $-\partial_y^2$ on $[-h,0]$.
		\item Also in view of Remark \ref{rem:TrivialSolutions}, the violation of alternative (ii)  only implies that there are on $\Cc$ no other configurations with a flat surface than the one we bifurcate from if $\sup\gamma'<\pi^2/h^2$. If this condition on $\gamma$ is violated, then we can not eliminate the possibility that $\Cc$ remains in the physically small regime (that is, $w$ close to $0$) although it contains points with large norm in our function space $\R\times X$. We will have a closer look at this in Section \ref{sec:nodalanalysis} when studying nodal properties.
		\item It is important to point out that, as a further advantage of our new reformulation, we could manage to lower the norm for blow-up of $w$ in alternative (i)(b) to $C^{0,\delta}$, $5/6<\delta\le1$. This seems to be the first time (except for the paper \cite{WahlenWeber21} of the present authors dealing with the capillary-gravity case) that this norm could be lowered that much (without restricting to a special case for $\gamma$ in the first place).
		\item If, for example, $\gamma$ is bounded, Theorem \ref{thm:GlobalBifurcation} remains obviously valid if alternative (i)(c) is left out.
\item We will say more in Appendix \ref{appx:degeneracy_conformal} about alternative (iv), that is, degeneracy of the conformal map. However, as clarified there, if we want to say more about this alternative, in particular by extracting a limiting configuration via some compactness, then we have to consider a more regular norm for $w$ in the unboundedness alternative (i)(b) -- this is in turn the reason why we defer the discussion to the appendix.\end{enumerate}
\end{remark}

\subsection{Deeper conclusions under a spectral assumption}\label{sec:nodalanalysis}
It is well-known (see, for example, the breakthrough papers \cite{ConsStr04,ConsStrVarv16}) that the pure gravity case enjoys the advantage over the capillary-gravity case that structural properties of the obtained solutions can be derived by means of applying certain maximum principles. We can also perform such a nodal analysis for our obtained solutions. To this end, however, we have to assume
\begin{align}\label{ass:Pruefergraph}
	\beta^{-(k_0\nu)^2,\lambda_0}>0\text{ on }(-h,0].
\end{align}
Notice that, by classical Sturm--Liouville theory for \eqref{eq:beta}, this condition equivalent to $-(k_0\nu)^2$ being smaller than the first Dirichlet eigenvalue of $-\partial_y^2-\gamma'(\psi^{\lambda_0})$ on $[-h,0]$. Thus, for example, \eqref{ass:Pruefergraph} is always satisfied whenever \eqref{eq:sup_gamma'_DirEV} holds (which in fact also guarantees \eqref{ass:SL-spectrum}). 

The strategy and the mathematical details of this section are similar to those in \cite{ConsStrVarv16}, however more general since we do not assume that $\gamma$ is a constant. We would also like to mention the recent paper \cite{Koz23} at this point. Throughout the rest of this Section \ref{sec:GlobalBifurcation}, we assume that the local theory has been set up by means of Theorem \ref{thm:LocalBifurcation} (and Remark \ref{rem:GlobalBifurcation}(a)), that is, we suppose that there exists $\lambda_0\ne0$ and $k_0\in\N$ such that the assumptions of Theorem \ref{thm:LocalBifurcation} are met for this $\lambda_0$ and $k_0$; therefore, due to Theorem \ref{thm:LocalBifurcation}, there is a local curve of solutions near $(\lambda_0,0,0,0)$.
\subsubsection*{Nodal analysis}
Let us denote by $\Cc_\pm$ the connected component of $(\Cc\cap\Om)\setminus\{(\lambda,q,w,\phi)\in\R\times X:w\equiv0\}$ to which $(\lambda^{\pm s_0},q^{\pm s_0},w^{\pm s_0},\phi^{\pm s_0})$ belongs, where $s_0\in(0,\varepsilon)$ is arbitrary and $\varepsilon>0$ as in Theorem \ref{thm:LocalBifurcation}. Notice that at this point it is not yet clear that $\Cc_+\ne\Cc_-$. Also, it is not always true that
$(\Cc\cap\Om)\setminus\{(\lambda,q,w,\phi)\in\R\times X:w\equiv0\}$ coincides with $(\Cc\cap\Om)\setminus\{(\lambda,q,w,\phi)\in\R\times X:q=0,w\equiv0,\phi\equiv0\}$. In fact, recalling Remark \ref{rem:TrivialSolutions}, this is true whenever \eqref{eq:sup_gamma'_DirEV} holds, but might be violated else.

For simplicity, we assume without loss of generality that $k_0=1$ for otherwise we could have done everything until this point with $(L/k_0,1)$ instead of $(L,k_0)$; notice also that by the uniqueness property in Theorem \ref{thm:LocalBifurcation} the local curves in the spaces with period $L$ and $L/k_0$ have to coincide, up to reparametrization, so that all solutions on the original local curve and thus on all of $\Cc_\pm$ are necessarily $L/k_0$-periodic. Moreover, for notational simplicity we will only consider $\Cc_+$ and the case $\lambda_0<0$, but an analogous result also holds for $\Cc_-$ and/or the case $\lambda_0>0$.
In fact, let us quickly remark that $\Cc_-$ coincides with $\Cc_+$ up to the translation $x\mapsto x+L/2$ of the horizontal variable, which leaves the space $\R\times X$ and the solution set invariant. Indeed, up to this translation, $\Cc_-$ and $\Cc_+$ coincide locally near $(\lambda_0,0,0,0)$ by virtue of \eqref{eq:CR_Ttheta} and the uniqueness statement in Theorem \ref{thm:LocalBifurcation}, and therefore also globally because $\Cc_-$ and $\Cc_+$ are both defined as connected components.

Let us study the following nodal properties for $(\lambda,q,w,\phi)\in\Cc_+$:
\begin{gather}
	w(x)>-h\quad\text{for all }x\in\R,\label{eq:NoIntersectionBed}\\
w'(x)<0\quad\text{for all }x\in(0,L/2),\label{eq:MonotoneCrestTrough}\\
	w\in C_{0,\per,\e}^{2,\alpha}(\R)\quad\text{and}\quad w''(0)<0,\;\quad w''(L/2)>0,\label{eq:MonotoneCrestTroughSecondOpen}\\
	0<x+(\Ch w)(x)<L/2\quad\text{for all }x\in(0,L/2),\label{eq:NoSelfIntersection}\\
	1+(\Ch w')(0)>0\quad\text{and}\quad 1+(\Ch w')(L/2)>0,\label{eq:NoSelfIntersectionSecondOpen}\end{gather}
(In general, one has to multiply the inequalities in \eqref{eq:MonotoneCrestTrough} and \eqref{eq:MonotoneCrestTroughSecondOpen} by $\mp\sgn\lambda_0$ when considering $(\lambda,q,w,\phi)\in\Cc_\pm$.) Here, \eqref{eq:NoIntersectionBed} means that no intersection of the free surface with the flat bed occurs and \eqref{eq:MonotoneCrestTrough} says that the height of the surface is strictly monotone from crest to trough. Moreover, \eqref{eq:NoSelfIntersection} is equivalent to the property that no self-intersection of the free surface occurs, provided \eqref{eq:MonotoneCrestTrough}; see \cite[Lemma 11]{ConsStrVarv16}. Notice that the condition \eqref{eq:NoIntersectionBed} defines an open set in $C_{0,\per,\e}^{0,\alpha}(\R)$, the combination of \eqref{eq:MonotoneCrestTrough} and \eqref{eq:MonotoneCrestTroughSecondOpen} an open set in $C_{0,\per,\e}^{2,\alpha}(\R)$, and the combination of \eqref{eq:NoSelfIntersection} and \eqref{eq:NoSelfIntersectionSecondOpen} an open set in $C_{0,\per,\e}^{1,\alpha}(\R)$.

Another important quantity we introduce is, for any $(\lambda,q,w,\phi)\in\Cc_+$,
\begin{align}\label{eq:def_f}
	f\coloneqq\frac{V_x(\phi_y+\psi^\lambda_y)-V_y\phi_x}{V_x^2+V_y^2}=-\psi_X\circ H,
\end{align}
to which we will apply later sharp maximum principles; here, $\psi=(\phi+\psi^\lambda)\circ H^{-1}$ is the stream function in the physical domain. In particular, we will consider the property
\begin{align}\label{eq:f_pos}
	f>0\text{ on }(0,L/2)\times(-h,0].
\end{align}

Let us now first prove that \eqref{eq:NoIntersectionBed}--\eqref{eq:NoSelfIntersectionSecondOpen} and \eqref{eq:f_pos} hold for solutions that are close to $(\lambda_0,0,0,0)$.
\begin{lemma}\label{lma:LocalNodal}
	There exists a neighborhood $\mathcal N$ of $(\lambda_0,0,0,0)$ in $\R\times X$ such that \eqref{eq:NoIntersectionBed}--\eqref{eq:NoSelfIntersectionSecondOpen} and \eqref{eq:f_pos} hold for all $(\lambda,q,w,\phi)\in\Cc_+\cap\mathcal N$.
\end{lemma}
\begin{proof}
	According to Theorem \ref{thm:LocalBifurcation}, the map $(-\varepsilon,\varepsilon)\ni s\mapsto(\lambda^s,q^s,w^s,\phi^s)\in\R\times X$ is continuous. Since \eqref{eq:NoIntersectionBed}, \eqref{eq:NoSelfIntersection}, and \eqref{eq:NoSelfIntersectionSecondOpen} hold for $w=w^0\equiv0$, they also hold for $w^s$, $s\in(-\varepsilon,\varepsilon)$, after possibly shrinking $\varepsilon$. Moreover, by Theorem \ref{thm:LocalBifurcation} the map
	\[s\mapsto\begin{cases}w^s/s,&s\ne0,\\-\cos(\nu\cdot)/\lambda_0,&s=0,\end{cases}\]
	is continuous at $s=0$ from $(-\varepsilon,\varepsilon)$ to $C_{0,\per,\e}^{2,\alpha}(\R)$. Since furthermore $w'<0$ on $(0,L/2)$, $w''(0)<0$, and $w''(L/2)>0$ for $w=\cos(\nu\cdot)$, we conclude that \eqref{eq:MonotoneCrestTrough} and \eqref{eq:MonotoneCrestTroughSecondOpen} hold for $w^s$, $s\in(0,\varepsilon)$, after again possibly shrinking $\varepsilon$. This, together with the uniqueness property in Theorem \ref{thm:LocalBifurcation}, yields the claim for \eqref{eq:NoIntersectionBed}--\eqref{eq:NoSelfIntersectionSecondOpen}.
	
	As for \eqref{eq:f_pos}, we first notice that, according to Theorem \ref{thm:LocalBifurcation} and \eqref{eq:V_explicit},
	\[V^s\coloneqq V[w^s+h]=y+h-\frac{s\sinh(\nu(y+h))}{\lambda_0\sinh(\nu h)}\cos(\nu x)+o(s).\]
	Therefore, also recalling the expansion of $\phi^s$ in Theorem \ref{thm:LocalBifurcation},
	\begin{align}\label{eq:f_firstorder}
		&V_x^s(\phi_y^s+\psi^{\lambda^s}_y)-V_y^s\phi_x^s\\
		&=\frac{s\sinh(\nu(y+h))}{\lambda_0\sinh(\nu h)}\nu\sin(\nu x)\psi^{\lambda_0}_y(y)\nonumber\\
		&\phantom{=\;}+s\left(\beta^{-\nu^2,\lambda_0}(y)-\frac{\psi_y^{\lambda_0}(y)\sinh(\nu(y+h))}{\lambda_0\sinh(\nu h)}\right)\nu\sin(\nu x)+o(s)\nonumber\\
		&=s\nu\beta^{-\nu^2,\lambda_0}(y)\sin(\nu x)+o(s).\nonumber
	\end{align}
	Following the same argument as in the first part of this proof and additionally using \eqref{ass:Pruefergraph}, $\frac{d}{dx}\sin(\nu x)\ne0$ for $x\in\{0,L/2\}$ as well as $\beta^{-\nu^2,\lambda_0}_y(-h)>0$, we infer that
	\[V_x^s(\phi_y^s+\psi^{\lambda^s}_y)-V_y^s\phi_x^s>0\text{ on }(0,L/2)\times(-h,0]\]
	for $s\in(0,\varepsilon)$, after again possibly shrinking $\varepsilon$. This yields finally the claim also for \eqref{eq:f_pos}.
\end{proof}
Let us remark, importantly, that for the validity of \eqref{eq:f_pos} in Lemma \ref{lma:LocalNodal}, \eqref{ass:Pruefergraph} is sufficient \textit{and} necessary in view of \eqref{eq:f_firstorder}. In particular, there is no hope to carry out the following nodal analysis building upon $f$ and trying to continue \eqref{eq:f_pos} along $\Cc_+$ without assuming \eqref{ass:Pruefergraph} in the first place.

We now show that we can extend the nodal properties of Lemma \ref{lma:LocalNodal} along $\Cc_+$.\begin{lemma}\label{lma:nodal_alternatives_end}
	One of the following alternatives occurs:
	\begin{enumerate}[label=(\alph*)]
\item \eqref{eq:NoIntersectionBed}--\eqref{eq:NoSelfIntersectionSecondOpen} and \eqref{eq:f_pos} hold for all $(\lambda,q,w,\phi)\in\Cc_+$; 
		\item there exists $(\lambda,q,w,\phi)\in\Cc_+$ such that \eqref{eq:NoIntersectionBed}, \eqref{eq:MonotoneCrestTrough}, \eqref{eq:MonotoneCrestTroughSecondOpen}, and \eqref{eq:NoSelfIntersectionSecondOpen} are satisfied, and (instead of \eqref{eq:NoSelfIntersection})
		\begin{subequations}\label{eq:SelfIntersection1}
		\begin{gather}
				0<x+(\Ch w)(x)\le L/2\quad\text{for all }x\in(0,L/2),\\
				x_0+(\Ch w)(x_0)=L/2\quad\text{for some }x_0\in(0,L/2),
		\end{gather}
		\end{subequations}
that is, self-intersection of the free surface occurs exactly above a trough.
	\end{enumerate}
\end{lemma}
\begin{proof}
By definition, $\Cc_+\subset\Om$ and in particular $w\in C_{0,\per,\e}^{2,\alpha}(\R)$ for all $(\lambda,q,w,\phi)\in\Cc_+$. Notice also that, since $\sgn\psi^{\lambda_0}_y(0)=\sgn\lambda_0=-1$ and $\SL\phi_y+\lambda\ne 0$, we have 
	\begin{align}\label{eq:sign_normalder_psi}
		\sgn(\SL\phi_y+\lambda)=-1\quad\text{for all }(\lambda,q,w,\phi)\in\Cc_{\pm}.
	\end{align}
	Suppose now that (a) fails to hold. Then there exists $(\lambda,q,w,\phi)\in\Cc_+$ such that one of the properties in \eqref{eq:NoIntersectionBed}--\eqref{eq:NoSelfIntersectionSecondOpen}, \eqref{eq:f_pos} is violated, but all of them hold along a sequence $(\lambda_n,q_n,w_n,\phi_n)\in\Cc_+$ converging to $(\lambda,q,w,\phi)$ in $\R\times X$ according to Lemma \ref{lma:LocalNodal}. Due to Lemma \ref{lma:M_prop}, we have in fact $w_n\to w$ in $C_\per^{2,\alpha}(\R)$. In particular, $w$ satisfies
	\begin{gather}
		w(x)\ge-h\quad\text{for all }x\in\R,\nonumber\\
w'(x)\le0\quad\text{for all }x\in(0,L/2),\label{eq:MonotoneCrestTroughLimit}\\
		w\in C_{0,\per,\e}^{2,\alpha}(\R)\quad\text{and}\quad w''(0)\le0,\;\quad w''(L/2)\ge0,\label{eq:MonotoneCrestTroughSecondOpenLimit}\\
		0\le x+(\Ch w)(x)\le L/2\quad\text{for all }x\in(0,L/2),\label{eq:NoSelfIntersectionLimit}\\
		1+(\Ch w')(0)\ge0\quad\text{and}\quad 1+(\Ch w')(L/2)\ge0.\label{eq:NoSelfIntersectionSecondOpenLimit}
	\end{gather}
	Let now $U$ and $V$ correspond to this $w$ (that is, $V=V[w+h]$, $U$ a harmonic conjugate of $-V$; notice that we do not claim at this point that $U+iV$ is conformal), and recall that $\SL U_x=1+\Ch w'$. First suppose that $w(x)=-h$ for some $x\in\R$. Then $w(L/2)=-h$ by \eqref{eq:MonotoneCrestTroughLimit} and evenness. Therefore, the harmonic function $V$ has a global minimum at $(L/2,0)$ and hence $V_y(L/2,0)<0$ by the Hopf boundary-point lemma. But $V_y(L/2,0)=U_x(L/2,0)\ge0$ by the Cauchy--Riemann equations and \eqref{eq:NoSelfIntersectionSecondOpenLimit}, so a contradiction arises. Thus, \eqref{eq:NoIntersectionBed} holds. 
	
	Next, by evenness we have $V_x(0,0)=0$ and therefore $V_y(0,0)\ne0$ because of $\K(w)>0$. By the Cauchy--Riemann equations it holds that $V_y(0,0)=U_x(0,0)=1+(\Ch w')(0)\ge0$. Combining these two results yields $1+(\Ch w')(0)>0$ altogether. The same reasoning (together with periodicity) at the point $(L/2,0)$ instead of $(0,0)$ yields also $1+(\Ch w')(L/2)>0$. Therefore, \eqref{eq:NoSelfIntersectionSecondOpen} holds.
	
	Let us now consider \eqref{eq:MonotoneCrestTroughLimit} and \eqref{eq:MonotoneCrestTroughSecondOpenLimit}. To this end, we consider
	\[f_n=\frac{(V_n)_x((\phi_n)_y+\psi^{\lambda_n}_y)-(V_n)_y(\phi_n)_x}{(V_n)_x^2+(V_n)_y^2},\]
	which satisfies \eqref{eq:f_pos} and
	\[\Delta f_n=-\gamma'(\phi_n+\psi^{\lambda_n})f_n|\nabla V_n|^2.\]
	Using elliptic regularity and recalling that solutions in $\Om$ are of class $C^{3,\alpha}$ by Proposition \ref{prop:3alpha_regularity}, we obtain that $f_n$ converges in $C_{\per}^{1,\beta}(\overline{\Omega_h})$ to some $f\in C_{\per}^{2,\alpha}(\overline{\Omega_h})$, which satisfies
	\begin{align}\label{eq:Delta_f}
		\Delta f=-\gamma'(\phi+\psi^\lambda)f|\nabla V|^2
	\end{align}
	and
	\[f=\frac{V_x(\phi_y+\psi^\lambda_y)-V_y\phi_x}{V_x^2+V_y^2},\]
	at least close to $\R\times\{0\}$, in fact everywhere whenever \eqref{eq:NoSelfIntersection} holds.

	Let us split the boundary of $\mathcal R\coloneqq(0,L/2)\times(-h,0)$ into four parts:
	\begin{gather*}
		\partial\mathcal R=\partial\mathcal R_l\cup\partial\mathcal R_r\cup\partial\mathcal R_t\cup\partial\mathcal R_b,\\
		\partial\mathcal R_l\coloneqq\{0\}\times[-h,0],\quad\partial\mathcal R_r\coloneqq\{L/2\}\times[-h,0],\\
		\partial\mathcal R_t\coloneqq[0,L/2]\times\{0\},\quad\partial\mathcal R_b\coloneqq[0,L/2]\times\{-h\}.
	\end{gather*}
	Notice first that $f\ge0$ on $\overline{\mathcal R}$ since $f_n$ has this property. Clearly, $f_n=0$ and thus $f=0$ on $\partial\mathcal R_l\cup\partial\mathcal R_r\cup\partial\mathcal R_b$ by evenness and periodicity or the bottom boundary condition for $V_n$ and $\phi_n$. Moreover,
	\[f=\frac{V_x(\phi_y+\psi^\lambda_y)}{\K(w)^2}=-\frac{w'\sqrt{2q+\lambda^2-2gw}}{\K(w)}\quad\text{on }\partial\mathcal R_t
	\]by \eqref{eq:OriginalBernoulliFlat} and \eqref{eq:sign_normalder_psi}; in particular, $f\not\equiv0$ as $w\not\equiv0$. Therefore, $f\ge0$ on $\partial\mathcal R_t$ in view of \eqref{eq:MonotoneCrestTroughLimit}. By \eqref{eq:Delta_f} and the strong maximum principle, it follows that $f>0$ on $\mathcal R$; in particular, \eqref{eq:f_pos} holds whenever all other properties in alternative (a) are valid.
	
	The strategy to prove \eqref{eq:MonotoneCrestTrough} and \eqref{eq:MonotoneCrestTroughSecondOpen} is now almost exactly the same as in \cite[Proof of Lemma 15]{ConsStrVarv16}: On the one hand, if $w'(x_0)=0$ for some $x_0\in(0,L/2)$, then $f$ obtains its minimum on $\overline{\mathcal R}$ in the point $(x_0,0)$. However, a computation shows that $f_y(x_0,0)=0$, yielding a contradiction in view of the Hopf boundary-point lemma. On the other hand, if $w''(x_0)=0$ for $x_0=0$ or $x_0=L/2$, one can compute that all first and second order derivatives of $f$ vanish at $(x_0,0)$, contradicting, in view of the Serrin corner-point lemma, the fact that $f$ obtains its minimum on $\overline{\mathcal R}$ in the point $(x_0,0)$. Indeed, the only slight difference in this computation compared to \cite{ConsStrVarv16} is that $f_{xx}(x_0,0)=0$ follows from $f(x_0,0)=f_{yy}(x_0,0)=0$ and, here, \eqref{eq:Delta_f}.

	Finally, let us now turn to \eqref{eq:NoSelfIntersectionLimit} and suppose that there exists $x_0\in(0,L/2)$ such that $U(x_0,0)=x_0+(\Ch w)(x_0)=0$. Then the harmonic function $U$ attains its minimum on $\overline{\mathcal R}$ in the point $(x_0,0)$. By the Hopf boundary-point lemma and the Cauchy--Riemann equations, it follows that $w'(x_0)=V_x(x_0,0)=-U_y(x_0,0)>0$. This contradicts \eqref{eq:MonotoneCrestTroughLimit}.
	
	To sum up, we have proved that, if alternative (a) fails to hold, then alternative (b) has to be valid. The proof is complete.
\end{proof}

\subsubsection*{Characterization of looping back to a trivial solution}

We now have a look at the possibility that the solution set loops back to a solution with $w\equiv0$, but no other alternative in Theorem \ref{thm:GlobalBifurcation} holds (in particular, $\Cc\subset\Om$ and alternative (a) in Lemma \ref{lma:nodal_alternatives_end} occurs). For simplicity we do not strive to handle the most general case in the following, but rather assume \eqref{eq:sup_gamma'_DirEV}, and explain only later in Remark \ref{rem:loop_back} which problems might occur in case \eqref{eq:sup_gamma'_DirEV} does not hold. By means of \eqref{eq:sup_gamma'_DirEV}, we have precisely alternative (ii) in Theorem \ref{thm:GlobalBifurcation} according to Remark \ref{rem:TrivialSolutions} and do not run into a somewhat unpleasant point of the form $(\lambda,q,0,\phi)$ with $\phi\not\equiv0$. In particular, also by \eqref{eq:sup_gamma'_DirEV}, we know that \eqref{ass:SL-spectrum} holds for any $\lambda\in\R$ and that any corresponding kernel is at most one-dimensional by Proposition \ref{prop:sol_disprel_bound_gamma'}. 

Therefore, let us now assume that there is a sequence $(\lambda_n,q_n,w_n,\phi_n)$ in $\Cc_+$ or $\Cc_-$ converging to some $(\bar\lambda,0,0,0)$. Then, $\bar\lambda$ has the same sign as $\lambda_0$, for otherwise there would have to exist a point $(\lambda,q,w,\phi)\in\Cc_+\cup\Cc_-$ with $\min_\R|\SL\phi_y+\lambda|=0$, leading to the contradiction $\Cc\cap\partial\Om\ne\emptyset$. Let us suppose that $\dim\ker F_{(q,w,\phi)}(\bar\lambda,0,0,0)=0$. Then $F_{(q,w,\phi)}(\bar\lambda,0,0,0)$ is an isomorphism by the open mapping theorem and since its index is zero (identity plus compact). The implicit function theorem implies that the only solutions in a neighborhood of $(\bar\lambda,0,0,0)$ are trivial solutions, yielding a contradiction. Thus, $\dim\ker F_{(q,w,\phi)}(\bar\lambda,0,0,0)=1$. Lemma \ref{lma:kernel} yields that $d(-(\bar k\nu)^2,\bar\lambda)=0$ for exactly one $\bar k\in\N$.

Now let us define
\[(p_n,v_n,\varphi_n)\coloneqq\frac{(q_n,w_n,\phi_n)}{\|(q_n,w_n,\phi_n)\|_X}\]
for $n\in\N$. Since, by Lemma \ref{lma:M_prop}, $\M\colon\Om\to Y$ is locally Lipschitz continuous, where
\begin{align*}
	Y&\coloneqq\R\times C_{0,\per,\e}^{2,\alpha}(\R)\times C_\per^{2,\alpha}(\overline{\Omega_h}),\\
	\|(q,w,\phi)\|_Y&\coloneqq|q|+\|w\|_{C_\per^{2,\alpha}(\R)}+\|\phi\|_{C_\per^{2,\alpha}(\overline{\Omega_h})},
\end{align*}
we find that
\begin{align*}
	\|(p_n,v_n,\varphi_n)\|_Y&=\frac{\|(q_n,w_n,\phi_n)\|_Y}{\|(q_n,w_n,\phi_n)\|_X}=\frac{\|\M(\lambda_n,q_n,w_n,\phi_n)-\M(\lambda_n,0,0,0)\|_Y}{\|(q_n,w_n,\phi_n)\|_X}\\
	&\le C
\end{align*}
for some constant $C>0$ independent of $n$. Therefore, $(p_n,v_n,\varphi_n)$ converges in $X$, after extracting a suitable subsequence, to some $(p,v,\varphi)$. Since $\|(p_n,v_n,\varphi_n)\|_X=1$, $n\in\N$, we have in particular $(p,v,\varphi)\ne(0,0,0)$. By local Lipschitz continuity of $F_{(q,w,\phi)}$ and $F(\lambda_n,q_n,w_n,\phi_n)=F(\lambda_n,0,0,0)=0$, $n\in\N$, it follows that
\begin{align*}
	&\|F_{(q,w,\phi)}(\lambda_n,0,0,0)(p_n,v_n,\varphi_n)\|_X\\
	&=\frac{\|F(\lambda_n,q_n,w_n,\phi_n)-F(\lambda_n,0,0,0)-F_{(q,w,\phi)}(\lambda_n,0,0,0)(q_n,w_n,\phi_n)\|_X}{\|(q_n,w_n,\phi_n)\|_X}\\
	&\le C\|(q_n,w_n,\phi_n)\|_X,
\end{align*}
where $C>0$ is some constant independent of $n$, and thus
\[F_{(q,w,\phi)}(\bar\lambda,0,0,0)(p,v,\varphi)=\lim_{n\to\infty}F_{(q,w,\phi)}(\lambda_n,0,0,0)(p_n,v_n,\varphi_n)=0.\]
Therefore, $(p,v,\varphi)$ is a nonzero element of $\ker F_{(q,w,\phi)}(\bar\lambda,0,0,0)$, and Lemma \ref{lma:kernel} yields that $v=c\cos(\bar k\nu\cdot)$ for some $c\in\R\setminus\{0\}$. But because of $v_n=w_n/\|(q_n,w_n,\phi_n)\|_X$, any $v_n$ has period $L$ and $v_n'\ne0$ on $(0,L/2)$ by Lemma \ref{lma:nodal_alternatives_end}. Therefore, necessarily $2\pi/(\bar k\nu)=L$, that is, $\bar k=1$. In particular, we get
\[d(-\nu^2,\lambda_0)=d(-\nu^2,\bar\lambda)=0.\]
This means that $\lambda\mapsto d(-\nu^2,\lambda)$ needs to have at least two zeros on $(0,\infty)$ if $\lambda_0>0$ (on $(-\infty,0)$ if $\lambda_0<0$). This is a purely algebraic property and to check whether it is possible only requires the computation of the trivial solutions for $\lambda>0$ if $\lambda_0>0$ (or $\lambda<0$ if $\lambda_0<0$). In fact, it can never hold if $\sup \gamma'< \pi^2/(4h^2)$ and $\lambda_0\gamma''\ge0$, and thus alternative (ii) can be left out completely in this case. Indeed, it is part of the argument leading to Proposition \ref{prop:transversality_redundant}, especially \eqref{eq:d_lambda_sign}, that for any $\lambda>0$ we have $d_\lambda(-\nu^2,\lambda)>0$ provided $d(-\nu^2,\lambda)=0$. Therefore, the map $\lambda\mapsto d(-\nu^2,\lambda)$ can clearly have at most one zero on $(0,\infty)$; similarly, the statement also holds for $(-\infty,0)$ instead.

\begin{remark}\label{rem:loop_back}
	Let us now have a look at the case when \eqref{eq:sup_gamma'_DirEV} does not hold. Then a sequence $(\lambda_n,q_n,w_n,\phi_n)$ in $\Cc_+$ or $\Cc_-$ as above can possibly also converge to some $(\bar\lambda,\bar q,0,\bar\phi)$ with $\bar\phi\not\equiv0$. Considering the corresponding $f_n$ and $\bar f$, defined in the obvious way by means of \eqref{eq:def_f}, we have $\bar f=0$ on $\partial\mathcal R$, $\Delta\bar f=-\gamma'(\bar\phi+\psi^{\bar\lambda})\bar f$ and, by Lemma \ref{lma:nodal_alternatives_end}, $\bar f=-\bar\phi_x\ge0$ on $\overline{\mathcal R}$. Therefore, $\bar\phi_x\equiv0$. Indeed, if not, then $\bar\phi_{xy}\ne0$ on $\partial\mathcal R_t$ by the Hopf boundary point lemma, contradicting the Bernoulli equation \eqref{eq:OriginalEquations_Bernoulli}, which demands here that $\bar\phi_y$ be independent of $x$ on the surface. Hence, $\bar\phi=\bar\phi(y)$ so that $\bar\phi+\psi^{\bar\lambda}=\psi^{\hat\lambda}$ with $\hat\lambda\coloneqq\bar\phi_y(0)+\bar\lambda$, and therefore $\bar q=(\hat\lambda^2-\bar\lambda^2)/2$ in view of \eqref{eq:def_q}. By doing the same computations as in Appendix \ref{appx:compder}, we see that $F_{(q,w,\phi)}(\bar\lambda,\bar q,0,\bar\phi)=F_{(q,w,\phi)}(\hat\lambda,0,0,0)$. At this point it seems natural to just mimic the previous strategy and construct a kernel element $(p,v,\varphi)$ of $F_{(q,w,\phi)}(\bar\lambda,\bar q,0,\bar\phi)$. However, for this it was important that $F_\lambda$ vanishes. Here, there is just no reason why this should be true at $(\bar\lambda,\bar q,0,\bar\phi)$: in fact, for this it would be necessary that $\A_\lambda=\A_\lambda(\bar\lambda,0,\bar\phi)\equiv0$, where $\A_\lambda$
	solves
	\[\Delta\A_\lambda=\left(\gamma'(\psi^{\bar\lambda})-\gamma'(\psi^{\hat\lambda})\right)\partial_\lambda\psi^\lambda,\]
	and the right-hand side does in general not vanish since $\bar\lambda\ne\hat\lambda$. Therefore, we restricted ourselves before to the case \eqref{eq:sup_gamma'_DirEV} and leave as an open problem what may happen else.
\end{remark}
\subsubsection*{Summary}
Let us now summarize all the results of this section so far.
\begin{theorem}\label{thm:NodalProperties}
	Assume \eqref{ass:Pruefergraph} and without loss of generality $k_0=1$. The nodal properties \eqref{eq:NoIntersectionBed}--\eqref{eq:NoSelfIntersectionSecondOpen} (together with \eqref{eq:f_pos}) on $\Cc_\pm$ (if $\mp\sgn\lambda_0=1$) can be continued as long as no configuration with a flat surface is reached and no self-intersection of the free surface exactly above a trough occurs.
	
	Moreover, in case \eqref{eq:sup_gamma'_DirEV} (which implies \eqref{ass:Pruefergraph}), in Theorem \ref{thm:GlobalBifurcation} we can
	\begin{itemize}[leftmargin=*]
		\item replace (ii) by
		\begin{enumerate}
			\item[(ii$'$)] there exists $\bar\lambda\ne0$ with $\bar\lambda\ne\lambda_0$, $(\bar\lambda,0,0,0)\in\Cc\cap(\overline{\Cc_+}\cup\overline{\Cc_-})$, $\sgn\bar\lambda=\sgn\lambda_0$, and $d(-\nu^2,\lambda_0)=d(-\nu^2,\bar\lambda)=0$ \end{enumerate}
		and leave out (ii) completely in case
		\begin{align}
		\label{eq:nodal cond}
		\sup \gamma'< \frac{\pi^2}{4h^2}\quad \text{and} \quad \lambda_0\gamma''\ge0;
		\end{align}
\item replace (v) by
		\begin{enumerate}
			\item[(v$'_\pm$)] there exists $(\lambda_\pm,q_\pm,w_\pm,\phi_\pm)\in\Cc_\pm$ such that \eqref{eq:NoIntersectionBed}, \eqref{eq:MonotoneCrestTrough}, \eqref{eq:MonotoneCrestTroughSecondOpen}, \eqref{eq:NoSelfIntersectionSecondOpen}, and \eqref{eq:SelfIntersection1} are satisfied for the point $(\lambda_\pm,q_\pm,w_\pm,\phi_\pm)$ (provided that $\mp\sgn\lambda_0=1$); in particular, self-intersection of the free surface occurs exactly above a trough;
		\end{enumerate}
		\item leave out (vi).
	\end{itemize}
Also, if $\Cc\cap\partial\Om=\emptyset$ (in particular, if (iii) and (iv) do not hold) and (ii$'$) is not satisfied, then $\Cc=\Cc_+\cup\Cc_-\cup\{(\lambda_0,0,0,0)\}$ disjointly.
	
	(In case $\mp\sgn\lambda_0=-1$, the inequalities in \eqref{eq:MonotoneCrestTrough} and \eqref{eq:MonotoneCrestTroughSecondOpen} always have to be reversed, and \eqref{eq:SelfIntersection1} has to read $0\le x+(\Ch w)(x)<L/2$ for all $x\in(0,L/2)$, $x_0+(\Ch w)(x_0)=0$ for some $x_0\in(0,L/2)$.)
\end{theorem}

\subsection{On downstream waves}\label{sec:downstream}
Let us now have a look at another special case, namely, downstream waves. This case appears provided $\lambda_0>0$ and $\gamma\ge0$ (or $\lambda_0<0$ and $\gamma\le0$). In fact, here we can say more about the unboundedness alternative in Theorem \ref{thm:GlobalBifurcation}.
\begin{proposition}In Theorem \ref{thm:GlobalBifurcation} the alternative (i)(b) can be left out completely provided $\lambda_0>0$ and $\gamma\ge0$ (or $\lambda_0<0$ and $\gamma\le0$) [in particular, in the irrotational case $\gamma\equiv0$ for any $\lambda_0\ne0$].
\end{proposition}
\begin{proof}
	Let us assume that $\lambda_0>0$ and $\gamma\ge0$; in the other case $\lambda_0<0$ and $\gamma\le0$ one can proceed completely analogously. Let $(\lambda,q,w,\phi)\in\Cc\cap\Om$, and first notice that $\SL|\phi_y+\lambda|>0$ and therefore we may assume that $\SL\phi_y+\lambda>0$ because of $\lambda_0>0$; indeed, in case $\inf_{(\lambda,q,w,\phi)\in\Cc}\min_\R|\SL\phi_y+\lambda|=0$, necessarily alternative (iii) or (iv) holds true recalling the discussion after \eqref{eq:BoundaryAlternative} in the proof of Theorem \ref{thm:GlobalBifurcation}. Let us now consider the function
	\[\zeta(x,y)\coloneqq\phi(x,y)+\psi^\lambda(y)-ym(\lambda)/h,\quad(x,y)\in\overline{\Omega_h},\]
	which solves
	\begin{align*}
		\Delta\zeta&=-\gamma(\phi+\psi^\lambda)|\nabla V|^2\le0&\text{in }\Omega_h,\\
		\zeta&=0&\text{on }y=0,\\
		\zeta&=0&\text{on }y=-h.
	\end{align*}
	By the weak maximum principle, $\zeta\ge0$ and hence in particular
	\[0\ge\SL\zeta_y=\SL\phi_y+\lambda-\frac{m(\lambda)}{h}.\]
	Therefore, we have
	\[0<\SL\phi_y+\lambda\le\frac{m(\lambda)}{h}\]
	and thus
	\[|\SL\phi_y+\lambda|\le\frac{m(\lambda)}{h}.\]
	Looking at the Bernoulli equation \eqref{eq:OriginalBernoulliFlat} we conclude that
	\[|w'|\le\K(w)\le\frac{m(\lambda)}{2h\sqrt{q+\lambda^2/2-gw}}.\]
	Thus, if $\lambda$ remains bounded and $q+\lambda^2/2-gw$ stays away from zero, it follows that the $C_b^1$-norm (and therefore all $C^{0,\delta}$-norms, $\delta\in(0,1]$) of $w$ remains bounded since $w$ has zero average. This proves the claim.
\end{proof}
Next, our goal is to show that, in the situation of a downstream wave, the flow is unidirectional and the waves cannot overhang. More precisely, while assuming that the conditions for local bifurcation are met, \enquote{downstream waves} correspond in the case $\lambda_0>0$ and $\gamma\ge0$ to points on $\Cc_0$, which is defined to be the connected component of $\{(\lambda,q,w,\phi)\in\Cc\cap\Om:\eqref{eq:intersection_with_bed}\text{ holds}\}$ to which $(\lambda_0,0,0,0)$ belongs. Notice that we have to assume that the free surface does not intersect the flat bed, but we can get rid of this assumption provided additionally \eqref{eq:sup_gamma'_DirEV} holds in view of the previous section. Very similar things can also be done in case $\lambda_0<0$ and $\gamma\le0$, but for simplicity we only consider the first case. 

We follow the ideas of \cite{ConsStrVarv21}, where similar results in the case of constant vorticity were obtained. Moreover, we like to mention \cite{KozlovLokharu20,StraussWheeler16} at this point, where bounds for the slope of the surface for unidirectional flows are established.

Let us now consider the properties
\begin{gather}
	\frac{V_x\phi_x+V_y(\phi_y+\psi^\lambda_y)}{V_x^2+V_y^2}>0\quad\text{on }\overline{\Omega_h},\label{eq:unidirectional}\\
	1+\Ch w'>0\quad\text{on }\R,\label{eq:nooverhang}
\end{gather}
for $(\lambda,q,w,\phi)\in\Cc_0$. Notice that \eqref{eq:nooverhang} implies that the corresponding free surface does not overhang and thus in particular $|\nabla V|^2\ne0$ on $\overline{\Omega_h}$ by Lemma \ref{lma:ConformalMapping}(iv). Therefore, the left-hand side of \eqref{eq:unidirectional} is well-defined and equals the horizontal component of the velocity in the physical domain, which in turn means that \eqref{eq:unidirectional} yields unidirectionality of the corresponding flow. Since $\phi=\A(\lambda,w,\phi)$ along $\Cc_0$, we notice furthermore that \eqref{eq:unidirectional} and \eqref{eq:nooverhang} define open sets in $\R\times X$ and are obviously satisfied for $(\lambda,q,w,\phi)=(\lambda_0,0,0,0)$ in view of $\psi^{\lambda_0}_y(0)=\lambda_0>0$ and $\psi^{\lambda_0}_{yy}=-\gamma(\psi^{\lambda_0})\le0$. Let us now suppose on the contrary that \eqref{eq:unidirectional}, \eqref{eq:nooverhang} are not satisfied for all points on $\Cc_0$. Then there exists $(\lambda,q,w,\phi)\in\Cc_0$ such that
\begin{gather*}
	\frac{V_x\phi_x+V_y(\phi_y+\psi^\lambda_y)}{V_x^2+V_y^2}\ge0\quad\text{on }\overline{\Omega_h},\\
	1+\Ch w'\ge0\quad\text{on }\R,
\end{gather*}
but \eqref{eq:unidirectional} or \eqref{eq:nooverhang} is violated. Because of $\K(w)\ne 0$ on $\R$, still in this case $x\mapsto(x+(\Ch w)(x),w(x)+h)$ is injective on $\R$ and Lemma \ref{lma:ConformalMapping}(iv) yields that $|\nabla V|^2\ne0$ on $\overline{\Omega_h}$, that $H=U+iV\colon\Omega_h\to\Omega_w$ is conformal, and that the free surface is of class $C^{1,\alpha}$; here and in the following, we frequently make use of $(\lambda,q,w,\phi)\in\Cc_0\subset\Om$. Moreover, $w\in C^{3,\alpha}(\R)$ and $\phi,V\in C^{3,\alpha}(\overline{\Omega_h})$ by Proposition \ref{prop:3alpha_regularity}, and thus $\psi=(\phi+\psi^\lambda)\circ H^{-1}\in C^{3,\alpha}(\overline{\Omega_w})$ due to $|\nabla V|^2\ne0$ on $\overline{\Omega_h}$ (so the surface is in fact of class $C^{3,\alpha}$). Then, \[\psi_Y=\frac{V_x\phi_x+V_y(\phi_y+\psi^\lambda_y)}{V_x^2+V_y^2}\circ H^{-1}\] satisfies
\[\Delta\psi_Y+\gamma'(\psi)\psi_Y=0\quad\text{in }\Omega_w,\qquad\psi_Y\ge0\quad\text{on }\overline{\Omega_w}.\]
Let us now suppose that \eqref{eq:nooverhang} is violated. Then clearly $\psi_Y$ vanishes somewhere at the surface in view of $\SL\phi_x=0$ and $\SL V_y=1+\Ch w'$. In particular, \eqref{eq:unidirectional} is violated as well. Therefore, it suffices to consider the case that \eqref{eq:unidirectional} is violated, that is, $\min_{\overline{\Omega_w}}\psi_Y=0$. By the strong maximum principle it follows that $\psi_Y$ admits its minimum, which is $0$, on the surface or on the bottom and that $\psi_Y>0$ in $\Omega_w$, noticing that $\psi_Y\not\equiv0$. Since $\psi$ is constant on the bottom, we have $\psi_{YY}=-\psi_{XX}-\gamma(\psi)\le0$ there. Thus, $\psi_Y$ cannot admit its minimum on the bottom in view of the Hopf boundary-point lemma. Therefore, $\psi_Y$ vanishes at some $(U(x_0,0),V(x_0,0))$ on the surface, which implies that $\SL V_y=1+\Ch w'$ vanishes at $x_0$ in view of $\SL\phi_x=0$ and $\SL(\phi_y+\psi^\lambda_y)\ne0$.

Moreover, the above yields that $|\nabla\psi|^2\ne0$ on $\overline{\Omega_w}$. Based on an idea by Spielvogel \cite{Spielvogel70}, we introduce the function
\[P\coloneqq\frac12|\nabla\psi|^2+g(Y-h)-q-\frac{\lambda^2}{2}+\int_0^\psi\gamma(s)\,ds,\]
which satisfies
\[\Delta P-\frac{2(P_X-\gamma(\psi)\psi_X)}{|\nabla\psi|^2}P_X-\frac{2(P_Y-2g-\gamma(\psi)\psi_Y)}{|\nabla\psi|^2}P_Y=\frac{g(g+\gamma(\psi)\psi_Y)}{|\nabla\psi|^2}\ge0\]
in $\Omega_w$, $P_Y=g>0$ on the bottom, and $P=0$ on the surface. By the strong maximum principle it follows that $P<0$ in $\Omega_w$ and on the bottom, and that the normal derivative of $P$ at the surface is positive. The argument that leads to a contradiction is now exactly the same as in \cite[Proof of Theorem 3]{ConsStrVarv21}.

Therefore, to conclude, the following is proved:
\begin{theorem}
	Suppose that the assumptions of Theorem \ref{thm:LocalBifurcation} are satisfied, and that $\lambda_0>0$ and $\gamma\ge0$ (or $\lambda_0<0$ and $\gamma\le0$). Then, for each $(\lambda,q,w,\phi)\in\Cc_0$,
	\begin{enumerate}[label=(\alph*)]
		\item the corresponding flow is unidirectional, that is, $\psi_Y>0$ (or $\psi_Y<0$) in the closure of the physical domain,
		\item the corresponding surfaces do not overhang; more precisely, $1+\Ch w'>0$ on $\R$.
	\end{enumerate}
\end{theorem}

\appendix

\section{Local bifurcation}\label{appx:LocalBifurcation}
Throughout Appendix \ref{appx:LocalBifurcation} we assume that the vorticity function, additionally to being globally Lipschitz, satisfies
\[\gamma\in C_{\text{loc}}^{2,1}(\R).\]
\subsection{Application of Crandall--Rabinowitz}\label{appx:LocalBifurcation:CR}
The goal of this section is to apply the following local bifurcation theorem of Crandall--Rabinowitz \cite[Theorem I.5.1]{Kielhoefer} in the framework of our new reformulation and consequently prove Theorem \ref{thm:LocalBifurcation}.
\begin{theorem}\label{thm:CrandallRabinowitz}
	Let $X$ be a Banach space, $\mathcal U\subset\R\times X$ open, and $F\colon \mathcal U\to X$ such that $(\lambda,0)\in \mathcal U$ and $F(\lambda,0)=0$ for any $\lambda\in\R\setminus\{0\}$. Assume that there exists $\lambda_0\in\R\setminus\{0\}$ such that $F$ is of class $C^2$ in an open neighborhood of $(\lambda_0,0)$, and suppose that $F_x(\lambda_0,0)$ is a Fredholm operator with index zero and one-dimensional kernel spanned by $x_0\in X$, and that the transversality condition $F_{\lambda x}(\lambda_0,0)x_0\notin\im F_x(\lambda_0,0)$ holds. Then there exists $\varepsilon>0$ and a $C^1$-curve $(-\varepsilon,\varepsilon)\ni s\mapsto(\lambda^s,x^s)$ with $(\lambda^0,x^0)=(\lambda_0,0)$ and $x^s\neq 0$ for $s\neq 0$, and $F(\lambda^s,x^s)=0$. Moreover, all solutions of $F(\lambda,x)=0$ in a neighborhood of $(\lambda_0,0)$ are on this curve or are trivial. Furthermore, the curve admits the asymptotic expansion $x^s=sx_0+o(s)$.
\end{theorem}
Motivated by \eqref{eq:trvial_in_O}, here and also when stating Theorems \ref{thm:Rabinowitz} and \ref{thm:analytic_globbi} below, we restrict ourselves to trivial solutions $(\lambda,0)$ with $\lambda\ne0$ and allow the case $(0,0)\in\partial \mathcal U$, but clearly these theorems remains true with this slight modification compared to what is stated in the respective references.

In order to apply this theorem, we have to study the linearized operator and also the transversality condition. This will be the content of the next subsections.
\subsubsection{Computing derivatives}\label{appx:compder}
First, it is important to calculate the partial derivative $F_{(q,w,\phi)}$ and, in particular, its value at a trivial solution. Let us introduce the abbreviation
\[B=B(\lambda,w,\phi)=|\SL\partial_y\A(\lambda,w,\phi)+\lambda|.\] With this we have
\[R=\frac{B}{\sqrt{2q+\lambda^2-2gw}}.\]
We first evaluate these quantities at a trivial solution $(\lambda,0,0,0)$, $\lambda\ne0$:
\begin{gather*}
	V=y+h,\quad\nabla V=\begin{pmatrix}0\\1\end{pmatrix},\quad\K=1,\quad\A=0,\quad B=|\lambda|,\quad R=1.
\end{gather*}
Let now $(\delta q,\delta w,\delta\phi)$ be a direction. First consider $V_w(0)\delta w$, the partial derivative of $V$ with respect to $w$ evaluated at $w\equiv0$ and applied to the direction $\delta w$, and abbreviate $V_w=V_w(0)\delta w$; we will also use this abbreviation similarly for other expressions and derivatives when there is no possibility of confusion. Since we will only need such partial derivatives evaluated at trivial solutions, we only state the following formulas for the case that the point at which we evaluate them satisfies $q=0$, $w\equiv0$, and $\phi\equiv0$. Now $V_w$ is the unique solution of
\begin{align*}
	\Delta V_w&=0&\text{in }\Omega_h,\\
	V_w&=\delta w&\text{on }y=0,\\
	V_w&=0&\text{on }y=-h.
\end{align*}
Next, $\A_w$ and $\A_\phi$ are the unique solutions of
\begin{align*}
\Delta\A_w&=-2\gamma(\psi^\lambda)\partial_yV_w&\text{in }\Omega_h,\\
		\A_w&=0&\text{on }y=0\text{ and }y=-h,
\end{align*}
	and
	\begin{align*}
\Delta\A_\phi&=-\gamma'(\psi^\lambda)\delta\phi&\text{in }\Omega_h,\\
			\A_\phi&=0&\text{on }y=0\text{ and }y=-h.
\end{align*}
Moreover, it holds that
\begin{gather*}
	\K_w=\Ch\delta w',\quad B_w=\frac{\lambda}{|\lambda|}\SL\partial_y\A_w,\quad B_\phi=\frac{\lambda}{|\lambda|}\SL\partial_y\A_\phi,\quad R_q=-\frac{\delta q}{\lambda^2},\\
	R_w=\frac{B_w}{|\lambda|}+\frac{g\delta w}{\lambda^2},\quad R_\phi=\frac{B_\phi}{|\lambda|},\quad F_q=\left(-\langle R_q\rangle,-\partial_x^{-1}(\Ch)^{-1}\Pro R_q,0\right),\\
	F_w=\left(-\langle R_w\rangle,\delta w-\partial_x^{-1}(\Ch)^{-1}\Pro R_w,-\A_w\right),\\
	F_\phi=\left(-\langle R_\phi\rangle,-\partial_x^{-1}(\Ch)^{-1}\Pro R_\phi,\delta\phi-\A_\phi\right).
\end{gather*}
Putting everything together, we conclude
\begin{subequations}\label{eq:Fder_trivial}
	\begin{align}
		F_q&=\left(\lambda^{-2}\delta q,0,0\right),\\
		F_w&=\left(-\lambda^{-1}\langle\SL\partial_y\A_w\rangle,\delta w-\lambda^{-1}\partial_x^{-1}(\Ch)^{-1}(\Pro\SL\partial_y\A_w+g\lambda^{-1}\delta w),-\A_w\right),\\
		F_\phi&=\left(-\lambda^{-1}\langle\SL\partial_y\A_\phi\rangle,-\lambda^{-1}\partial_x^{-1}(\Ch)^{-1}\Pro\SL\partial_y\A_\phi,\delta\phi-\A_\phi\right).
	\end{align}
\end{subequations}

Let us note that, since $\M$ is compact (on each $\Om_\varepsilon$) due to Lemma \ref{lma:M_prop}, also its derivative evaluated at a fixed point is compact. Thus, at each point, $F_{(q,w,\phi)}=\mathrm{Id}-\M_{(q,w,\phi)}$ is a compact perturbation of the identity and hence a Fredholm operator with index zero.
			
\subsubsection{The good unknown}
Before we characterize the kernel of  $F_{(q,w,\phi)}$ and the transversality condition, we first introduce an isomorphism, which facilitates the computations later and is sometimes called $\T$-isomorphism in the literature (for example, in \cite{EhrnEschWahl11,Varholm20}). The discovery of the importance of such a new variable (here $\theta$) goes back to Alinhac \cite{Alinhac89}, who called it the \enquote{good unknown} in a very general context, and Lannes \cite{Lannes05}, who introduced it in the context of water wave equations.
			
In the following, $V[v]$ denotes the unique ($L$-periodic) solution of the boundary value problem \eqref{eq:BVP_for_V} with Dirichlet data $v$ instead of $w+h$ on $y=0$.
			
\begin{lemma}
Let
\[Y\coloneqq\left\{\theta\in C_{\per,\e}^{0,\alpha}(\overline{\Omega_h})\cap H_\per^1(\Omega_h):\SL\theta\in C_{0,\per,\e}^{1,\alpha}(\R),\theta=0\mathrm{\ on\ }y=-h\right\},\]
equipped with
\[\|\theta\|_Y\coloneqq\|\theta\|_{C_\per^{0,\alpha}(\overline{\Omega_h})\cap H_\per^1(\Omega_h)}+\|\SL\theta\|_{C_\per^{1,\alpha}(\R)},\]
and assume that $\lambda\neq 0$. Then
\[\T(\lambda)\colon Y\to\tilde X,\quad\T(\lambda)\theta=\left(-\frac{\SL\theta}{\lambda},\theta-\frac{\psi_y^\lambda}{\lambda}V[\SL\theta]\right)\]
is an isomorphism. Its inverse is given by
\[[\T(\lambda)]^{-1}(\delta w,\delta\phi)=\delta\phi-\psi_y^\lambda V[\delta w].\]
\end{lemma}
\begin{proof}
Both $\T(\lambda)$ and $[\T(\lambda)]^{-1}$ are well-defined, and a simple computation shows that they are inverse to each other.
\end{proof}
			
Let us now consider a trivial solution $(\lambda,0,0,0)$, $\lambda\ne0$. In view of the isomorphism $\T(\lambda)$, we introduce
\[\LL(\lambda)\coloneqq [F^{2,3}_{(w,\phi)}(\lambda,0,0,0)]\circ [\T(\lambda)]\colon Y\to\tilde X,\]
where $F^{2,3}=(F^2,F^3)$ consists of the second and third component of $F=(F^1,F^2,F^3)$. In order to simplify $\LL(\lambda)\theta=(\LL^1(\lambda),\LL^2(\lambda))\theta$, we notice that
\[\A_w(\SL\theta)+\A_\phi(\psi_y^\lambda V[\SL\theta])=(\psi_y^\lambda-\lambda)V[S\theta].\]
Indeed, the function $f\coloneqq\A_w(\SL\theta)+\A_\phi(\psi_y^\lambda V[\SL\theta])-(\psi_y^\lambda-\lambda)V[S\theta]$ satisfies
\begin{align*}
	&\Delta f=-2\gamma(\psi^\lambda)\partial_yV[\SL\theta]-\gamma'(\psi^\lambda)\psi_y^\lambda V[\SL\theta]-\psi_{yyy}^\lambda V[\SL\theta]-2\psi_{yy}^\lambda\partial_y V[\SL\theta]\\
	&=-2\gamma(\psi^\lambda)\partial_yV[\SL\theta]-\gamma'(\psi^\lambda)\psi_y^\lambda V[\SL\theta]+\gamma'(\psi^\lambda)\psi_y^\lambda V[\SL\theta]+2\gamma(\psi^\lambda)\partial_y V[\SL\theta]\\
	&=0
\end{align*}
in $\Omega_h$ and vanishes at the top and bottom. By \eqref{eq:Fder_trivial} and additionally using $\Ch\SL\theta_x=\SL\partial_y V[\SL\theta]$, we can thus write
\begin{subequations}\label{eq:Ltheta}
	\begin{align}
		\LL^1(\lambda)\theta&=-\lambda^{-1}\SL\theta-\lambda^{-1}\partial_x^{-1}(\Ch)^{-1}\left(\Pro\SL\partial_y\A_\phi\theta+\lambda^{-1}(\gamma(0)-\lambda^{-1}g)\SL\theta\right)\label{eq:L1theta1}\\
		&=-\lambda^{-1}\partial_x^{-1}(\Ch)^{-1}\left(\Pro\SL\partial_y(\A_\phi\theta+V[\SL\theta])+\lambda^{-1}(\gamma(0)-\lambda^{-1}g)\SL\theta\right),\\
		\LL^2(\lambda)\theta&=\theta-(\A_\phi\theta+V[\SL\theta]).\label{eq:L1theta2}
	\end{align}
\end{subequations}
Notice that, under the assumption $\theta\in C_\per^{2,\alpha}(\overline{\Omega_h})$, $\LL^2(\lambda)\theta$ is the unique solution of
\begin{align}\label{eq:L2theta_rewritten}
	\Delta[\LL^2(\lambda)\theta]=\Delta\theta+\gamma'(\psi^\lambda)\theta\text{ in }\Omega_h,\quad\LL^2(\lambda)\theta=0\text{ on }y=0\text{ and }y=-h,
\end{align}
and $\LL^1(\lambda)\theta$ is (in the set of periodic functions with zero average) uniquely determined by
\begin{align}\label{eq:L1theta_rewritten}
	-\lambda^2\partial_x\Ch(\LL^1(\lambda)\theta)=\lambda\Pro\SL\partial_y(\theta-\LL^2(\lambda)\theta)+(\gamma(0)-\lambda^{-1}g)\SL\theta.
\end{align}
Furthermore, notice that
\begin{align*}
	&([F^1_{(w,\phi)}(\lambda,0,0,0)]\circ [\T(\lambda)])\theta\\
	&=-\lambda^{-1}\left(\langle\SL\partial_y((1-\lambda^{-1}\psi^\lambda_y)V[\SL\theta])\rangle+\langle\SL\partial_y\A_\phi\theta\rangle\right)\\
	&=-\lambda^{-1}\langle\SL\partial_y\A_\phi\theta\rangle.	
\end{align*}
Therefore, the action of $F_{(q,w,\phi)}(\lambda,0,0,0)$ can be expressed in matrix form as
\[[F_{(q,w,\phi)}(\lambda,0,0,0)]\begin{pmatrix}\text{id}_\R&0\\0&\T(\lambda)\end{pmatrix}\begin{pmatrix}\delta q\\\theta\end{pmatrix}=\begin{pmatrix}\lambda^{-2}\delta q-\lambda^{-1}\langle\SL\partial_y\A_\phi\theta\rangle\\\LL(\lambda)\theta\end{pmatrix}.\]
Hence, having an application of Theorem \ref{thm:CrandallRabinowitz} in mind, we easily see that
\begin{gather}
	(\delta q,\delta w,\delta\phi)\in\ker F_{(q,w,\phi)}(\lambda,0,0,0)\label{eq:char_kernel}\\
	\Leftrightarrow\quad\theta\in\ker\LL(\lambda)\text{ and }\delta q=\lambda\langle\SL\partial_y\A_\phi\theta\rangle,\text{ where }(\delta w,\delta\phi)=\T(\lambda)\theta.\nonumber
\end{gather}
Similarly, from
\begin{align*}
	&[F_{\lambda,(q,w,\phi)}(\lambda,0,0,0)]\begin{pmatrix}\text{id}_\R&0\\0&\T(\lambda)\end{pmatrix}\begin{pmatrix}\delta q\\\theta\end{pmatrix}\\
	&=\begin{pmatrix}-2\lambda^{-3}\delta q+\lambda^{-2}\langle\SL\partial_y\A_\phi\theta\rangle-\lambda^{-1}\langle\SL\partial_y\A_{\phi\lambda}\theta\rangle\\\LL_\lambda(\lambda)\theta\end{pmatrix}\\
	&\phantom{=\;}-[F_{(q,w,\phi)}(\lambda,0,0,0)]\begin{pmatrix}0&0\\0&\T_\lambda(\lambda)\end{pmatrix}\begin{pmatrix}\delta q\\\theta\end{pmatrix}
\end{align*}
we infer that
\begin{gather}
	[F_{\lambda,(q,w,\phi)}(\lambda,0,0,0)](\delta q,\delta w,\delta\phi)\in\im F_{(q,w,\phi)}(\lambda,0,0,0)\label{eq:char_image}\\
	\Leftrightarrow\quad\LL_\lambda(\lambda)\theta\in\im\LL(\lambda),\text{ where }(\delta w,\delta\phi)=\T(\lambda)\theta.\nonumber
\end{gather}
			
\subsubsection{Kernel}
As we have just seen in \eqref{eq:char_kernel}, it suffices to study the kernel of $\LL$; here and in the following, we will suppress the dependency of $\LL$ on $\lambda$. From \eqref{eq:Ltheta}, \eqref{eq:L2theta_rewritten}, and \eqref{eq:L1theta_rewritten} we infer that
\[\LL\theta=0\;\Longleftrightarrow\;\theta\in C_\per^{2,\alpha}(\overline{\Omega_h})\text{ and }\begin{cases}\Delta\theta+\gamma'(\psi^\lambda)\theta=0\quad\text{and}\\\lambda\Pro\SL\partial_y\theta+(\gamma(0)-\lambda^{-1}g)\SL\theta=0;\end{cases}\]
notice that $\LL\theta=0$ implies $\SL\theta\in C^{2,\alpha}_{0,\per}(\R)$ by \eqref{eq:L1theta1} and thus $\theta=\A_\phi\theta+V[\SL\theta]\in C_\per^{2,\alpha}(\overline{\Omega_h})$ by \eqref{eq:L1theta2}. Let us now write $\theta(x,y)=\sum_{k=0}^\infty\theta_k(y)\cos(k\nu x)$ as a Fourier series. Then we easily see that $\LL\theta=0$ if and only if
\begin{align}\label{eq:LLtheta0=0}
	\left(\partial_y^2+\gamma'(\psi^\lambda)\right)\theta_0=0
\end{align}
and, for all $k\ge1$,
\begin{subequations}\label{eq:LLthetak=0}
	\begin{align}
		\left(\partial_y^2+\gamma'(\psi^\lambda)-(k\nu)^2\right)\theta_k&=0,\label{eq:LL1thetak=0}\\
		\lambda\partial_y\theta_k(0)+(\gamma(0)-\lambda^{-1}g)\theta_k(0)&=0;
	\end{align}
\end{subequations}
notice that $\theta_0(0)=0$ is already included in the definition of $Y$. Henceforth, we shall assume \eqref{ass:SL-spectrum}. Thus, we see that \eqref{eq:LLtheta0=0} only has the trivial solution $\theta_0\equiv0$, recalling that $\theta_0(0)=\theta_0(-h)=0$ by membership of $\theta$ in $Y$.
			
Let us now turn to $k\ge 1$ and recall that $\theta_k(-h)=0$ by $\theta\in Y$. First suppose that $-(k\nu)^2$ is in the Dirichlet spectrum of $-\partial_y^2-\gamma'(\psi^\lambda)$ on $[-h,0]$. Then necessarily $\theta_k(0)=0$ provided \eqref{eq:LL1thetak=0} is satisfied. Hence, \eqref{eq:LLthetak=0} can only hold if $\theta_k\equiv0$. On the other hand, suppose that $-(k\nu)^2$ is not in the Dirichlet spectrum of $-\partial_y^2-\gamma'(\psi^\lambda)$ on $[-h,0]$. Then we find that \eqref{eq:LL1thetak=0} is equivalent to the statement that $\theta_k$ is a multiple of $\beta^{-(k\nu)^2,\lambda}$, where $\beta=\beta^{\mu,\lambda}$ is defined by means of \eqref{eq:beta}. 

To sum up, \eqref{eq:LLthetak=0} has a nontrivial solution $\theta_k$ if and only if the dispersion relation
\[d(-(k\nu)^2,\lambda)=0,\]
with $d$ given in \eqref{eq:d(k,lambda)}, is satisfied, and in this case $\theta_k$ is a multiple of $\beta^{-(k\nu)^2,\lambda}$.
\begin{remark}
	The quantity $d(\mu,\lambda)$ is at first defined if and only if $\mu$ is not in the Dirichlet spectrum of $-\partial_y^2-\gamma'(\psi^\lambda)$ on $[-h,0]$. If this property fails to hold, we set $d(\mu,\lambda)\coloneqq\infty$ in the following.
\end{remark}
We summarize our results concerning the kernel:
\begin{lemma}\label{lma:kernel}
	Given $\lambda\neq 0$ and under the assumption \eqref{ass:SL-spectrum}, a function $\theta\in Y$, admitting the Fourier decomposition $\theta(x,y)=\sum_{k=0}^\infty\theta_k(y)\cos(k\nu x)$, is in the kernel of $\LL(\lambda)$ if and only if $\theta_0=0$ and for each $k\ge 1$
	\begin{enumerate}[label=(\alph*)]
		\item $\theta_k=0$, or
		\item $-(k\nu)^2$ is not in the Dirichlet spectrum of $-\partial_y^2-\gamma'(\psi^\lambda)$ on $[-h,0]$, $\theta_k$ is a multiple of $\beta^{-(k\nu)^2,\lambda}$, and the dispersion relation
		\[d(-(k\nu)^2,\lambda)=0\]
		holds, with $d$ given in \eqref{eq:d(k,lambda)}.
	\end{enumerate}
\end{lemma}
			
\subsubsection{Range}\label{appx:range}
Before we proceed with the investigation of the transversality condition, we first prove that the range of $\LL$ can be written as an orthogonal complement with respect to a suitable inner product. This will be helpful for characterizing the transversality condition. To this end, we introduce the inner product
\[\langle(f_1,g_1),(f_2,g_2)\rangle\coloneqq\left\langle\sqrt{\Ch\partial_x}f_1,\sqrt{\Ch\partial_x}f_2\right\rangle_{L^2([0,L])}+\langle\nabla g_1,\nabla g_2\rangle_{L^2(\Omega_h^*)}\]
for $f_1,f_2\in H^{1/2}_{0,\per}(\R)$ (fractional Sobolev space), $g_1,g_2\in H^1_\per(\Omega_h)$, where $\sqrt{\Ch\partial_x}$ is the Fourier multiplier with symbol $\sqrt{k\coth(k\nu h)}$ and where $\Omega_h^*=(0,L)\times(-h,0)$; in order to avoid misunderstanding, we point out that the index \enquote{0} in $H^{1/2}_{0,\per}(\R)$ means \enquote{zero average} as before. This inner product is positive definite on the space
\[Z\coloneqq H^{1/2}_{0,\per}(\R)\times\left\{\tilde g\in H^1_\per(\Omega_h):\tilde g\big|_{y=-h}=0\right\}.\]
Notice that
\[\left\langle\sqrt{\Ch\partial_x}f_1,\sqrt{\Ch\partial_x}f_2\right\rangle_{L^2([0,L])}=\langle f_1,\Ch\partial_x f_2\rangle_{L^2([0,L])}=\langle\Ch\partial_x f_1,f_2\rangle_{L^2([0,L])}\]
if $f_1,f_2\in H^1_{0,\per}(\R)$ and that
\[\langle\nabla g_1,\nabla g_2\rangle_{L^2(\Omega_h^*)}=-\langle g_1,\Delta g_2\rangle_{L^2(\Omega_h^*)}+\langle\SL g_1,\SL\partial_y g_2\rangle_{L^2([0,L])}\]
if $g_2\in H^2_\per(\Omega_h)$ and $g_1=0$ on $y=-h$.
			
Recalling \eqref{eq:Ltheta} we now compute for smooth $\theta,\vartheta\in Y$
\begin{align*}
	&\langle(-\lambda\SL\theta,\theta),\LL\vartheta\rangle\\
	&=\langle\SL\theta,\Pro\SL\partial_y(\A_\phi\vartheta+V[S\vartheta])+\lambda^{-1}(\gamma(0)-\lambda^{-1}g)\SL\vartheta\rangle_{L^2([0,L])}\\
	&\phantom{=\;}-\langle\theta,\Delta(\vartheta-(\A_\phi\vartheta+V[\SL\vartheta]))\rangle_{L^2(\Omega_h^*)}\\
	&\phantom{=\;}+\langle\SL\theta,\SL\partial_y(\vartheta-(\A_\phi\vartheta+V[\SL\vartheta]))\rangle_{L^2([0,L])}\\
	&=\lambda^{-1}(\gamma(0)-\lambda^{-1}g)\langle\SL\theta,\SL\vartheta\rangle_{L^2([0,L])}+\langle\nabla\theta,\nabla\vartheta\rangle_{L^2(\Omega_h^*)}-\langle\theta,\gamma'(\psi^\lambda)\vartheta\rangle_{L^2(\Omega_h^*)}
\end{align*}
making use of $\langle\SL\theta\rangle=0$. Noticing that the terms at the beginning and at the end of this computation only involve at most first derivatives of $\theta$ and $\vartheta$, an easy approximation argument shows that this relation also holds for general $\theta,\vartheta\in Y$. Moreover, since the last expression is symmetric in $\theta$ and $\vartheta$, we can also go in the opposite direction with reversed roles and arrive at the symmetry property
\[\langle(-\lambda\SL\theta,\theta),\LL\vartheta\rangle=\langle\LL\theta,(-\lambda\SL\vartheta,\vartheta)\rangle.\]
Thus, the range of $\LL$ is the orthogonal complement of $\{(-\lambda\SL\theta,\theta):\theta\in\ker\LL\}$ with respect to $\langle\cdot,\cdot\rangle$. Indeed, one inclusion is an immediate consequence of the symmetry property and the other inclusion follows from the facts that we already know that $\LL$ is Fredholm with index zero and that $\LL$ gains no additional kernel when extended to functions $\theta$ with $(\SL\theta,\theta)\in Z$.
			
\subsubsection{Transversality condition}
We can now characterize the transversality condition, assuming that the kernel of $\LL$ is one-dimensional and spanned by $\theta(x,y)=\beta^{-(k\nu)^2,\lambda}(y)\cos(k\nu x)$ for some $k\in\N$. To this end, in view of \eqref{eq:char_image}, we have to investigate whether $\LL_\lambda\theta$ is not in the range of $\LL$, which is equivalent to
\[\langle(-\lambda\SL\theta,\theta),\LL_\lambda\theta\rangle\neq 0\]
by the previous section. Differentiating \eqref{eq:Ltheta} with respect to $\lambda$, for general $\theta$ it holds that
\begin{align*}
	\LL_\lambda^1\theta&=-\lambda^{-1}\LL^1\theta-\lambda^{-1}\partial_x^{-1}(\Ch)^{-1}\left(\Pro\SL\partial_y\A_{\phi\lambda}\theta-(\lambda^{-2}\gamma(0)-2\lambda^{-3}g)\SL\theta\right),\\
	\LL_\lambda^2\theta&=-\A_{\phi\lambda}\theta,
\end{align*}
where $\A_{\phi\lambda}\theta\in C_\per^{2,\alpha}(\overline{\Omega_h})$ is the unique solution of
\begin{align*}
	\Delta(\A_{\phi\lambda}\theta)&=-\gamma''(\psi^\lambda)\partial_\lambda\psi^\lambda\theta&\text{in }\Omega_h,\\
	\A_{\phi\lambda}\theta&=0&\text{on }y=0\text{ and }y=-h.
\end{align*}
Thus, we have
\begin{align*}
	&\langle(-\lambda\SL\theta,\theta),\LL_\lambda\theta\rangle\\
	&=\langle\SL\theta,\Pro\SL\partial_y\A_{\phi\lambda}\theta-(\lambda^{-2}\gamma(0)-2\lambda^{-3}g)\SL\theta\rangle_{L^2([0,L])}\\
	&\phantom{=\;}-\langle\theta,\Delta(-\A_{\phi\lambda}\theta)\rangle_{L^2(\Omega_h^*)}+\langle\SL\theta,-\SL\partial_y\A_{\phi\lambda}\theta\rangle_{L^2([0,L])}\\
	&=-\langle\theta,\gamma''(\psi^\lambda)\partial_\lambda\psi^\lambda\theta\rangle_{L^2(\Omega_h^*)}-(\lambda^{-2}\gamma(0)-2\lambda^{-3}g)\langle\SL\theta,\SL\theta\rangle_{L^2([0,L])}
\end{align*}
whenever $\LL^1\theta=0$. Now let $\theta(x,y)=\beta^{-(k\nu)^2,\lambda}(y)\cos(k\nu x)$ and note that $f=\partial_\lambda\beta^{-(k\nu)^2,\lambda}$ solves
\begin{gather*}
	f_{yy}+(\gamma'(\psi^\lambda)-(k\nu)^2)f=-\gamma''(\psi^\lambda)\beta^{-(k\nu)^2,\lambda}\partial_\lambda\psi^\lambda\text{ on }[-h,0],\\
	f(0)=f(-h)=0.
\end{gather*}
Therefore,
\begin{align*}
	&\frac2L\langle(-\lambda\SL\theta,\theta),\LL_\lambda\theta\rangle\\
	&=\int_{-h}^0\beta^{-(k\nu)^2,\lambda}\left(\partial_\lambda\beta^{-(k\nu)^2,\lambda}_{yy}+(\gamma'(\psi^\lambda)-(k\nu)^2)\partial_\lambda\beta^{-(k\nu)^2,\lambda}\right)\,dy\\
	&\phantom{=\;}-(\lambda^{-2}\gamma(0)-2\lambda^{-3}g)\\
	&=\partial_\lambda\beta_y^{-(k\nu)^2,\lambda}(0)-\lambda^{-2}\gamma(0)+2\lambda^{-3}g
\end{align*}
after integrating by parts. Thus, we have proved:
\begin{lemma}\label{lma:transversality_condition}
	Given $\lambda\neq 0$ and assuming that the kernel of $\LL(\lambda)$ is one-dimensional spanned by $\theta(x,y)=\beta^{-(k\nu)^2,\lambda}(y)\cos(k\nu x)$ for some $k\ge 1$, the transversality condition
	\[\LL_\lambda(\lambda)\theta\notin\im\LL(\lambda)\]
	is equivalent to
	\[d_\lambda(-(k\nu)^2,\lambda)\neq 0,\]
	with $d$ given in \eqref{eq:d(k,lambda)}.
\end{lemma}
			
\subsubsection{Result on local bifurcation}
We are now in the position to apply Theorem \ref{thm:CrandallRabinowitz} in order to prove Theorem \ref{thm:LocalBifurcation}.
\begin{proof}[Proof of Theorem \ref{thm:LocalBifurcation}]
	It is straightforward to apply Theorem \ref{thm:CrandallRabinowitz} in view of Lemmas \ref{lma:M_prop}, \ref{lma:kernel}, and \ref{lma:transversality_condition}, recalling \eqref{eq:char_kernel} and \eqref{eq:char_image}. Moreover, the asymptotic expansion tells us that $w^s\neq 0$ after possibly shrinking $\varepsilon$. Furthermore, after again possibly shrinking $\varepsilon$, we see that also \eqref{eq:additional_requirements} is satisfied along the curve. Notice also that here $\delta q=0$ in \eqref{eq:char_kernel}. The only claim that is not obvious is that $o(s)/s\to 0$ in $\R\times C_{0,\per,\e}^{2,\alpha}(\R)\times C_{\per,\e}^{2,\alpha}(\overline{\Omega_h})$ as $s\to0$ because this at first holds only in $X$. But indeed, by Lemma \ref{lma:M_prop} and $\M_\lambda(\lambda_0,0,0,0)=0$, $[\M_{(w,\phi)}(\lambda_0,0,0,0)](\T(\lambda_0)\theta)=\T(\lambda_0)\theta$,
	\begin{align*}
		\frac{o(s)}{s}&=\frac{(q^s,w^s,\phi^s)-(0,s\T(\lambda_0)\theta)}{s}\\
		&=\frac{\M(\lambda^s,q^s,w^s,\phi^s)-\M(\lambda_0,0,0,0)}{s}\\
		&\phantom{=\;}-\frac{[\M_{(\lambda,q,w,\phi)}(\lambda_0,0,0,0)](\lambda^s-\lambda_0,q^s,w^s,\phi^s)}{s}\\
		&\phantom{=\;}+[\M_{(q,w,\phi)}(\lambda_0,0,0,0)]\frac{o(s)}{s}
	\end{align*}
	also tends to $0$ in $\R\times C_{0,\per,\e}^{2,\alpha}(\R)\times C_{\per,\e}^{2,\alpha}(\overline{\Omega_h})$.
\end{proof}
			
\subsection{Conditions for local bifurcation}\label{appx:LocalBifurcation:cond}
A shortcoming of Theorem \ref{thm:LocalBifurcation} is that the dispersion relation and transversality condition are more or less stated as assumptions. Thus, it makes sense to study these conditions in more detail.

We first remark that the possibility of kernels of dimension larger than one cannot be excluded in general. In some cases, which we will study below, the kernel can be at most or is exactly one-dimensional. However, in general, kernels of arbitrarily large dimension should be expected to appear; for more discussion on this we refer to \cite{AasVar18,EHR09}.
\subsubsection{General conclusions}The term in the dispersion relation which deserves a closer look at is $\beta^{\mu,\lambda}_y(0)$. To this end and recalling \eqref{eq:beta}, the usual strategy is to introduce the Prüfer angle $\vartheta=\vartheta(y,\mu)$ corresponding to the initial value problem
\begin{gather*}
	-u''-\gamma'(\psi^\lambda)u=\mu u\text{ on }[-h,0],\\
	u(-h)=0,\quad u'(-h)=1,
\end{gather*}
as the unique continuous representative of $\mathrm{arg}(u'+iu)$ with $\vartheta(-h,\mu)=\mathrm{arg}(1)=0$. Thus,
\[\beta^{\mu,\lambda}_y(0)=\cot(\vartheta(0,\mu)),\]
and one is, for fixed $\lambda$, left to search for intersections of this quantity with the horizontal line $\mu\mapsto\lambda^{-2}g-\lambda^{-1}\gamma(0)$ when studying the dispersion relation.
In \cite{Varholm20} (using $z=-\mu$ instead of $\mu$ and with the length of the interval scaled to $1$) it was proved that $\beta^{\mu,\lambda}_y(0)$ is strictly monotonically decreasing in $\mu$ on each branch and satisfies the bounds
\begin{align}\label{eq:bounds_beta}
	v(\mu+\sup\gamma')\le\beta^{\mu,\lambda}_y(0)\le v(\mu+\inf\gamma')
\end{align}
with
\[v(z)\coloneqq\sqrt{-z}\coth(h\sqrt{-z})\]
on the (possibly empty) intervals
\[I_j=\begin{cases}(j^2\pi^2/h^2-\inf\gamma',(j+1)^2\pi^2/h^2-\sup\gamma'),&j\in\N,\\(-\infty,\pi^2/h^2-\sup\gamma'),&j=0.\end{cases}\]
Because of strict monotonicity it is clear that, for fixed $\lambda\ne0$, there is exactly one simple solution $\mu$ of the dispersion relation $d(\mu,\lambda)=0$ on each branch of the Prüfer graph.

\subsubsection*{Dispersion relation}
It is also of interest to know when we can expect to find a single negative solution. In the irrotational case, a well known criterion is that the Froude number is less than one, or equivalently, $\lambda^2<gh$. This can also be generalized, whereby a particular nice case is that $\gamma'$ is \enquote*{not too positive}, as we observe in the following.
\begin{proposition}\label{prop:sol_disprel_bound_gamma'}
	If $\sup \gamma'\le \pi^2/h^2$, then there is at most one non-positive solution $\mu$ to the equation $d(\mu, \lambda)=0$.
\end{proposition}
\begin{proof}
	As already seen above, there is exactly one solution on the leftmost branch of $\mu \mapsto \beta_y^{\mu, \lambda}(0)$, and by assumption and \eqref{eq:bounds_beta} this branch has non-negative right endpoint.
\end{proof}

Now, if $\sup \gamma'< \pi^2/h^2$, then the leftmost branch of $\mu \mapsto \beta_y^{\mu, \lambda}(0)$ has positive right endpoint. Moreover, if additionally
\[
\lambda^{-2}g -\lambda^{-1}\gamma(0)> 
\begin{cases} 
	\sqrt{\inf \gamma'} \cot (h\sqrt{\inf \gamma'}) & \text{if } \inf \gamma'>0,\\
	1/h & \text{if } \inf \gamma'=0,\\
	\sqrt{|\inf \gamma'|} \coth (h\sqrt{|\inf \gamma'}|) & \text{if } \inf \gamma'<0,
\end{cases}
\]
then $\lambda^{-2}g-\lambda^{-1}\gamma(0)>\beta_y^{0,\lambda}(0)$. Therefore, there is exactly one negative solution $\mu_0$ to the equation $d(\mu,\lambda)=0$, and in fact Proposition \ref{prop:solvability_disprel} is proved.

\subsubsection*{Transversality condition}
We next try to find a pretty general sufficient condition for the transversality condition. For this, we will make use of the following general lemma.
\begin{lemma}\label{lma:max_principle}
	Let $c\in C([-h,0])$ with $\sup c<\pi^2/(4h^2)$, $u\in C^2((-h,0))\cap C([-h,0])$, and $Lu\coloneqq-u''-cu$. We have:
	\begin{enumerate}[label=(\roman*)]
		\item If $Lu\le0$ on $(-h,0)$, $u(-h)=0$, and $u(0)>0$, then $u'(0)>0$.
		\item If $Lu\le0$ on $(-h,0)$, $u(-h)=0$, and $u(0)\le0$, then $u\le0$ on $[-h,0]$.
		\item If $Lu\ge0$ on $(-h,0)$, $u(0)=0$, and $u'(0)>0$, then $u<0$ on $[-h,0)$.
	\end{enumerate}
\end{lemma}
\begin{proof}
	Let $y_0>0$ be sufficiently small so that $\sup c < \frac{\pi^2}{4(h+y_0)^2}$  and choose  $r>0$ so that
	\[
	\sup c < r^2< \frac{\pi^2}{4(h+y_0)^2}.
	\]
	Then the function $p(y)\coloneqq\cos (r(y-y_0))$ satisfies $p(y)>0$ and $p'(y)>0$ on $[-h,0]$, and
	\[
	-p''-r^2 p=0.
	\]
	Moreover, an easy computation shows that $v\coloneqq u/p$ satisfies
	\[Mv\coloneqq-v''-\frac{2p'}{p}v'+(r^2-c)v=\frac{Lu}{p}\quad\mathrm{on\ }(-h,0).\]
	Now, to this equation we can apply standard maximum principles as $r^2-c\ge0$.
	
	In the situation of (i), we have $Mv\le0$ on $(-h,0)$, $v(-h)=0$, and $v(0)>0$. By the weak maximum principle, it follows that $v(0)=\max v$. Therefore, $v'(0)\ge0$ and consequently $u'(0)=v'(0)p(0)+v(0)p'(0)>0$.
	
	In the situation of (ii), we have $Mv\le0$ on $(-h,0)$, $v(-h)=0$, and $v(0)\le0$. By the weak maximum principle, it follows that $v\le0$ and thus $u\le0$ on $[-h,0]$.
	
	Finally, in the situation of (iii), we have $Mv\ge0$ on $(-h,0)$, $v(0)=0$, and $v'(0)=u'(0)/p(0)>0$. By the strong maximum principle, it follows that $v$ cannot have an interior minimum point on any $(a,0)$, $-h\le a<0$. Hence, $v$ is monotonically increasing on $[-h,0]$ because of $v'(0)>0$. In particular, $v<v(0)=0$ and thus $u<0$ on $[-h,0)$.
\end{proof}
Now let us assume that $\sup \gamma' < \pi^2/(4h^2)$, and moreover $\mu\le 0$. Observe that $\beta=\beta^{\mu,\lambda}$ satisfies
\[
-\beta''-(\gamma'(\psi^\lambda)+\mu)\beta=0, \quad \beta(-h)=0, \quad \beta(0)=1.
\]
By Lemma \ref{lma:max_principle}(i),(ii), it follows that $\beta'(0)>0$ and $\beta\ge0$ on $[-h,0]$.
Moreover, since
\[
-(\partial_\lambda \psi^\lambda)''-\gamma'(\psi^\lambda)\partial_\lambda \psi^\lambda=0,\quad \partial_\lambda \psi^\lambda(0)=0, \quad (\partial_\lambda \psi^\lambda)'(0)=1,
\]
we find that $\partial_\lambda \psi^\lambda<0$ on $[-h, 0)$ applying Lemma \ref{lma:max_principle}(iii).
Finally, note that $\partial_\lambda \beta$ satisfies 
\[
-\partial_\lambda \beta''-(\gamma'(\psi^\lambda)+\mu)\partial_\lambda \beta=\gamma''(\psi^\lambda)\partial_\lambda \psi^\lambda \beta, \quad 
\partial_\lambda \beta(-h)=0, \quad \partial_\lambda \beta(0)=0.
\]
Assuming that $\gamma''\ge 0$, the right-hand side is non-positive and we can apply Lemma \ref{lma:max_principle}(ii) to conclude that $\partial_\lambda \beta\le 0$ on $[-h, 0]$ and therefore $(\partial_\lambda \beta)'(0) \ge 0$. 
Next note that
\[
\gamma(0)=-\lambda \beta'(0)+\lambda^{-1}g
\]
if $d(\mu,\lambda)=0$ holds. Substituting this into the transversality condition, we get
\begin{align}\label{eq:d_lambda_sign}
	d_\lambda(\mu,\lambda)=\partial_\lambda \beta'(0) +2\lambda^{-3}g-\lambda^{-2} \gamma(0)=\partial_\lambda \beta'(0) +\lambda^{-1}\beta'(0)+\lambda^{-3}g>0
\end{align}
if $\lambda>0$. On the other hand, if $\gamma'' \le 0$ and $\lambda<0$, we get by the same argument that $d_\lambda(\mu,\lambda)<0$.
In the special case of affine vorticity, we have $\gamma''\equiv 0$, and so the argument works irrespective of the sign of $\lambda$.

To summarize, we have in fact proved Proposition \ref{prop:transversality_redundant}.

It should be noted that the condition $\sup\gamma'< \pi^2/(4h^2)$ still allows for waves with critical layers. If we, for example, take 
\[
\gamma(\psi)=a\psi+b
\]
with $a<0$ and $b>0$, then
\[
\psi_y^\lambda(y)=\lambda\cosh(\sqrt{-a}y)-\frac{b}{\sqrt{-a}}\sinh(\sqrt{-a}y),
\]
which changes sign if 
\[
-\frac{b}{\sqrt{-a}} \tanh(\sqrt{-a}h)<\lambda<0.
\]
Finally, we remark that the above result can be easily extended to the case with surface tension.

\subsubsection{Examples}We further illustrate our results with certain examples for which the term $\beta^{-(k\nu)^2,\lambda}_y(0)$ can be computed explicitly, namely, constant and affine vorticity. Notice that, as we just saw, we do not have to care about the transversality condition in these cases.
\subsubsection*{Constant vorticity}
First we suppose that $\gamma$ is a constant. Then,
\[\psi^\lambda(y)=-\frac\gamma2 y^2+\lambda y,\quad m(\lambda)=\frac{\gamma h^2}{2}+\lambda h.\]
Clearly, $-(k\nu)^2$ is not in the Dirichlet spectrum of $-\partial_y^2$ on $[-h,0]$ for any $k\ge 0$, and it holds that
\[\beta^{-(k\nu)^2,\lambda}(y)=\frac{\sinh(k\nu(y+h))}{\sinh(k\nu h)}.\]
Thus, we have
\[d(-(k\nu)^2,\lambda)=k\nu\coth(k\nu h)+\lambda^{-1}\gamma-\lambda^{-2}g,\]
and $d(-(k\nu)^2,\lambda)=0$ can equivalently be written as
\begin{align*}
	\lambda&=\frac{-\gamma\pm\sqrt{\gamma^2+4gl\coth(lh)}}{2l\coth(lh)}\\
	&=-\frac{\gamma\tanh(lh)}{2l}\pm\sqrt{\frac gl\tanh(lh)+\frac{\gamma^2}{4l^2}\tanh^2(lh)}\eqqcolon\lambda^\pm(l)
\end{align*}
with $l=k\nu$, a formula which can also be found in \cite{Varholm20} and (for $m=m(\lambda)$ instead of $\lambda$) in \cite{ConsStrVarv16}.
Moreover, it is clear that, for given $\lambda\ne 0$, the dispersion relation can have at most one solution $k\in\N$ since the function $x\mapsto x\coth x$ is injective on $(0,\infty)$ (or in view of Proposition \ref{prop:sol_disprel_bound_gamma'}).

For a further, very detailed investigation of the constant vorticity case we refer to \cite{ConsStrVarv16}.

\subsubsection*{Affine vorticity}
If the vorticity is affine, that is, $\gamma(\psi)=a\psi+b$ with $a\ne 0$, then the trivial solutions are given by
\[\psi^\lambda(y)=\begin{cases}\frac{\lambda}{\sqrt{a}}\sin(\sqrt{a}y)+\frac{b}{a}(\cos(\sqrt{a}y)-1),&a>0,\\\frac{\lambda}{\sqrt{-a}}\sinh(\sqrt{-a}y)+\frac{b}{a}(\cosh(\sqrt{-a}y)-1),&a<0.\end{cases}\]
In general, these trivial solutions can have arbitrarily many critical layers if $a>0$. A short computation shows that $-(k\nu)^2$ is in the Dirichlet spectrum of $-\partial_y^2-a$ on $[-h,0]$ if and only if
\begin{align}\label{eq:affine_vort_Dir_spec}
	\exists m\in\N:a-(k\nu)^2=\frac{\pi^2}{h^2}m^2.
\end{align}
To compute $\beta^{-(k\nu)^2,\lambda}$ and $d(-(k\nu)^2,\lambda)$, we have to distinguish three cases:
\begin{enumerate}[label=\arabic*),leftmargin=*]
	\item $a-(k\nu)^2>0$, which is obviously the only case where \eqref{eq:affine_vort_Dir_spec} can occur: We have
	\[\beta^{-(k\nu)^2,\lambda}(y)=\frac{\sin(\sqrt{a-(k\nu)^2}(y+h))}{\sin(\sqrt{a-(k\nu)^2}h)},\]
	provided \eqref{eq:affine_vort_Dir_spec} fails to hold, and thus
	\[d(-(k\nu)^2,\lambda)=\sqrt{a-(k\nu)^2}\cot(\sqrt{a-(k\nu)^2}h)+\lambda^{-1}b-\lambda^{-2}g.\]
	\item $a-(k\nu)^2=0$: Here we have
	\[\beta^{-(k\nu)^2,\lambda}(y)=\frac{y+h}{h}\]
	and thus
	\[d(-(k\nu)^2,\lambda)=h^{-1}+\lambda^{-1}b-\lambda^{-2}g.\]
	\item $a-(k\nu)^2<0$: Here we have
	\[\beta^{-(k\nu)^2,\lambda}(y)=\frac{\sinh(\sqrt{(k\nu)^2-a}(y+h))}{\sinh(\sqrt{(k\nu)^2-a}h)}\]
	and thus
	\[d(-(k\nu)^2,\lambda)=\sqrt{(k\nu)^2-a}\coth(\sqrt{(k\nu)^2-a}h)+\lambda^{-1}b-\lambda^{-2}g.\]
\end{enumerate}
Moreover, by Proposition \ref{prop:sol_disprel_bound_gamma'}, it is clear that the dispersion relation has at most one solution $k\in\N$ for given $\lambda\ne0$ provided $a\le\pi^2/h^2$. The somewhat richer and more difficult case is $a>\pi^2/h^2$. We refer to \cite{AasVar18,EhrnEschVill12,EhrnEschWahl11,EhrnWahl15} for more details in this situation.

\section{Abstract global bifurcation theorems}\label{appx:GlobalBif}
The global bifurcation theorem by Rabinowitz that we use in Theorem \ref{thm:GlobalBifurcation} is:
\begin{theorem}\label{thm:Rabinowitz}
	Let $X$ be a Banach space, $\mathcal U\subset\R\times X$ open with $(\lambda,0)\in \mathcal U$ for any $\lambda\in\R\setminus\{0\}$, and $F\in C(\mathcal U;X)$. Assume that $F$ admits the form $F(\lambda,x)=x+f(\lambda,x)$ with $f$ compact, and that $F_x(\cdot,0)\in C(\R\setminus\{0\};L(X,X))$. Moreover, suppose that $F(\lambda_0,0)=0$ for some $\lambda_0\in\R\setminus\{0\}$ and that $F_x(\lambda,0)$ has an odd crossing number at $\lambda=\lambda_0$. Let $\mathfrak S$ denote the closure of the set of nontrivial solutions of $F(\lambda,x)=0$ in $\R\times X$ and $\Cc$ denote the connected component of $\mathfrak S$ to which $(\lambda_0,0)$ belongs. Then one of the following alternatives occurs:
	\begin{enumerate}[label=(\roman*)]
		\item $\Cc$ is unbounded;
		\item $\Cc$ contains a point $(\bar\lambda,0)$ with $\bar\lambda\neq\lambda_0$;
		\item $\Cc$ contains a point on the boundary of $\mathcal U$.
	\end{enumerate}
\end{theorem}
The proof of this theorem in the case $\mathcal U=X$ can be found in \cite[Theorem II.3.3]{Kielhoefer} and is practically identical to the proof for general $\mathcal U$.

Moreover, in Theorem \ref{thm:GlobalBifurcation_Analytic} below, we use the following analytic global bifurcation theorem (which includes the local, analytic Crandall--Rabinowitz theorem) in the spirit of \cite{BuffTol03,Dancer73} and presented in the version of \cite{ConsStrVarv16}.
\begin{theorem}\label{thm:analytic_globbi}
	Let $X$, $Y$ be Banach spaces, $\mathcal U\subset\R\times X$ open, and $F\colon \mathcal U\to Y$ a real-analytic function. Suppose that
	\begin{enumerate}[label=($H_\arabic*$)]
		\item $(\lambda,0)\in \mathcal U$ and $F(\lambda,0)=0$ for all $\lambda\in\R\setminus\{0\}$;
		\item for some $\lambda_0\in\R\setminus\{0\}$, $F_x(\lambda_0,0)$ is a Fredholm operator with index zero and one-dimensional kernel spanned by $x_0\in X$, and the transversality condition $F_{\lambda x}(\lambda_0,0)x_0$ $\notin\im F_x(\lambda_0,0)$ holds;
		\item $F_x(\lambda,x)$ is a Fredholm operator of index zero for any $(\lambda,x)\in \mathcal U$ such that $F(\lambda,x)=0$;
		\item for some sequence $(\mathcal Q_j)$ of bounded, closed subsets of $\mathcal U$ with $\mathcal U=\bigcup_{j\in\N}\mathcal Q_j$, the set $\{(\lambda,x)\in \mathcal U:F(\lambda,x)=0\}\cap\mathcal Q_j$ is compact for each $j\in\N$.
	\end{enumerate}
	Then there exists in $\mathcal U$ a continuous curve $\Cc_a=\{(\lambda^s,x^s):s\in\R\}$ of solutions $F(\lambda,x)=0$ such that:
	\begin{enumerate}[label=($C_\arabic*$)]
		\item $(\lambda^0,x^0)=(\lambda_0,0)$;
		\item $x^s=sx_0+o(s)$ in $X$ as $s\to0$;
		\item all solutions of $F(\lambda,x)=0$ in a neighborhood of $(\lambda_0,0)$ are on this curve or are trivial;
		\item $\Cc_a$ has a real-analytic reparametrization locally around each of its points;
		\item one of the following alternatives occur:
		\begin{enumerate}[label=(\roman*)]
			\item for every $j\in\N$, there exists $s_j>0$ such that $(\lambda^s,x^s)\notin\mathcal Q_j$ for all $s\in\R$ with $|s|>s_j$;
			\item there exists $T>0$ such that $(\lambda^{s+T},x^{s+T})=(\lambda^s,x^s)$ for all $s\in\R$.
		\end{enumerate}
	\end{enumerate}
	Moreover, such a curve of solutions to $F(\lambda,x)=0$ having the properties ($C_1$)--($C_5$) is unique (up to reparametrization).
\end{theorem}

\section{Real-analytic vorticity functions}\label{appx:GlobalBifurcation_Analytic}

Stronger results for the solution set $\Cc$ can be derived when more assumptions on the nonlinear operator are made. In particular, if the operator is real-analytic, then one is in the regime of analytic global bifurcation and can deduce, for example, regularity properties of the solution set.
Clearly, our operator $F\colon\Om\to X$ is real-analytic provided $\gamma$ is real-analytic. It is not hard to see that the analytic global bifurcation theorem stated as Theorem \ref{thm:analytic_globbi} above is then applicable in our setting (with $X=Y$ and $\mathcal U=\Om$): While the first two conditions there (($H_1$) and ($H_2$)) are just the assumptions for the local Crandall--Rabinowitz theorem (the second of) which we impose anyway in Theorem \ref{thm:LocalBifurcation}, the other two conditions (($H_3$) and ($H_4$)) are satisfied simply by virtue of our formulation as \enquote{identity plus compact}. Indeed, it is clear that the linearized operator $F_{(q,w,\phi)}$ is Fredholm with index zero -- even at any point, not only at solutions as needed. Moreover, in order to find bounded, closed sets $\mathcal Q_j\subset\Om$, $j\in\N$, that exhaust $\Om$ such that their intersection with the solution set is compact for any $j\in\N$, one can simply define
\[\mathcal Q_j\coloneqq\Om_{1/j}\cap\{(\lambda,q,w,\phi)\in\Om:|\lambda|+\|(q,w,\phi)\|_X\le j\},\]
with $\Om_{1/j}$ according to \eqref{eq:def_Omeps}. Therefore, combining this with the same arguments as used in the proof of Theorem \ref{thm:GlobalBifurcation}, we conclude the following.
\begin{theorem}\label{thm:GlobalBifurcation_Analytic}
	Suppose that $\gamma$ is real-analytic. Moreover, assume that there exists $\lambda_0\neq 0$ such that \eqref{ass:SL-spectrum} holds for $\lambda=\lambda_0$ and such that the dispersion relation
	\[d(-(k\nu)^2,\lambda_0)=0,\]
	with $d$ given by \eqref{eq:d(k,lambda)}, has exactly one solution $k_0\in\N$, and assume that the transversality condition
	\[d_\lambda(-(k_0\nu)^2,\lambda_0)\neq 0\]
	holds. Then there exists in $\Om$ a continuous curve $\Cc_a=\{(\lambda^s,q^s,w^s,\phi^s):s\in\R\}$ of solutions $F(\lambda,q,w,\phi)=0$ such that (for the interpretations we refer back to Theorem \ref{thm:GlobalBifurcation})
	\begin{enumerate}[label=(\greek*)]
		\item $(\lambda^0,q^0,w^0,\phi^0)=(\lambda_0,0,0,0)$;
		\item $(q^s,w^s,\phi^s)=(0,s\T(\lambda_0)\theta)+o(s)$ in $\R\times C_{0,\per,\e}^{2,\alpha}(\R)\times C_{\per,\e}^{2,\alpha}(\overline{\Omega_h})$ as $s\to0$, with $\T(\lambda_0)\theta$ given in \eqref{eq:CR_Ttheta};
		\item all solutions of $F(\lambda,q,w,\phi)=0$ in a neighborhood of $(\lambda_0,0,0,0)$ are on this curve or are trivial;
		\item $\Cc_a$ has a real-analytic reparametrization locally around each of its points;
		\item one of the following alternatives occurs:
		\begin{enumerate}[label=(\roman*)]
			\item for any $\delta\in(5/6,1]$, $p>1$,
			\begin{align*}
				\min\Big\{&\frac{1}{1+|\lambda^s|+\|w^s\|_{C_\per^{0,\delta}(\R)}+\|\gamma((\phi^s+\psi^{\lambda^s})\circ H[w^s+h]^{-1})\|_{L^p(\Omega_{w^s}^*)}},\\
				&q^s+(\lambda^s)^2/2-g\max_\R w^s,\min_\R\K(w^s)\Big\}\to0
				\end{align*}
			as $s\to\pm\infty$;
			\item there exists $T>0$ such that
			\[(\lambda^{s+T},q^{s+T},w^{s+T},\phi^{s+T})=(\lambda^s,q^s,w^s,\phi^s)\]
			for all $s\in\R$.
	\end{enumerate}
	\item all points in $\Cc_a$ satisfy \eqref{eq:additional_requirements} or there exists $s_*\in\R$ such that
	\begin{gather*}
		x\mapsto(x+(\Ch w^{s_*})(x),w^{s_*}(x)+h)\mathrm{\ is\ \textit{not}\ injective\ on\ }\R,\mathrm{\ or}\\
		\min_\R w^{s_*}=-h.
	\end{gather*}
	\end{enumerate}
	Moreover, such a curve of solutions to $F(\lambda,q,w,\phi)=0$ having the properties ($\alpha$)--($\varepsilon$) is unique (up to reparametrization).
\end{theorem}
Let us remark here that while Theorem \ref{thm:GlobalBifurcation} was used as a starting point for the analysis in Sections \ref{sec:nodalanalysis} and \ref{sec:downstream}, it is also possible to begin with Theorem \ref{thm:GlobalBifurcation_Analytic} instead and then obtain similar conclusions in case $\gamma$ is real-analytic.

\section{On degeneracy of the conformal map}\label{appx:degeneracy_conformal}
As already advertised in Remark \ref{rem:GlobalBifurcation}(f), we also want to say more about alternative (iv) in Theorem \ref{thm:GlobalBifurcation}. We first have the following.
\begin{proposition}\label{prop:degeneracy1}
	In Theorem \ref{thm:GlobalBifurcation} the alternative (iv) can be replaced by
	\begin{enumerate}
		\item[(iv$'$)]
		\begin{enumerate}[label=(\alph*)]
			\item $\inf_{(\lambda,q,w,\phi)\in\Cc}\min_\R|\SL\phi_y+\lambda|=0$, or
			\item $\sup_{(\lambda,q,w,\phi)\in\Cc}q=\infty$ (that is, the Bernoulli constant is unbounded from above).
		\end{enumerate}
	\end{enumerate}
	Furthermore, it can also be replaced by
	\begin{enumerate}
		\item[(iv$''$)] $\Cc\cap\Om$ contains a sequence $(\lambda_n,q_n,w_n,\phi_n)$ such that $(w_n)$ converges to some $w$ in the space $C_{0,\per,\e}^{1,\alpha'}(\R)$ for some $\alpha'\in(0,\alpha)$ and the surface $S_w$ determined by $w$ via \eqref{eq:Sw} is \emph{not} of class $C^{1,\beta}$ for any $\beta>0$,\end{enumerate}
	provided alternative (i)(b) is also changed to
	\begin{enumerate}[leftmargin=35pt]
		\item[(i)(b$'$)] $\|w_n\|_{C_\per^{1,\alpha'}(\R)}\to\infty$ for any $\alpha'\in(0,\alpha]$.
	\end{enumerate}
\end{proposition}
\begin{proof}
	Recalling the original Bernoulli equation (see \eqref{eq:OriginalBernoulliFlat}, \eqref{eq:length_of_nablaV_on_top}) we have
	\[\frac{|\SL\phi_y+\lambda|^2}{2\K(w)^2}+gw=q+\frac{\lambda^2}{2}\]
	for $(\lambda,q,w,\phi)\in\Cc\cap\Om$. Thus, if (iv) holds in Theorem \ref{thm:GlobalBifurcation}, that is,
	\[\inf_{(\lambda,q,w,\phi)\in\Cc}\min_\R\K(w)=0,\]
	then (i)(a), (iv$'$)(a), or (iv$'$)(b) occurs or $w$ is unbounded from below, the latter of which, however, is already absorbed in alternative (vi) (or, in other words, is prevented by the bed). This proves the first part of the statement.
	
	Concerning the second part, due to (iv) there exists a sequence of points $(\lambda_n,q_n,w_n,\phi_n)\in\Cc\cap\Om$ with $\min_\R\K(w_n)\to0$ as $n\to\infty$. Now let us assume that $w$ remains bounded along $\Cc$ in $C_\per^{1,\tilde\alpha}(\R)$ for some $\tilde\alpha\in(0,\alpha]$. Then $w_n$ converges, after extracting a suitable subsequence, to some $w$ in $C_\per^{1,\alpha'}(\R)$, where $\alpha'\in(0,\tilde\alpha)$ is arbitrary, but fixed. Thus, clearly
	\begin{align}\label{eq:alternative(iii)}
		1+(\Ch w')(x)=w'(x)=0\text{ for some }x\in\R.
	\end{align}
	(We now essentially follow the proof of the reverse direction of \cite[Theorem 2.2]{ConsVarva11}.) Then $H[w+h]=U[w+h]+iV[w+h]$ is holomorphic, and, since $w\in C_\per^{1,\alpha'}(\R)$, moreover $U[w+h],V[w+h]\in C^{1,\alpha'}(\overline{\Omega_h})$ by \cite[Lemma 2.1]{ConsVarva11}. Additionally assuming that \eqref{eq:additional_requirements} holds for $w$ (otherwise we (iv) or (v) would be valid), an application of the Darboux--Picard Theorem shows that $H[w+h]$ is a conformal mapping from $\Omega_h$ onto $\Omega_w$, which admits an extension as a homeomorphism between the closures of	these domains, with $\R\times\{0\}$ being mapped onto $S_w$ and $\R\times\{-h\}$ being mapped onto $\R\times\{0\}$. Since this conformal mapping is unique up to translations in the variable $x$ (which of course leave regularity properties invariant) by Lemma \ref{lma:ConformalMapping}(ii), the surface cannot be of class $C^{1,\beta}$ for any $\beta>0$ in view of \eqref{eq:alternative(iii)} and Lemma \ref{lma:ConformalMapping}(iv) combined with \eqref{eq:length_of_nablaV_on_top}.
\end{proof}
One might think that the appearance of a somewhat irregular surface as in (iv$''$) is not possible provided the Bernoulli constant $Q$ remains bounded and no surface stagnation occurs (that is, (iii) in Theorem \ref{thm:GlobalBifurcation} does not hold). Indeed, one can derive a pointwise bound of the gradient of the stream function in the physical domain and thus conclude that the limiting stream function is Lipschitz. However, this seems not to imply $C^{1,\beta}$-regularity of the free surface, also in the absence of surface stagnation, as numerical evidence \cite{RRP17}, where a  \enquote{curvy M}-shaped surface is approached, shows. It seems that there are only results available in the literature that need some initial regularity as an assumption to bootstrap from; typical statements in this regard are \enquote{Lipschitz implies $C^{1,\alpha}$} or \enquote{flat implies Lipschitz} \cite{AltCaffarelli81,Caffarelli1,Caffarelli2,DeSilva11}. In this spirit, our alternative (iv$''$) can certainly be weakened slightly, but it is unclear whether it can be completely removed in general.

In fact, we can even somehow get rid of (iv$'$)(a) above. However, for the below argument to work, we need to choose $C_{0,\per,\e}^{3,\alpha}(\R)$ as the space of $w$'s instead of $C_{0,\per,\e}^{1,\alpha}(\R)$ in the definition of $X$ from the beginning, since we will need third order derivatives of $V$ (and thus of $w$) for a certain argument, but we only have a gain of regularity away from the set where the conformal map degenerates and it is not clear why Proposition \ref{prop:3alpha_regularity} should extend there. It is straightforward to see that everything up to now can also be done in such a functional-analytic setting. However, the heavy price we have to pay is that we therefore have to use a much more regular norm for $w$ in the unboundedness alternative (i)(b). In the authors' opinion, it seems much more favorable to have a less regular norm for $w$ in the unboundedness alternative than getting rid of the degeneracy of the conformal map, and in some sense the term $1/\K(w)$ can be thought of a \enquote{part of the norm for $w$} in an unboundedness alternative. Thus, the following statement -- and similarly Proposition \ref{prop:degeneracy1} -- should rather be viewed as a remark and not as a main theorem.

\begin{proposition}
	With the just mentioned realization of the functional-analytic setting and in case \eqref{eq:sup_gamma'_DirEV}, we can replace alternative (iv) in Theorem \ref{thm:GlobalBifurcation} by
	\begin{enumerate}[leftmargin=40pt]
		\item[(iv$'$)(b)] $\sup_{(\lambda,q,w,\phi)\in\Cc}q=\infty$,
	\end{enumerate}
	provided alternative (i)(b) is also changed to
	\begin{enumerate}[leftmargin=35pt]
		\item[(i)(b$''$)] $\|w_n\|_{C_\per^{3,\alpha'}(\R)}\to\infty$ for any $\alpha'\in(0,\alpha]$.
	\end{enumerate}
\end{proposition}
\begin{proof}
	Let us first recall that by means of \eqref{eq:sup_gamma'_DirEV} we are in the situation of Section \ref{sec:nodalanalysis}, and we thus use the nodal properties (and the notation) there.
	
	Suppose that (iv) holds, but all other alternatives in Theorem \ref{thm:GlobalBifurcation} (with (i)(b$''$) instead of (i)(b)) not. Then there exists a sequence $(\lambda_n,q_n,w_n,\phi_n)\in\Cc_\pm$ satisfying $\min_\R\K(w_n)\to0$ as $n\to\infty$. Since $w_n$ is bounded in $C_{0,\per,\e}^{3,\alpha'}(\R)$ for some $\alpha'\in(0,\alpha]$, $(\lambda_n,w_n)$ converges, up to a subsequence, to some $(\lambda,w)$ in $\R\times C_{0,\per,\e}^{3,\beta}(\R)$, where $\beta\in(0,\alpha')$ is arbitrary, but fixed. Therefore,
	\begin{align}\label{eq:degeneracy_x0}
		V_x(x_0,0)=V_y(x_0,0)=0
	\end{align}
	for some $x_0\in[0,L/2]$ by evenness and periodicity, where $V=V[w+h]\in C^{3,\beta}(\overline{\Omega_h})$. Clearly, by Schauder estimates, $(\phi_n)$ is a Cauchy sequence in $C_{\per,\e}^{2,\alpha}(\overline{\Omega_h})$, so it converges to some $\phi$ in this space.
	
	Let us now argue by contradiction and additionally suppose that (iv$'$)(b) does not hold. Then $q_n$ is bounded in $\R$, so we can assume without loss of generality that it converges to some $q\in\R$. (Let us for simplicity again assume that $(\lambda,q,w,\phi)\in\Cc_+$ and $\lambda_0<0$.) Moreover, by Theorem \ref{thm:NodalProperties}, \eqref{eq:MonotoneCrestTroughLimit} is satisfied and it holds that $\SL\phi_y+\lambda\le0$. 
Now we can just follow the the corresponding part of \cite[Proof of Theorem 13]{ConsStrVarv16} and lead \eqref{eq:degeneracy_x0} to a contradiction.
\end{proof}

\subsubsection*{Acknowledgment} We want to thank the anonymous referees for very constructive feedback and helpful comments that improved the manuscript.

\bibliographystyle{amsplain}
\bibliography{bib_2DBifurcation}

\providecommand{\bysame}{\leavevmode\hbox to3em{\hrulefill}\thinspace}
\providecommand{\MR}{\relax\ifhmode\unskip\space\fi MR }
\providecommand{\MRhref}[2]{%
  \href{http://www.ams.org/mathscinet-getitem?mr=#1}{#2}
}
\providecommand{\href}[2]{#2}
\begin{thebibliography}{10}

\bibitem{AasVar18}
A.~Aasen and K.~Varholm, \emph{Traveling gravity water waves with critical
  layers}, J. Math. Fluid Mech. \textbf{20} (2018), no.~1, 161--187.

\bibitem{Alinhac89}
S.~Alinhac, \emph{Existence d'ondes de rar\'{e}faction pour des syst\`emes
  quasi-lin\'{e}aires hyperboliques multidimensionnels}, Comm. Partial
  Differential Equations \textbf{14} (1989), no.~2, 173--230.

\bibitem{AltCaffarelli81}
H.~W. Alt and L.~A. Caffarelli, \emph{Existence and regularity for a minimum
  problem with free boundary}, J. Reine Angew. Math. \textbf{325} (1981),
  105--144.

\bibitem{AmbroseStraussWright16}
D.~M. Ambrose, W.~A. Strauss, and J.~D. Wright, \emph{Global bifurcation theory
  for periodic traveling interfacial gravity-capillary waves}, Ann. Inst. H.
  Poincar\'{e} Anal. Non Lin\'{e}aire \textbf{33} (2016), no.~4, 1081--1101.

\bibitem{Babenko87b}
K.~I. Babenko, \emph{Some remarks on the theory of surface waves of finite
  amplitude}, Dokl. Akad. Nauk SSSR \textbf{294} (1987), no.~5, 1033--1037.

\bibitem{BuffTol03}
B.~Buffoni and J.~Toland, \emph{Analytic theory of global bifurcation},
  Princeton Series in Applied Mathematics, Princeton University Press,
  Princeton, NJ, 2003, An introduction.

\bibitem{Caffarelli1}
L.~A. Caffarelli, \emph{A {H}arnack inequality approach to the regularity of
  free boundaries. {I}. {L}ipschitz free boundaries are {$C^{1,\alpha}$}}, Rev.
  Mat. Iberoamericana \textbf{3} (1987), no.~2, 139--162.

\bibitem{Caffarelli2}
\bysame, \emph{A {H}arnack inequality approach to the regularity of free
  boundaries. {II}. {F}lat free boundaries are {L}ipschitz}, Comm. Pure Appl.
  Math. \textbf{42} (1989), no.~1, 55--78.

\bibitem{ConstantinBook}
A.~Constantin, \emph{Nonlinear water waves with applications to wave-current
  interactions and tsunamis}, CBMS-NSF Regional Conference Series in Applied
  Mathematics, vol.~81, Society for Industrial and Applied Mathematics (SIAM),
  Philadelphia, PA, 2011.

\bibitem{ConstEhrnWahlen07}
A.~Constantin, M.~Ehrnstr\"{o}m, and E.~Wahl\'{e}n, \emph{Symmetry of steady
  periodic gravity water waves with vorticity}, Duke Math. J. \textbf{140}
  (2007), no.~3, 591--603.

\bibitem{ConstantinEscher04}
A.~Constantin and J.~Escher, \emph{Symmetry of steady periodic surface water
  waves with vorticity}, J. Fluid Mech. \textbf{498} (2004), 171--181.

\bibitem{ConsStr04}
A.~Constantin and W.~A. Strauss, \emph{Exact steady periodic water waves with
  vorticity}, Comm. Pure Appl. Math. \textbf{57} (2004), no.~4, 481--527.

\bibitem{ConsStrVarv16}
A.~Constantin, W.~A. Strauss, and E.~V\u{a}rv\u{a}ruc\u{a}, \emph{Global
  bifurcation of steady gravity water waves with critical layers}, Acta Math.
  \textbf{217} (2016), no.~2, 195--262.

\bibitem{ConsStrVarv21}
\bysame, \emph{Large-amplitude steady downstream water waves}, Comm. Math.
  Phys. \textbf{387} (2021), no.~1, 237--266.

\bibitem{ConsVarva11}
A.~Constantin and E.~V\u{a}rv\u{a}ruc\u{a}, \emph{Steady periodic water waves
  with constant vorticity: regularity and local bifurcation}, Arch. Ration.
  Mech. Anal. \textbf{199} (2011), no.~1, 33--67.

\bibitem{Dancer73}
E.~N. Dancer, \emph{Bifurcation theory for analytic operators}, Proc. London
  Math. Soc. (3) \textbf{26} (1973), 359--384.

\bibitem{DeSilva11}
D.~De~Silva, \emph{Free boundary regularity for a problem with right hand
  side}, Interfaces Free Bound. \textbf{13} (2011), no.~2, 223--238.

\bibitem{Dubreil34}
M.-L. Dubreil-Jacotin, \emph{Sur la d{\'e}termination rigoureuse des ondes
  permanentes p{\'e}riodiques d'ampleur finie.}, J. Math. Pures Appl.
  \textbf{13} (1934), 217--291.

\bibitem{DyaHur19b}
S.~A. Dyachenko and V.~M. Hur, \emph{Stokes waves with constant vorticity:
  folds, gaps and fluid bubbles}, J. Fluid Mech. \textbf{878} (2019), 502--521.

\bibitem{DyaHur19a}
\bysame, \emph{Stokes waves with constant vorticity: {I}. {N}umerical
  computation}, Stud. Appl. Math. \textbf{142} (2019), no.~2, 162--189.

\bibitem{EhrnEschVill12}
M.~Ehrnstr\"{o}m, J.~Escher, and G.~Villari, \emph{Steady water waves with
  multiple critical layers: interior dynamics}, J. Math. Fluid Mech.
  \textbf{14} (2012), no.~3, 407--419.

\bibitem{EhrnEschWahl11}
M.~Ehrnstr\"{o}m, J.~Escher, and E.~Wahl\'{e}n, \emph{Steady water waves with
  multiple critical layers}, SIAM J. Math. Anal. \textbf{43} (2011), no.~3,
  1436--1456.

\bibitem{EHR09}
M.~Ehrnstr\"{o}m, H.~Holden, and X.~Raynaud, \emph{Symmetric waves are
  traveling waves}, Int. Math. Res. Not. IMRN (2009), no.~24, 4578--4596.

\bibitem{EhrnVill08}
M.~Ehrnstr\"{o}m and G.~Villari, \emph{Linear water waves with vorticity:
  rotational features and particle paths}, J. Differential Equations
  \textbf{244} (2008), no.~8, 1888--1909.

\bibitem{EhrnWahl15}
M.~Ehrnstr\"{o}m and E.~Wahl\'{e}n, \emph{Trimodal steady water waves}, Arch.
  Ration. Mech. Anal. \textbf{216} (2015), no.~2, 449--471.

\bibitem{EKLM20}
J.~Escher, P.~Knopf, C.~Lienstromberg, and B.-V. Matioc, \emph{Stratified
  periodic water waves with singular density gradients}, Ann. Mat. Pura Appl.
  (4) \textbf{199} (2020), no.~5, 1923--1959.

\bibitem{GilbargTrudinger}
D.~Gilbarg and N.~S. Trudinger, \emph{Elliptic partial differential equations
  of second order}, Classics in Mathematics, Springer-Verlag, Berlin, 2001,
  Reprint of the 1998 edition.

\bibitem{Groves04}
M.~D. Groves, \emph{Steady water waves}, J. Nonlinear Math. Phys. \textbf{11}
  (2004), no.~4, 435--460.

\bibitem{Haziot21}
S.~V. Haziot, \emph{Stratified large-amplitude steady periodic water waves with
  critical layers}, Comm. Math. Phys. \textbf{381} (2021), no.~2, 765--797.

\bibitem{HHSTWWW22}
S.~V. Haziot, V.~M. Hur, W.~A. Strauss, J.~F. Toland, E.~Wahl\'en, S.~Walsh,
  and M.~H. Wheeler, \emph{Traveling water waves --- the ebb and flow of two
  centuries}, Quart. Appl. Math. \textbf{80} (2022), 317--401.

\bibitem{HaziotWheeler21}
S.~V. Haziot and M.~H. Wheeler, \emph{Large-amplitude steady solitary water
  waves with constant vorticity}, Arch. Ration. Mech. Anal. \textbf{247}
  (2023), 27.

\bibitem{HenMat14}
D.~Henry and A.-V. Matioc, \emph{Global bifurcation of
  capillary--gravity-stratified water waves}, Proc. Roy. Soc. Edinburgh Sect. A
  \textbf{144} (2014), no.~4, 775--786.

\bibitem{HenMat13}
D.~Henry and B.-V. Matioc, \emph{On the existence of steady periodic
  capillary-gravity stratified water waves}, Ann. Sc. Norm. Super. Pisa Cl.
  Sci. (5) \textbf{12} (2013), no.~4, 955--974.

\bibitem{HurWheeler20}
V.~M. Hur and M.~H. Wheeler, \emph{Exact free surfaces in constant vorticity
  flows}, J. Fluid Mech. \textbf{896} (2020), R1, 10.

\bibitem{HurWheeler21}
\bysame, \emph{Overhanging and touching waves in constant vorticity flows}, J.
  Differential Equations \textbf{338} (2022), 572--590.

\bibitem{Kielhoefer}
H.~Kielh\"{o}fer, \emph{Bifurcation theory}, second ed., Applied Mathematical
  Sciences, vol. 156, Springer, New York, 2012, An introduction with
  applications to partial differential equations.

\bibitem{Koz23}
V.~Kozlov, \emph{On loops in water wave branches and monotonicity of water
  waves}, J. Math. Fluid Mech. \textbf{25} (2023), no.~9.

\bibitem{KozKuz14}
V.~Kozlov and N.~Kuznetsov, \emph{Dispersion equation for water waves with
  vorticity and {S}tokes waves on flows with counter-currents}, Arch. Ration.
  Mech. Anal. \textbf{214} (2014), no.~3, 971--1018.

\bibitem{KKL14}
V.~Kozlov, N.~Kuznetsov, and E.~Lokharu, \emph{Steady water waves with
  vorticity: an analysis of the dispersion equation}, J. Fluid Mech.
  \textbf{751} (2014), R3, 13.

\bibitem{KozlovLokharu20}
V.~Kozlov and E.~Lokharu, \emph{Global bifurcation and highest waves on water
  of finite depth}, arXiv:2010.14156. To appear in Arch. Ration. Mech. Anal.

\bibitem{Lannes05}
D.~Lannes, \emph{Well-posedness of the water-waves equations}, J. Amer. Math.
  Soc. \textbf{18} (2005), no.~3, 605--654.

\bibitem{Martin13}
C.~I. Martin, \emph{Local bifurcation and regularity for steady periodic
  capillary-gravity water waves with constant vorticity}, Nonlinear Anal. Real
  World Appl. \textbf{14} (2013), no.~1, 131--149.

\bibitem{Matioc14}
B.-V. Matioc, \emph{Global bifurcation for water waves with capillary effects
  and constant vorticity}, Monatsh. Math. \textbf{174} (2014), no.~3, 459--475.

\bibitem{Rabinowitz71}
P.~H. Rabinowitz, \emph{Some global results for nonlinear eigenvalue problems},
  J. Functional Analysis \textbf{7} (1971), 487--513.

\bibitem{RRP17}
R.~Ribeiro, Jr., P.~A. Milewski, and A.~Nachbin, \emph{Flow structure beneath
  rotational water waves with stagnation points}, J. Fluid Mech. \textbf{812}
  (2017), 792--814.

\bibitem{SimmenSaffman85}
J.~A. Simmen and P.~G. Saffman, \emph{Steady deep-water waves on a linear shear
  current}, Stud. Appl. Math. \textbf{73} (1985), no.~1, 35--57.

\bibitem{Spielvogel70}
E.~R. Spielvogel, \emph{A variational principle for waves of infinite depth},
  Arch. Rational Mech. Anal. \textbf{39} (1970), 189--205.

\bibitem{StraussWheeler16}
W.~A. Strauss and M.~H. Wheeler, \emph{Bound on the slope of steady water waves
  with favorable vorticity}, Arch. Ration. Mech. Anal. \textbf{222} (2016),
  no.~3, 1555--1580.

\bibitem{Teles-daSilvaPeregrine88}
A.~F. {Teles da Silva} and D.~H. Peregrine, \emph{Steep, steady surface waves
  on water of finite depth with constant vorticity}, J. Fluid Mech.
  \textbf{195} (1988), 281--302.

\bibitem{Toland96}
J.~F. Toland, \emph{Stokes waves}, Topol. Methods Nonlinear Anal. \textbf{7}
  (1996), no.~1, 1--48.

\bibitem{Vanden-Broeck96}
J.-M. Vanden-Broeck, \emph{Periodic waves with constant vorticity in water of
  infinite depth}, IMA Journal of Applied Mathematics \textbf{56} (1996),
  no.~3, 207--217.

\bibitem{Varholm20}
K.~Varholm, \emph{Global bifurcation of waves with multiple critical layers},
  SIAM J. Math. Anal. \textbf{52} (2020), no.~5, 5066--5089.

\bibitem{Wahlen06b}
E.~Wahl\'{e}n, \emph{Steady periodic capillary-gravity waves with vorticity},
  SIAM J. Math. Anal. \textbf{38} (2006), no.~3, 921--943.

\bibitem{Wahlen06a}
\bysame, \emph{Steady periodic capillary waves with vorticity}, Ark. Mat.
  \textbf{44} (2006), no.~2, 367--387.

\bibitem{Wahlen09}
\bysame, \emph{Steady water waves with a critical layer}, J. Differential
  Equations \textbf{246} (2009), no.~6, 2468--2483.

\bibitem{WahlenWeber21}
E.~Wahlén and J.~Weber, \emph{{Global bifurcation of capillary-gravity water
  waves with overhanging profiles and arbitrary vorticity}}, Int. Math. Res.
  Not. IMRN (2022), rnac280.

\bibitem{Walsh09}
S.~Walsh, \emph{Stratified steady periodic water waves}, SIAM J. Math. Anal.
  \textbf{41} (2009), no.~3, 1054--1105.

\bibitem{Walsh14a}
\bysame, \emph{Steady stratified periodic gravity waves with surface tension
  {I}: {L}ocal bifurcation}, Discrete Contin. Dyn. Syst. \textbf{34} (2014),
  no.~8, 3241--3285.

\bibitem{Walsh14b}
\bysame, \emph{Steady stratified periodic gravity waves with surface tension
  {II}: global bifurcation}, Discrete Contin. Dyn. Syst. \textbf{34} (2014),
  no.~8, 3287--3315.

\end{thebibliography}

\end{document}